\documentclass[a4paper,11pt]{article}
\usepackage[utf8]{inputenc}
\usepackage[hmargin=3cm, vmargin=3.5cm,marginparwidth=2cm]{geometry}
\usepackage{amssymb,amsthm,amsmath,amstext}
\usepackage{mathtools}
\usepackage{graphicx}
\usepackage[shortlabels]{enumitem}
\usepackage[dvipsnames]{xcolor}
\usepackage{marginnote}
\usepackage{stmaryrd}
\usepackage{constants}
\usepackage{verbatim}
\usepackage{pgfplots}
\usepackage{tikz}
\usepackage{hyperref}

\pgfplotsset{compat=1.16}
\usetikzlibrary{decorations, arrows, arrows.meta, hobby, plotmarks, backgrounds}
\hypersetup{
  hidelinks,
  breaklinks=true,
  bookmarksnumbered=true,
  plainpages=false,
  unicode=true,
  pdfauthor={M. Birkner, A. Callegaro, J. Černý},
  pdftitle={Linear spreading speed in non-monotone population models}
}

\numberwithin{equation}{section}
\setlist[enumerate]{label=(\alph*)}
\theoremstyle{plain}
\newtheorem{theorem}{Theorem}[section]
\newtheorem{corollary}[theorem]{Corollary}
\newtheorem{lemma}[theorem]{Lemma}
\newtheorem{proposition}[theorem]{Proposition}

\theoremstyle{definition}
\newtheorem{definition}[theorem]{Definition}
\newtheorem{assumption}{Assumption}
\newtheorem{remark}[theorem]{Remark}
\newtheorem*{claim*}{Claim}
\newtheorem{example}[theorem]{Example}

\newcommand{\E}{\mathbb{E}}
\renewcommand{\L}{\mathbb{L}}
\newcommand{\N}{\mathbb{N}}
\renewcommand{\P}{\mathbb{P}}
\newcommand{\Q}{\mathbb{Q}}
\newcommand{\R}{\mathbb{R}}
\newcommand{\Z}{\mathbb{Z}}
\newcommand{\ind}{\mathbf{1}}
\newcommand{\D}{\mathop{}\!\mathrm{d}}

\newcommand{\Block}{\mathtt{Block}}

\newcommand{\bl}{\mathrm{bl}}

\newcommand{\dens}{\mathrm{dens}}

\newcommand{\ext}{\mathrm{ext}}
\newcommand{\hit}{\mathrm{hit}}

\newcommand{\loc}{{\mathrm{loc}}}

\newcommand{\sspace}{{\mathrm{space}}}
\newcommand{\ttime}{{\mathrm{time}}}
\newcommand{\tv}{\mathrm{TV}}

\DeclareMathOperator{\Supp}{Supp}

\DeclareMathOperator{\argmin}{argmin}
\DeclareMathOperator{\argmax}{argmax}

\DeclarePairedDelimiter{\abs}{\lvert}{\rvert}
\DeclarePairedDelimiter{\ceil}{\lceil}{\rceil}
\DeclarePairedDelimiter{\floor}{\lfloor}{\rfloor}
\DeclarePairedDelimiter{\bbracket}{\llbracket}{\rrbracket}
\DeclarePairedDelimiter{\bbfloor}{\llfloor}{\rrfloor}
\DeclarePairedDelimiter{\bbceil}{\llceil}{\rrceil}

\makeatletter
\newcommand{\subjclass}[2][2020]{%
  \let\@oldtitle\@title%
  \gdef\@title{\@oldtitle\footnotetext{#1
      \emph{Mathematics subject classification:} #2}}%
}
\newcommand{\keywords}[1]{%
  \let\@@oldtitle\@title%
  \gdef\@title{\@@oldtitle\footnotetext{\emph{Key words and phrases:} #1.}}%
}
\makeatother

\title{Linear spreading speed in non-monotone population models}
\author{Matthias Birkner, Alice Callegaro, Jiří Černý }

\begin{document}

\maketitle

\begin{abstract}
  For a broad class of discrete-time, finite-range interacting particle
  systems on $\Z$, we establish a linear spreading speed and a
  one-dimensional shape theorem on the event of survival, without assuming
  monotonicity or attractiveness of the dynamics. The method requires
  that the system admits a coupling with supercritical oriented percolation on
  a coarse-grained lattice. The central technical step is an approximate
  subadditivity property for the hitting times, obtained through
  a `shifted coupling' that compensates for the absence of monotonicity. As a
  concrete application, we show that a discrete-time branching annihilating random walk fits into this framework, and consequently
  exhibits a linear spreading speed.
\end{abstract}

\tableofcontents

\section{Introduction and main results}
\label{sect:Intro+main}

In this paper we develop a method to establish a linear spreading speed and a
one-dimensional shape theorem for a broad class of discrete-time interacting
particle systems on $\Z$, without assuming monotonicity or attractiveness.
The method applies whenever the dynamics satisfies a set of general
assumptions, the most important being the existence, after
coarse-graining, of a coupling with supercritical oriented percolation. At
its core, the proof relies on a classical subadditive argument, but
establishing the required (approximate) subadditivity is far from
straightforward in the absence of monotonicity. The key ingredient is a
new coupling that allows us to compare the growth of two copies of the dynamics
without invoking stochastic ordering. As a concrete application, we show that
the branching annihilating random walk introduced in~\cite{BCC24} fits into
this framework.

The study of shape theorems for random growth models has a long history, going
back at least to the study of the Eden
model for tumour growth~\cite{E61}. The existence of a linear speed was later
proved for the wider class of Richardson models~\cite{Ric73}, including the
Eden model, using subadditivity methods introduced
in~\cite{HammersleyWelsh:1965} for first-passage percolation. Since then,
Kingman's subadditive ergodic theorem~\cite{Kingman:AOP1} has become a
central tool in proving shape results for random growth models.

In Richardson models, extinction cannot occur and the occupied set is
non-decreasing in time. This is the case also in the frog model, for which a
shape theorem was obtained in~\cite{RS04} (continuous time) and
in~\cite{AMPY02} (discrete time), and for the set of visited sites in the
branching random walk in random environment from~\cite{CP07}.

For models in which extinction is possible, the shape result is proved
conditioning on survival. The possibility of extinction poses technical
difficulties because hitting times could be infinite, and conditioning on
survival can break independence. For example, in the contact process
(see~\cite{DG82}, and later~\cite{GM12} for random environment),
this is resolved by considering only \emph{essential} hitting times, that is,
hitting times such that the process started with a single particle at the hit
site does not go extinct. In these cases, a shape result can be obtained with
an approximate version of subadditivity, in which an error term is added to
Kingman's theorem (we refer to the lemma in~\cite{Hammersley:AOP2}, page 674).

Another degree of difficulty arises when the process under consideration is
not monotone. In this case, there is no obvious way to compare the future
evolution of the dynamics with the process started only from a subset of the
current configuration, and thus sub-additivity arguments fail. Results on existence of a shape without using (literal) sub-additivity have been proved for a specific spatial infection model~\cite{KS08}. Recently,~\cite{MP24}
considered a non-monotone branching random walk with competition and proved
that, under small competition effects, the asymptotic shape is close to that
of the model without competition, in terms of Hausdorff distance. The
existence of a (linear) speed, however, remains open for their model.

The methods developed in this paper address precisely this gap, for a large
class of models satisfying a set of five
general assumptions that give us sufficient control over the dynamics. They
are inspired by, but not identical to, the set presented in \cite{BCD16} (see
  Section 3 therein) to understand the dynamics of ancestral lineages in
population models. Besides the existence of the linear spreading speed, these
assumptions imply the existence of a unique non-trivial ergodic invariant
measure.
Later in the paper we will show that they hold
for the model of the branching annihilating random walk that we studied in
\cite{BCC24}.
Our assumptions are rather natural, and we believe that they can be verified
for many models of interest, including a wider class of
probabilistic cellular automata and other population models with different
types of regulation and noise. We postpone a more detailed discussion to
Section~\ref{sec:outlook}.

\medskip

We are now in a position to present our setting, assumptions, and main results.
We consider
a discrete-time process $\eta=(\eta_n)_{n \ge 0}$ with values in $E^{\Z}$
where $(E, \mathcal E)$ is a measurable space. We assume that $E$ contains a
special state~$0$, representing the empty site. (When $\eta_n(x) \neq 0$, we
  sometimes say that the site $(x,n)$ contains a particle). We use
$B_r(x) \coloneq \{ y \in \Z : \abs{x-y} \le r\}$ to denote the ball around
$x\in \mathbb Z$ with radius $r\in \mathbb N_0$ and write
$V_r \coloneq \abs{ B_r(0)} = 2r+1$ for its volume.

Our first assumption requires that the process $\eta$ evolves from a fixed
starting configuration as a stochastic flow on $E^\Z$ whose dynamics is
`local' in time and space. Moreover, the process should have a \emph{finite
  range}, and the empty state should be (locally and globally) a fixed point
of the dynamics.

\begin{assumption}
  \label{ass:flow}
  Let $U = \{ U(y,j): y \in \Z, \, j \in \N \}$ be a family of i.i.d.~random
  variables uniformly distributed in $[0,1]$. There exists a
  measurable function $F: E^{\Z}\times [0,1]^{\Z} \to E^\Z$ so that
  for every $n \in \N_0$, given $\eta_n$,
  \begin{equation}\label{eqn:flow1}
    \eta_{n+1} = F(\eta_n,U(\cdot, n+1)).
  \end{equation}
  $F$ is translation invariant, that is, denoting by $\theta_x$ the
  translation acting naturally on $\mathbb Z$, $E^\Z$ and
  $[0,1]^\Z$,
  \begin{equation*}
    F(\theta_x \eta , \theta_x U) = \theta_x F(\eta ,U)
    \quad\text{for all $x\in \mathbb Z$, $\eta \in E^\Z$, and
      $U\in [0,1]^{\mathbb Z}$.}
  \end{equation*}
  Moreover, there exists $R\in \mathbb N$ so that, for every $x \in \Z$,
  $\eta_{n+1}(x)$ depends only on $U(y,n+1)$ and $\eta_n(y)$ with
  $y\in B_{R}(x)$. Finally, for $x\in \mathbb Z$, if
  $\eta_n(y)=0$ for all $y\in B_{R}(x)$, then
  $\eta_{n+1}(x)=0$.
\end{assumption}

As a consequence of~\eqref{eqn:flow1}, for every $n,m \in \N_0$ with $n>m$,
given $\eta_m=\zeta$, $\eta_n$ can be represented as
\begin{equation}
  \label{eqn:eta_U_flow}
  \eta_{n}=F_{m,n}(\zeta, U),
\end{equation}
where the function
$F_{m,n}(\zeta,\cdot):  [0,1]^{\Z \times \N} \to E^{\Z}$ is measurable
with respect to the $\sigma $-algebra
\begin{equation}
  \label{eqn:calU}
  \mathcal{U}_{m,n}\coloneq \sigma\big(\{U(y,j): y \in \Z, m < j \le n\}\big).
\end{equation}
The locality of the dynamics implies that $\eta_n(x)$ depends on
$\eta_m(y)$ only if $y\in B_{(n-m)R}(x)$.

Our second assumption will allow us to compare $\eta $ with oriented
percolation, using a coarse-graining construction. We start with introducing
some notation. For $L_\bl, T_\bl \in \mathbb N$ we set
\begin{equation}
  \label{eqn:L_bl}
  \L_\sspace \coloneq L_\bl \Z,
  \qquad \L_\ttime \coloneq T_\bl \N_0,
\end{equation}
and define the coarse-grained lattice
\begin{equation}
  \L \coloneq \L_\sspace \times \L_\ttime.
\end{equation}
For $(z,t) \in \mathbb{L}$, we set
\begin{equation}
  \label{eqn:block}
  \Block(z,t)  \coloneq
  \big\{(x,n) \in \Z \times \N_0 :
    x\in B_{K L_\bl}(z), \, t < n \le t+T_\bl \big\},
\end{equation}
where $K$ is an integer constant that can be fixed arbitrarily, but must
satisfy
\begin{equation}
  \label{eqn:K}
  K \ge 4.
\end{equation}
Note that for $z' \in \L_\sspace$ with $\abs{z-z'} \le 2KL_\bl$, the blocks
$\Block(z,t)$ and $\Block(z',t)$ overlap spatially. On the other hand, blocks
at distinct times $t \neq t'$ in $\L_\ttime$ are disjoint, see
Figure~\ref{fig:blocka}.

We will use the following notation to impose certain local properties of
$\eta_n$ and $U$. For any $G^\eta\subset E^{B_{K L_\bl}(0)}$,
$x\in \mathbb Z$ and $\eta \in E^\Z$ we write
\begin{equation}
  \label{eqn:Geta}
  \eta \in G_\loc^\eta(x)
  \quad \text{if} \quad \{\eta(x+y):y\in B_{K L_\bl}(0)\} \in G^\eta
\end{equation}
(that is, $\eta_n$---translated so that $x$ becomes the origin, and
  restricted to $B_{K L_\bl}(0)$---is in~$G^\eta$). Similarly, for any
$G^U\subset [0,1]^{\Block(0,0)}$, $(x,t) \in \mathbb L$, and
$U\in [0,1]^{\Z\times \N}$ we write
\begin{equation}
  \label{eqn:GU}
  U \in G_\loc^U(x,t)
  \quad \text{if} \quad
  \{{U(x+y,t+s):(y,s)\in \Block(0,0)}\} \in G^U.
\end{equation}

\begin{assumption}
  \label{ass:coarse}
  There are $L_\bl, T_\bl \in \mathbb N$ and $K$ as in
  \eqref{eqn:K} such that
  \begin{equation}
    \label{eqn:L}
    L_\bl \ge 2R,
  \end{equation}
  and (`good') sets
  \begin{equation}
    \label{eqn:goodsets}
    G^\eta\subset E^{B_{K L_\bl}(0)}\setminus \{ 0 \}^{B_{K L_\bl}(0)}
    \quad \text{and}
    \quad G^U \subseteq [0,1]^{\Block(0,0)}
  \end{equation}
  with the following properties holding for every $(z,t)\in \L$:
  \begin{enumerate}
    \item (uniform local survival) There is $\varepsilon_2>0$ such that if
    $\eta_n\in G_\loc^\eta (z)$, $n\in \mathbb N_0$, then
    \begin{equation}
      \label{eqn:locgood_to1}
      \P(\eta_{n+1}(x) \neq 0 \mid \eta_n)
      \ge \varepsilon_2 \quad \text{ for } \abs{x-z} \leq L_\bl.
    \end{equation}

    \item (spreading of `goodness')  If $U\in G_\loc^U(z,t)$
    and $\zeta \in G_\loc^\eta(z)$, then
    \begin{equation*}
      \begin{split}
        F_{t,t+T_\bl}(\zeta,U) &\in G_\loc^\eta(z')
        \text{ for every }
        z' \in \{z-L_\bl, z, z+L_\bl\}.
      \end{split}
    \end{equation*}

    \item (production of local agreement) If $U\in G_\loc^U(z,t)$
    and $\zeta, \zeta' \in G_\loc^\eta(z)$, then
    \begin{equation*}
      F_{t,t+T_\bl}(\zeta,U)(x) =F_{t,t+T_\bl}(\zeta',U)(x)
      \quad \text{for all }
      x \in B_{2L_\bl}(z'),  z' \in \{z-L_\bl, z, z+L_\bl\}.
    \end{equation*}

    \item (coupling preservation)
    If $U\in G_\loc^U(z,t)$, $\zeta ,\zeta' \in G_\loc^\eta (z)$ and
    $\zeta(x)=\zeta'(x)$ for all $x \in B_{2 L_\bl}(z)$, then
    \begin{equation*}
      F_{t,t+n}(\zeta,U)(x) = F_{t,t+n}(\zeta',U)(x)
      \quad\text{for all } x \in B_{L_\bl/2}(z)
      \text{ and }  n=1,\dots,T_\bl.
    \end{equation*}
  \end{enumerate}
\end{assumption}

We call the block $\Block(z,t)$ \emph{good} if $U\in G^U_\loc(z,t)$.
Similarly, we say that the block $\Block(z,t)$ (or its base $B_{K L_\bl}(z)$)
is \emph{well-started} if $\eta_t \in G^\eta_\loc(z)$. For an illustration of
Assumption~\ref{ass:coarse}, see Figure~\ref{fig:blocka}.

\begin{figure}[ht]
  \centering
  \usetikzlibrary{decorations.pathreplacing}

\begin{tikzpicture}[xscale=0.45,yscale=0.7]
  %background
  \draw[step=0.1cm,black!5](-13,-.3) grid (13,4.4);
  %good blocks
  \fill [solid,gray!40] (-10,0) rectangle (10,2);
  %agreement region
  \draw[blue] (-1.25,2) rectangle (1.25,0);
  \foreach \x in {0, 0.1, ..., 2} {
    \draw[blue] (-1.25,\x) rectangle (1.25,\x);
  }
  % grid and labels
  \draw[xstep=2.5cm,ystep=2] (-13,-.3) grid (13,4.4);
  \node[below] at (-10,-0.15){\small $z-K L_\bl$};
  \node[below] at (-2.5,-0.15){\small $z-L_\bl$};
  \node[below] at (0,-0.2){\small $z$};
  \node[below] at (2.5,-0.15){\small $z+L_\bl$};
  \node[below] at (10,-0.15){\small $z+K L_\bl$};
  \node[right] at (13,0){\small $t$};
  \node[right] at (13,2){\small $t+T_\bl$};
  % well-started configurations
  \draw[thick,decorate,decoration={brace,amplitude=4pt,mirror,raise=1pt}]
  (-10,0) -- (10,0);
  \draw[thick,decorate,decoration={brace,amplitude=4pt,raise=2pt}]
  (-12.5,2) -- (7.5,2);
  \draw[thick,decorate,decoration={brace,amplitude=4pt,mirror,raise=2pt}]
  (-10,2) -- (10,2);
  \draw[thick,decorate,decoration={brace,amplitude=4pt,raise=2pt}]
  (-7.5,2) -- (12.5,2);
  %agreement
  \foreach \x in {-5, -4.5, -4, ..., 5} {
    \draw plot[mark=x, mark size=4pt, mark options={thick,blue}]
    coordinates {(\x,0)} ;
  }
  \foreach \x in {-7.5, -7, ..., 7.5} {
    \draw plot[mark=x, mark size=4pt, mark options={thick,red}]
    coordinates {(\x,2)} ;
  }
\end{tikzpicture}
  \caption{Illustration for Assumption~\ref{ass:coarse} with $K=4$. The thin
    black lines correspond to the lattice $\L$. The dark grey rectangle is
    $\Block(z,t)$ which is assumed to be good and well started here.
    Well-started blocks are depicted by thick black braces. The fact that
    $\Block(z,t)$ is well started and good implies that there are three
    well-started blocks at level $t+T_\bl$, by part~(b) of the assumption. In
    the region marked with red crosses, the configuration $\eta_{t+T_\bl}$
    depends only on $U$'s in this block, but not on
    $\eta_t \in G^\eta_\loc(z)$, by part~(c). If two configurations in
    $G^\eta_\loc(z)$ agree in the region marked with blue crosses, then they
    agree in the blue rectangle, by part~(d).}
  \label{fig:blocka}
\end{figure}
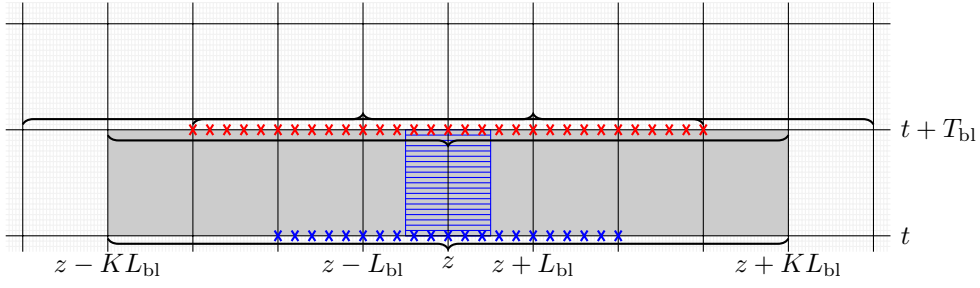

The next assumption implies that good blocks are typical.

\begin{assumption}
  \label{ass:domination}
  Define
  \begin{equation}
    \label{eqn:Yzt}
    Y(z,t) \coloneq \ind_{U\in G^U_\loc(z,t)}, \qquad (z,t) \in \L.
  \end{equation}
  There is a coupling of $(Y(z,t))_{(z,t) \in \mathbb{L}}$ with an
  i.i.d.~Bernoulli field
  $(\tilde Y(z,t))_{(z,t) \in \mathbb{L}}$ satisfying
  \begin{enumerate}
    \item $Y(z,t) \ge \tilde Y(z,t)$ for every $(z,t) \in \L$,
    \item $p \coloneq \mathbb P(\tilde Y(0,0) = 1) > p_c$,
    where $p_c$ is the critical parameter of an oriented site percolation
    process on $\L$ (where $(z,t),(z',t')\in \mathbb L$ are connected by an
      oriented edge iff $t'=t+T_\bl$ and $\abs{z-z'} \le L_\bl$).
  \end{enumerate}
\end{assumption}

\begin{remark}
  \label{rem:LSS}
  By construction, the random variables $Y$ have a finite range of
  dependence. Therefore, a field $\tilde Y$ satisfying
  Assumption~\ref{ass:domination}(a) always exists by
  \cite{liggett1997domination}. For (b) to hold, the probability that a block
  is good, $\P(Y(0,0)=1)$, should be made large enough.
\end{remark}

The fourth assumption states that any non-zero starting configuration
produces, with uniformly positive probability, a well-started block locally
and relatively quickly. Note that this assumption implies a certain
`irreducibility' of the process.

\begin{assumption}
  \label{ass:irreducibility}
  There are $\varepsilon_4 >0$ and $t_4 \in \L_\ttime$ such that, for every
  $\zeta \in E^\Z $  satisfying
  $\{x:\zeta(x)\neq 0\}\cap B_{L_\bl}(0) \neq \emptyset $,
  \begin{equation}
    \mathbb P\big(\eta_{t_4}\in G^\eta_\loc(0) \mid \eta_0 = \zeta  \big)
    \ge \varepsilon_4.
  \end{equation}
\end{assumption}

The final assumption is the most technical one. It requires that any two
configurations that have no particles to the right of the origin and a
particle in $B_{L_\bl}(0)$ have a positive probability of becoming identical
near their right edge, with a well-started block close to that edge.

\begin{assumption}
  \label{ass:coupling}
  There are $\varepsilon_5 >0$ and $t_5 \in \L_\ttime$ such that for every
  $\zeta^1$, $\zeta^2$ such that $\zeta^i(z) =0$ for all $z >0$ and
  $\{x:\zeta^i(x)\neq 0\}\cap B_{L_\bl}(0)\neq \emptyset$, $i=1,2$, there is a
  coupling (denoted $\mathbb P_{\zeta^1,\zeta^2}$) of processes $\eta^1$,
  $\eta^2$ started from $\zeta^1$ and $\zeta^2$ respectively, so that
  \begin{equation}
    \label{eqn:asscoupling}
    \mathbb P_{\zeta^1,\zeta^2}\big(\eta^1_{t_5} \in G_\loc^\eta(0),
      \eta^1_{t_5}(z) = \eta^2_{t_5}(z)
      \ \forall z\ge - K L_\bl\big)\ge \varepsilon_5.
  \end{equation}
\end{assumption}

\begin{remark}
  \label{rmk:Ass5_Efinite}
  Depending on the specific model, Assumption~\ref{ass:coupling} can be easy
  or difficult to check, mainly due to the fact that it requires the
  configurations $\eta^1$, $\eta^2$ to be \emph{exactly equal} over a certain
  space interval. Its verification is straightforward if the state space $E$
  is finite and the model is `irreducible enough'.  In this case the event
  in \eqref{eqn:asscoupling} happens with (tiny, but) positive probability just
  by chance while running the two copies independently. This is the case in
  the example application we consider in Section~\ref{sec:barw}. On the other
  hand, if $E$ is not finite, or even uncountable, proving
  Assumption~\ref{ass:coupling} may require some non-trivial contraction
  arguments; see also the discussion in Section~\ref{sec:outlook}.
\end{remark}

In what follows, we write $\mathbb P_\zeta $ for the distribution of the
process conditioned on $\eta_0 = \zeta \in E^\Z$.
Setting
\begin{equation*}
  \tau_\ext  \coloneq \inf \{ n \in \N : \eta_n\equiv 0\},
\end{equation*}
we use
\begin{equation}
  \label{eqn:barPzeta}
  \bar \P_\zeta (\,\cdot\,)
  \coloneq \P_\zeta (\,\cdot\mid \tau_\ext = \infty)
\end{equation}
to denote the law of the process conditioned on survival. Note that
  Assumptions~\ref{ass:flow}--\ref{ass:irreducibility} imply that
  survival has positive probability for any non-trivial initial condition
  (see Remark~\ref{rem:survival} below). Finally, we set
\begin{equation*}
  \tau_\hit(x)  \coloneq
  \inf \{ n \in \N: \exists z\ge x \text{ such that }\eta_n(z) \neq 0 \}
\end{equation*}
to be the first time there is an occupied site to the right of $x$.
We can now state our main result.

\begin{theorem}
  \label{thm:shape1}
  Under Assumptions \ref{ass:flow}--\ref{ass:coupling}, there exists
  $v\in (0,\infty)$ such that for every $\zeta \in E^\Z$ satisfying
  $\zeta (0)\neq 0$ and $\zeta(y)=0$ for $y > 0$,
  \begin{equation*}
    \lim_{x \to \infty } \frac{x}{\tau_\hit(x)}
    = v, \qquad \bar\P_\zeta \text{-almost surely.}
  \end{equation*}
\end{theorem}

This theorem and its proof imply the following (one-dimensional) shape
theorem for the occupied set. There, in addition, we require that
$F$ is reflection invariant, that is, $F(r\eta ,r U) = r F(\eta ,U)$ for all
$\eta $ and $U$, where $r$ is the reflection around the origin acting
naturally on all relevant spaces.

\begin{theorem}
  \label{thm:real_shape}
  If Assumptions~\ref{ass:flow}--\ref{ass:coupling} hold and $F$ is
  reflection invariant, then there is $c<\infty$ such that for every
  $\varepsilon >0$ and $\zeta$
  which is non-zero only at the origin, there is a
  $\bar \P_\zeta $-a.s.~finite random time $T$ such that,
  $\bar \P_\zeta $-a.s.~for all $n\ge T$
  (with  $v\in (0,\infty)$ as in Theorem~\ref{thm:shape1}):
  \begin{enumerate}
    \item
    $\{x\in \mathbb Z :\eta_n(x) \neq 0\}
    \subset [-(1+\varepsilon) nv,(1+\varepsilon)nv]$,

    \item
    $B_{c\log n}(y) \cap \{x\in \mathbb Z: \eta_n(x) \neq 0\} \neq \emptyset$
    for every $y\in  [-(1-\varepsilon) nv,(1-\varepsilon)nv]\cap \mathbb Z$.
  \end{enumerate}
\end{theorem}

\begin{remark}
  We could alternatively define $\tau_\hit(x)$ to be the first time at which
  there is a particle exactly at $x$, or in $B_R(x)$. Our choice slightly
  simplifies the proof of Theorem~\ref{thm:perturbed_subadditivity}. On the
  other hand, our techniques (see Lemma~\ref{lem:incoupled} and the proof of
    Lemma~\ref{lem:hittingtime} below) imply that the difference between the
  various possible definitions of $\tau_\hit(x)$ has (at most) stretched exponential
  tails. As a consequence, Theorem~\ref{thm:shape1} holds for all these
  definitions.
\end{remark}

The main tool in proving Theorems~\ref{thm:shape1} and~\ref{thm:real_shape} is
the following approximate subadditivity property for the hitting times.

\begin{theorem}
  \label{thm:perturbed_subadditivity}
  Under Assumptions~\ref{ass:flow}--\ref{ass:coupling}, for every starting
  condition $\zeta $ that is non-zero only at the origin and for every
  $x,y\in \mathbb N$ there is a probability space
  $(\Omega , \mathcal A, \mathbb Q)$ with random variables
  $\bar \tau_\hit^\zeta(x)$, $\bar \tau_\hit^\zeta(x+y)$,
  $\bar \sigma^\zeta_\hit(y)$, and $V(x,y)$ such that
  \begin{enumerate}
    \item $\bar \tau_\hit^\zeta(x)$ has the same distribution as
      $\tau_\hit(x)$ under $\bar\P_\zeta $,
      \\ $\bar \tau_\hit^\zeta(x+y)$ has the same distribution as
      $\tau_\hit(x+y)$ under $\bar\P_\zeta $,
      \\ $\bar \sigma_\hit^{\zeta}(y)$ has the same
      distribution as $\tau_\hit(y)$ under $\bar \P_\zeta$, and is
      independent of $\bar \tau_\hit^{\zeta}(x)$,

    \item $\Q$-a.s.,
    $\bar \tau_\hit^{\zeta }(x+y)
    \le \bar \tau_\hit^{\zeta }(x)+ \bar \sigma_\hit^{\zeta }(y)+V(x,y)$,

    \item there is $C<\infty$ (independent of $x$ and $y$) such that
    $\E^\Q[V(x,y)^2] \le C(x+y)$.
  \end{enumerate}
\end{theorem}

Theorem~\ref{thm:shape1} will follow directly from
Theorem~\ref{thm:perturbed_subadditivity} using the classical arguments from
\cite{Hammersley:AOP2, BG80-II}, with some additional steps needed to cover
all initial configurations considered in Theorem~\ref{thm:shape1}.

Theorem~\ref{thm:perturbed_subadditivity} is the central
innovation of this paper. It shows that subadditive arguments,
which are usually rather straightforward in the context of monotone systems,
can also be used for non-monotone models satisfying Assumptions
\ref{ass:flow}--\ref{ass:coupling}.

\paragraph{Organisation of the paper.}

The rest of the paper is organised in two parts. The first one, giving the
proofs of Theorems~\ref{thm:shape1}--\ref{thm:perturbed_subadditivity},
consists of five sections. Section~\ref{sec:percol} reviews some standard
results on oriented percolation, and then explains how these can be used to
control the process $\eta$.  In particular, it gives a first (non-optimal)
lower bound on the speed of spreading of the occupied region.
Section~\ref{sec:liverpool} establishes an intermediate result leading to
Theorem~\ref{thm:perturbed_subadditivity}. It shows that at time
$\tau_\hit(x)$ there are, with high probability, many well-started blocks
close to $x$ (see Proposition~\ref{prop:hitandblock}).  This is a key
property that allows us to restart at time $\tau_\hit(x)$ with a well-behaved
configuration, which then leads to the proof of
Theorem~\ref{thm:perturbed_subadditivity} in Section~\ref{sec:mainproof}.
Its central tool, and the new
technical ingredient of the paper, is a \emph{shifted coupling} (introduced
  in Section~\ref{sec:couplefronts}) that compares two copies of the dynamics
started from well-behaved configurations without invoking stochastic
ordering. Relatively short proofs of Theorems~\ref{thm:shape1} and
\ref{thm:real_shape} are presented in Sections~\ref{sec:proof_speed} and
\ref{sec:proof_shape}.

In the second part of the paper, we first show, in Section~\ref{sec:barw},
how to apply the techniques developed in the first part to a specific model,
namely the branching annihilating random walk whose study we initiated in
\cite{BCC24}. In the closing Section~\ref{sec:outlook} we then comment on
other models which (potentially) satisfy our assumptions and discuss future
directions.

\paragraph{Additional notation.}
We frequently use the following notation related to the coarse-grained
lattice $\L$. For $r\ge 0$ and $z\in \L_\sspace$ we write
\begin{equation}
  B_r^{\bl}(z)  \coloneq
  \big\{ z' \in \mathbb{L}_\sspace : \abs{z-z'} \le r L_\bl \big\}.
\end{equation}
For $z\in \mathbb Z$,  we use $\bbracket z$ to denote the point
in $\L_\sspace$ closest
to $z$, that is,
\begin{equation}
   \bbracket{z} = y \quad \text{if } y \in \L_\sspace \text{ and }
   -L_\bl /2 < z-y \le L_\bl /2.
\end{equation}
Similarly, we use $\bbfloor z  = \floor{z/L_\bl}L_\bl$ and
$\bbceil z = \ceil{z/L_\bl}L_\bl$ to denote the closest points on
$\L_\sspace$ to the left and right of $z$, respectively. With some abuse of
notation, for temporal variables $t\in \N_0$, we write
$\bbfloor t =\floor {t/T_\bl}  T_\bl$ and $\bbceil t =\ceil{ t/T_\bl} T_\bl$.

Throughout the paper we use $c, C, \tilde c, \dots$ to denote finite positive
constants whose value may change from line to line. Constants whose values are
fixed throughout the paper are numbered.

\section{Spreading of well-started configurations}
\label{sec:percol}

The goal of this section is to establish some relatively standard, but rather
explicit control of our process in terms of oriented percolation. Its main
result, Lemma~\ref{lem:spread_goodconf} below, roughly states that if
$\eta_0$ contains a sufficiently large number of well-started blocks then,
with high probability, at later times there will be a large number of
well-started blocks in specified regions of space. Another important result
is Lemma~\ref{lem:manygoodblocks} ensuring that the probability of ever
reaching a configuration with many well-started blocks is bounded from below,
uniformly over all non-trivial initial conditions.

\subsection{Some results on oriented percolation}
\label{sec:comparison_perc}

We start by recalling some well-known results on oriented percolation,
restated in the notation used in this paper.
Recall from Assumption~\ref{ass:domination} that
$(\tilde Y(z,t))_{(z,t) \in \L}$ are i.i.d.~Bernoulli random variables with
success probability $p>p_c$, where $p_c$ is the critical parameter of the
oriented percolation.

A path from $(z,t) \in\L$ to $(z',t') \in \L$ (with $t'\ge t$) is a sequence
\begin{equation*}
  \big((z_i,t_i)\in \L: i = 0,\dots, m \big),
\end{equation*}
so that $m=(t'-t)/T_\bl$, $(z,t)=(z_0,t_0)$, $(z',t') = (z_m,t_m)$, and
$t_i-t_{i-1} = T_\bl$, $\abs{z_i-z_{i-1}}\le L_\bl$ for every $1\le i \le m$.
Such a path is called open in $\tilde Y$ if $\tilde Y(z_i,t_i) =1 $ for all
$0\le i < m$. Note that we do not require
$\tilde Y(z_m,t_m) = \tilde Y(z',t')=1$ at the final point of the path. The
reason for this is that the values of $U$ in $\Block(z',t')$ do not
influence the values of $\eta_{t'}$, by Assumption~\ref{ass:flow}.

Later, by Assumption~\ref{ass:coarse}, open paths will be used to propagate
well-started configurations. Therefore, for $(z,t)\in \L$ with $t \geq T_\bl$
and $A\subset \L_\sspace$ we define
\begin{equation}
\label{eq:clusterindicator}
  \tilde Z^A_t(z)
  \coloneq \ind\big\{\exists z'\in A
    \text{ such that there is an open path in $\tilde Y$
      from $(z',0)$ to $(z,t)$}
    \big\}.
\end{equation}
Further, we define the hitting time of $z \in \L_\sspace$ and the extinction
time related to this percolation process,
\begin{equation}
  \label{eqn:hit_times_percol}
  \begin{split}
    \tilde \tau_\hit^A(z)
    &\coloneq \inf \big\{t \in \mathbb{L}_\ttime : \tilde Z^A_t(z) = 1 \big\},
    \\ \tilde \tau_\ext^A &\coloneq
    \inf \big\{t \in T_\bl \mathbb N : \tilde Z_t^A(\cdot)\equiv 0\big\}.
  \end{split}
\end{equation}
If $A = \{z\}$ we write only $\tilde Z^{z}_t(\cdot)$,
$\tilde \tau^{z}_\hit(\cdot)$, etc.

It is well known that supercritical oriented Bernoulli site percolation has a
non-trivial upper invariant measure~\cite{Durrett1984} (also note that our
  underlying directed graph is literally the `alternative model' in the
  notation of \cite{GrimmettHiemer2002}). We denote this measure by $\nu$.
Taking the initial condition $A$ to be random and $\nu $-distributed, we define
the stationary process $(\tilde Z^\nu_t(z): (z,t)\in \L)$ (note that this
  process uses the same $\tilde Y$'s as $\tilde Z^A$). Finally, we set
\begin{equation}
    \label{eq:p_nu}
    p_\nu = \P (\tilde Z^\nu_0(0) =1)
\end{equation}
to be the density of $\nu $.

Given $z\in \L_\sspace$ and $A\subset \L_\sspace$, we define the coupled regions
\begin{equation}
  \label{eqn:Kdef}
  \tilde K_t^z \coloneq
  \big\{ z' \in \mathbb{L}_\sspace : \, \tilde Z_s^z(z')= \tilde Z_s^\nu(z')
    \text{ for all } s\in \L_\ttime, s\ge t\big\},
  \qquad \tilde K_t^A \coloneq \bigcup_{z\in A}\tilde K_t^z.
\end{equation}

We summarise known results about oriented percolation in the next lemma.
For proofs, see for example~\cite{GM12},~\cite{GM14},~\cite{DG82},
\cite{Durrett1984}, \cite[Thm.~2.30]{Liggett1999book},
\cite[Lemma~A.1]{BirknerCernyDepperschmidtGantert2013}. Part
\eqref{perc:manyzeros} for all $p>p_c$ was recently proved in Lemma~B.2
of~\cite{BBDS26}.

\begin{lemma}
  \label{lem:percolation}
  There exist constants  $\tilde C < \infty$, $\tilde \gamma>0$
  and $\tilde v > 0$ such that
  \begin{enumerate}[ref=\alph*]
    \item \label{perc:extinction}
    $\P(\tilde \tau_\ext^A < \infty)
    \le (1-p_\nu)\wedge\tilde Ce^{-\tilde\gamma\abs A}$
    for every $A \subseteq \mathbb{L}_\sspace$, $A\neq \emptyset$;

    \item \label{perc:die_late}
    $ \P(t< \tilde \tau_\ext^0 <\infty) \le \tilde Ce^{-\tilde \gamma t}$ for
    every $t \in \mathbb{L}_\ttime$;

    \item \label{perc:hit_late}
    $\P(\tilde \tau_\hit^0(z)>t, \tilde \tau_\ext^0=\infty) \le \tilde
    Ce^{-\tilde \gamma t}$ for every $(z,t) \in \L$ with
    $z \in B_{\tilde v t/T_\bl}^\bl(0)$;

    \item\label{perc:couple_late}
    there is a coupling of $\tilde Z^\nu $ and $\tilde Z^0$ so that
    $\P(z \notin \tilde K_t^0, \tilde \tau_\ext^0=\infty) \le \tilde
    Ce^{-\tilde \gamma t}$
    for every $(z,t) \in \L$ with $z \in B_{\tilde v t/T_\bl}^\bl(0)$;

    \item\label{perc:large_dev}
    letting
    $\tilde S_r^\nu(z,t) \coloneq \sum_{z' \in B_r^\bl(z)} \tilde Z_t^\nu(z')$,
    there exists $r_0>0$ such that
    \begin{equation*}
      \P \big( \tilde S_r^\nu(z,t) < p_\nu  V_r/2 \big)
      \le \tilde Ce^{-\tilde \gamma r} \quad
      \text{for every } r \ge r_0 \text{ and } (z,t) \in \L;
    \end{equation*}

    \item\label{perc:manyzeros}
    for every $V=\{(z_i,t_i):i=1,\dots,k\}\subset \mathbb L$ with
    $0\le t_1<\dots < t_k$,
    \begin{equation*}
      \mathbb P(\tilde Z_{t}^\nu(z) = 0 \text{ for all } (z,t)\in V)
      \le (1-p_\nu)\wedge \tilde C e^{-\tilde \gamma k}.
    \end{equation*}
  \end{enumerate}
\end{lemma}

\begin{remark}
  \label{rem:lowerboundspeed}
  The constant $\tilde v$ is a lower bound on the speed of spreading of the
  oriented percolation (on the block scale). Converted to the space-time scale
  of the process $\eta $, it will provide a lower bound on the speed $v$ of
  spreading of the population $\eta $ from Theorems~\ref{thm:shape1}
  and~\ref{thm:real_shape},
  \begin{equation}
    \label{eq:vlowerbound}
    v \ge  \tilde v L_\bl/T_\bl.
  \end{equation}
\end{remark}

The next technical result extends
Lemma~\ref{lem:percolation}\eqref{perc:couple_late} and controls the
coupled region $\tilde K^A$ when $\abs A$ is large, showing that it then
spreads at least linearly, with high probability.

\begin{lemma}
  \label{lem:shift_blocks}
  Let $\ell, \rho \in \N$ and $\emptyset \neq A\subseteq B_{\ell}^\bl(0)$. If
  $(z,t)\in \L$ with $t\ge T_\bl (\abs{z}L_\bl^{-1}+\rho+\ell)/\tilde v$, then
  \begin{equation}
    \label{eq:shift_blocks_v1}
    \P( B_{\rho}^\bl(z) \nsubseteq \tilde K_t^A)
    \le \tilde C e^{-\tilde \gamma \abs{A}}
    + \tilde C V_\rho \abs{A} e^{-\tilde\gamma t}
    \le \tilde C' (1+V_\rho) e^{-\tilde \gamma' \abs{A}},
  \end{equation}
  where $\tilde C, \tilde \gamma$ and $\tilde v$ are as in
  Lemma~\ref{lem:percolation}, and $\tilde C', \tilde \gamma '\in (0,\infty)$.
\end{lemma}

\begin{proof}
  Using the definition of $\tilde K^A_t$ in~\eqref{eqn:Kdef}, observing that
  $\{\tilde{\tau}_\ext^x<\infty\} \subset \{ \tilde{K}^x_t = \emptyset \}$, we
  obtain
  \begin{align*}
    \P\big( B_{\rho}^\bl(z) \nsubseteq \tilde K_t^A\big)
    & \le \P \big(  \tilde \tau_\ext^A <\infty \big)
    + \P \Big( \exists \, x \in A, y \in B_{\rho}^\bl(z) :
      \, \tilde \tau_\ext^x = \infty, y \not\in \tilde K_t^x \Big)
    \\ & \le \P \big(  \tilde \tau_\ext^A <\infty \big)
    + \sum_{x \in A} \sum_{y \in B_{\rho}^\bl(z)}
      \P \big( \tilde \tau_\ext^x = \infty, y \not\in \tilde K_t^x \big).
  \end{align*}
  If $x \in A \subset B_{\ell}^\bl(0)$ and $y \in B_{\rho}^\bl(z)$, then
  $t \ge T_\bl(\abs{z}/L_\bl+\rho+\ell)/\tilde v$ implies
  $y \in B_{\tilde v t/T_\bl}^\bl(x)$. The first inequality in
  \eqref{eq:shift_blocks_v1} then follows by
  Lemma~\ref{lem:percolation}(\ref{perc:extinction},\ref{perc:couple_late})
  together with the translation invariance of the model. The second
  inequality in \eqref{eq:shift_blocks_v1} follows by observing that
  ${\abs A \le 2\tilde{v} t/T_\bl + 1}$ by assumption. %
\end{proof}

\subsection{Spreading lemma}

We now transfer the results from the previous section to the context of
the process $\eta$. For this, Assumptions~\ref{ass:coarse} and
\ref{ass:domination} are instrumental.

Recall the sets $G^\eta $ and $G_\loc^\eta(x)$ from
Assumption~\ref{ass:coarse} and the process $Y$ from
Assumption~\ref{ass:domination}. For $\eta \in E^\Z$, let
\begin{equation}
  \label{eqn:Weta}
  W(\eta) \coloneq \{ z \in \L_\sspace : \eta \in G_\loc^\eta(z) \}
\end{equation}
be the set of block centres where $\eta$ is well started. By
Assumption~\ref{ass:coarse}(b), if $\eta_0 \in G_\loc^\eta (z)$ and there is
an open path in $Y$ from $(z,0)$ to $(z',t)$, then
$\eta_{t} \in G_\loc^\eta (z')$. Moreover, by
Assumption~\ref{ass:domination}, $Y(z,t)\ge \tilde Y(z,t)$, so the existence
of an open path in $\tilde Y$ implies that the same path is open in $Y$. From
this reasoning it follows that, for $(z,t)\in \L$,
\begin{equation}
  \label{eqn:fromZtoeta}
  \tilde Z^{W(\eta_0)}_t(z) = 1 \qquad \implies \qquad z\in W(\eta_t).
\end{equation}

The next lemma controls the size of the intersections of $W(\eta_n)$ with
sufficiently large balls, given that the initial condition contains
sufficiently many well-started blocks. (Note that this lemma is formulated on
  the space and time scale of the process $\eta $, and not of the blocks.)

\begin{lemma}
  \label{lem:spread_goodconf}
  Let $\tilde C, \tilde \gamma, \tilde v, \tilde C', \tilde \gamma'$ be as in
  Lemmas~\ref{lem:percolation},~\ref{lem:shift_blocks}. For every
  $\ell, \rho  \in \N$, $a \in \{1,\dots,V_\ell\}$, $\eta_0 \in E^\Z$ such
  that $\abs{ B_\ell^\bl(0) \cap W(\eta_0)}  \ge a$, $z\in \mathbb Z$, and
  $n \in \L_\ttime$ such that $n \ge (\rho+\abs{z}/L_\bl+\ell+1)T_\bl/\tilde v$
  \begin{equation*}
    \P_{\eta_0} \Big(
      \abs[\big]{B_\rho^\bl(\bbracket z ) \cap W(\eta_n)}
      < \frac{p_\nu V_{\rho}}{2} \Big)
    \le \tilde C e^{-\tilde \gamma \rho}
    + \tilde C' (1+V_\rho) e^{ -\tilde \gamma' a},
  \end{equation*}
  and
  \begin{equation*}
    \P_{\eta_0}(\tau_\ext< \infty) \le  \tilde C e^{-\tilde \gamma a}.
  \end{equation*}
\end{lemma}

\begin{proof}
  By Lemma~\ref{lem:shift_blocks}, for $\eta_0$, $a$, $z$, and $n$ as in the
  statement,
  \begin{equation*}
    \P_{\eta_0}(B^\bl_\rho(\bbracket z) \nsubseteq \tilde K_n^{W(\eta_0)})
    \le  \tilde C' (1+V_\rho )
    e^{-\tilde \gamma ' \abs{W(\eta_0) \cap B_\ell^\bl(0)}}
    \le  \tilde C' (1+V_\rho )
    e^{-\tilde \gamma ' a}.
  \end{equation*}
  In addition, by definition of $\tilde K_n^A$ and
  Lemma~\ref{lem:percolation}\eqref{perc:large_dev},
  \begin{equation*}
    \begin{split}
      &\P\bigg(\Big\{\abs[\big]{\{z'\in B_\rho^\bl(\bbracket z):
              \tilde Z^{W(\eta_0)}_n(z') =1\}} < \frac{p_\nu V_\rho }2\Big\}
        \cap \{B^\bl_\rho({\bbracket z}) \subseteq \tilde K_n^{W(\eta_0)}\}
        \bigg)
      \\&\qquad \le\P\Big(
        \tilde S^\nu_\rho (\bbracket z, n) < \frac{p_\nu  V_\rho }{2}\Big)
      \le \tilde Ce^{-\tilde \gamma \rho }.
    \end{split}
  \end{equation*}
  Together with \eqref{eqn:fromZtoeta}, this implies the first claim of the
  lemma. The second one directly follows from
  Lemma~\ref{lem:percolation}\eqref{perc:extinction} and \eqref{eqn:fromZtoeta}.
\end{proof}

\subsection{Reaching many well-started blocks}

In order to apply Lemma~\ref{lem:spread_goodconf}, we must show that, with
positive probability, a sufficient number of well-started blocks can be
produced from any non-trivial initial configuration. Assumption
\ref{ass:irreducibility} ensures this.

\begin{lemma}
  \label{lem:manygoodblocks}
  There exist $\varepsilon'>0$ and $c_0, r_1 \in \N$ such that for every
  $\ell, r\in \N$ with $r\ge r_1$, letting $n= r c_0 T_\bl \in \L_\ttime$,
  \begin{equation}
    \label{manygoodblocks}
    \P_{\zeta }\Big( \exists z \in B_{\ell}(0) \cap \mathbb{L}_\sspace :
      \abs[\big]{ B^\bl_{r}(z) \cap W(\eta_n) }
      \ge \frac{1}{2} p_\nu V_r \Big)
    \ge \varepsilon'
  \end{equation}
  uniformly over all $\zeta  \in E^{\Z}$ such that
  $\{ x \in B_\ell(0) : \zeta (x) \neq0 \} \neq \emptyset$.
\end{lemma}

\begin{proof}
  Assumption~\ref{ass:irreducibility} and the translation invariance of the
  model imply that there are $t_{\ref{ass:irreducibility}}\in \L_\ttime$ and
  $\varepsilon_{\ref{ass:irreducibility}} >0$ such that, for any $\zeta $ as
  in the statement of the lemma,
  \begin{equation*}
    \P_\zeta(W(\eta_{t_{\ref{ass:irreducibility}}}) \cap B_{\ell}(0)
      \neq \emptyset)
    \ge \varepsilon_{\ref{ass:irreducibility}}.
  \end{equation*}
  If the event above occurs, let $y\in B_\ell(0)\cap \L_\sspace$ be an
  element of $W(\eta_{t_{\ref{ass:irreducibility}}})\cap B_\ell(0)$ chosen by
  some arbitrary rule. Using the Markov property, after time
  $t_{\ref{ass:irreducibility}}$ we can control $\eta$ using a comparison
  with the oriented percolation started from the site $y$. By Lemma
  \ref{lem:percolation}%
  (\ref{perc:extinction},\,\ref{perc:couple_late},\,\ref{perc:large_dev}),
  using the notation of this lemma, there are $\tilde \varepsilon >0$ and
  $r_1'\in \mathbb N$ such that if $r>r_1'$ and $s>r T_\bl/ \tilde v$,
  $s\in \L_\ttime$,
  \begin{equation*}
    \P\big(B^\bl_r(y) \subset \tilde K^y_s,
      \tilde S_r^\nu (y, s) > p_\nu V_r/2\big)
    > \tilde \varepsilon.
  \end{equation*}
  Implication \eqref{eqn:fromZtoeta} with $W(\eta_0) = \{y\}$ then implies
  the claim of the lemma with
  $\varepsilon' = \varepsilon_{\ref{ass:irreducibility}}\tilde\varepsilon$,
  $c_0 = 2 \ceil{1 /  \tilde v}$, and  $r_1\ge r_1'$ so that
  $t_{\ref{ass:irreducibility}}< r_1 T_\bl/\tilde v$.
\end{proof}

\begin{remark}
  \label{rem:survival}
  Lemmas~\ref{lem:spread_goodconf}, \ref{lem:manygoodblocks} together imply
  that the survival probability $\P_\zeta({\tau_\ext=\infty})$ is bounded
  from below by a positive constant uniformly over all starting
  configurations $\zeta\not\equiv 0$, that is, over all initial conditions
  containing at least one particle.
\end{remark}

\begin{remark}
  For future reference, observe that the proof of
  Lemma~\ref{lem:spread_goodconf} only requires
  Assumptions~\ref{ass:flow}--\ref{ass:domination}.
  Lemma~\ref{lem:manygoodblocks}
  requires, in addition, Assumption~\ref{ass:irreducibility}.
\end{remark}

\section{Well-started blocks close to the right edge}
\label{sec:liverpool}

The goal of this section is to provide control of the configuration
$\eta_t$ close to its right edge. This control will be essential in order to
prove the approximate subadditivity stated in
Theorem~\ref{thm:perturbed_subadditivity}. The main result of this section,
Proposition~\ref{prop:hitandblock} below, shows that for $x>0$ large, at the
time $\tau_\hit(x)$ when there is an occupied site to the right of $x$ for the
first time, there are, with high probability, many well-started blocks in the
configuration $\eta_{\tau_\hit(x)}$ at distance $O(\log x)$ from $x$.

To state the result, we recall $W(\eta )$ from \eqref{eqn:Weta},
and for $\rho\in \mathbb N$ and $\eta\in E^{\mathbb Z}$ we define
\begin{equation}
  \label{def:D_rho(x,eta)}
  D_\rho(x, \eta)
  \coloneq \inf\bigg\{ r \in \N: \exists \, z \in B_r^\bl(\bbracket x )
    \text{ such that } \abs{W(\eta) \cap B_\rho^\bl(z)}
    \ge \frac{p_\nu V_\rho}{2}\bigg\}
\end{equation}
to be the distance (at the block scale) from $\bbracket x$ to the centre of
the first ball of radius $\rho$ containing sufficiently many well-started
blocks in $\eta$.

\begin{proposition}
  \label{prop:hitandblock}
  Under Assumptions~\ref{ass:flow}--\ref{ass:irreducibility}
  there are constants $\Cl{c:liverbig}, \Cl{c:liversmall}\in (0,\infty)$, and
  $\varepsilon \in (0,1)$ such that for every $\eta_0$ which is non-zero only at
  the origin, every $x>0$, and every $\rho < \varepsilon \sqrt x$,
  \begin{equation*}
    \P_{\eta_0}\big( D_\rho(x, \eta_{\tau_\hit(x)})> \rho^2,
      \tau_\hit(x)<\infty\big)
    < \Cr{c:liverbig}\big(e^{-\Cr{c:liversmall} \sqrt x}
      + e^{-\Cr{c:liversmall} \rho + 2\log x}\big).
  \end{equation*}
\end{proposition}

\begin{remark}
  For future reference, we remark that the proof of
  Proposition~\ref{prop:hitandblock} only requires (a), (b) from
  Assumption~\ref{ass:coarse}.
\end{remark}

To prove this proposition we need three auxiliary lemmas. The first one
states that at a deterministic large time $n$, if $\eta_n$ is non-zero at a
given site $x$, then, with high probability, there are many well-started
blocks close to $x$.

\begin{lemma}
  \label{lem:neverwalkalone}
  There exist constants
  $\Cl{c:nwabig}, \Cl{c:nwasmall}, \Cl{c:nwalb} \in (0,\infty)$ such that for
  every $n\in \mathbb N$, $\rho \le \Cr{c:nwalb} \sqrt n$, $x\in \mathbb Z$,
  and every $\eta_0\in E^\Z$
  \begin{equation*}
    \P_{\eta_0} \big( \eta_n(x) \neq 0, D_\rho(x,\eta_n)>\rho^2 \big)
    \le  \Cr{c:nwabig} e^{-\Cr{c:nwasmall} \rho}.
  \end{equation*}
\end{lemma}

\begin{proof}
  The idea of the proof is simple. If there is a particle at $x$ at time $n$,
  then at every time $n'<n$ there must be a particle in $B_{(n-n')R}(x)$, by
  the last part of Assumption~\ref{ass:flow}. If $n-n'$ is large enough
  (approximately, larger than $c \rho T_\bl$), then by applying first
  Lemma~\ref{lem:manygoodblocks} and then Lemma~\ref{lem:spread_goodconf}, the
  particle present at time $n'$ has uniformly positive probability to create a
  ball of radius $\rho $ where $W(\eta_n)$ is large enough. By restricting
  $n'$ to a suitable subset of $[0,n]$ of size of order $\rho$, the events for
  various $n'$'s can be made essentially independent, which then yields the
  required upper bound.

  We proceed with the formal proof. Throughout, we assume that $\rho $ is
  not too large compared to $n$, so that all considered times are
  non-negative. It can be verified that there is a constant
  $\Cr{c:nwalb}<\infty$  so that the condition
  $\rho \le \Cr{c:nwalb} \sqrt n$ in the statement ensures this.

  Fix $\rho, n  \in \N$, $x\in \mathbb Z$, and denote by $B = B(n,\rho ,x)$
  the event in the statement,
  \begin{equation}
    \label{eqn:nwaB}
    \begin{split}
      B &= \big\{\eta_n(x) \neq 0, D_\rho (x,\eta_n) > \rho^2\big\}
      \\&= \big\{\eta_n(x) \neq 0\big\}
      \cap \bigg(\bigcap_{z\in B^\bl_{\rho^2}(\bbracket x)}
        \big\{\abs{W(\eta_n) \cap B_\rho (z)} < p_\nu V_\rho /2\big\}\bigg).
    \end{split}
  \end{equation}
  Let $t_0 = (r_1 \vee \rho ) c_0$ with $r_1, c_0$ as in
  Lemma~\ref{lem:manygoodblocks}, and let $\tilde v$ be as in
  Lemma~\ref{lem:percolation}. Set
  \begin{equation*}
    J = \bigg\{\ceil[\Big]{\frac{2\rho +1}{\tilde v}} + i t_0 :
      i \in \{0, \dots, \floor{C \rho} \}\bigg\},
  \end{equation*}
  where $C$ is a constant not depending on $\rho $ such that
  \begin{equation}
    \label{eqn:nwaC}
    (\max J + t_0) T_\bl R \le \rho^2 L_\bl/2.
  \end{equation}
  Consider the event
  \begin{equation*}
    A = \bigcup_{k\in J}
    \bigcup_{y\in B^\bl_{\rho^2}(\bbracket x) }
    \big\{\abs{W(\eta_{n-kT_\bl}) \cap B^\bl_\rho(y)}\ge p_\nu V_\rho /2\big\}
    \eqcolon  \bigcup_{k\in J}
    \bigcup_{y\in B^\bl_{\rho^2}(\bbracket x) } A_{k,y}.
  \end{equation*}
  Note that $A$ occurs if there is a time $s\in n - JT_\bl$ when there is a
  ball of radius $\rho $ (on the block scale) not far from $x$ where $\eta_s$
  contains many well-started blocks. We observe that
  \begin{equation}
    \label{eqn:Bub}
    \begin{split}
      \P_{\eta_0}(B)
      &\le \P_{\eta_0}\big(B \cap (\cup_{k,y} A_{k,y})\big)
      + \P_{\eta_0}(B \cap A^c )
      \\&\le \sum_{k\in J}
      \sum_{y\in B^\bl_{\rho^2}(\bbracket x)}
      \P_{\eta_0}(B \mid A_{k,y}) + \P_{\eta_0}(\{\eta_n(x) \neq 0\} \cap A^c ).
    \end{split}
  \end{equation}
  Therefore, since $\abs J =  C \rho +1 $ and
  $\abs{B^\bl_{\rho^2}(\bbracket x)} \le c \rho^2$, the lemma is proved if we
  can bound all probabilities on the right-hand side by
  $c e^{-c' \rho}$ for some finite positive constants $c,c'$.

  We first estimate the conditional probabilities appearing in the sum. Since
  $k T_\bl \ge (2\rho +1)T_\bl/\tilde v $ for every $k\in J$,
  Lemma~\ref{lem:spread_goodconf} applied with $\rho = \ell$, $z=0$, and
  $a=p_\nu V_\rho /2$ implies that (using also the translation invariance and
    the Markov property at time $n-kT_\bl$)
  \begin{equation*}
    \P_{\eta_0}\big(\abs{W(\eta_{n}) \cap B^\bl_\rho(y)}< p_\nu V_\rho /2
      \bigm|
      A_{k,y}\big)
    \le c e^{-c'\rho }
  \end{equation*}
  for some constants $c,c'\in (0,\infty)$. Therefore, using \eqref{eqn:nwaB},
  also $\P(B\mid A_{k,y}) \le c e^{-c' \rho}$ as required.

  We proceed with bounding the last probability in \eqref{eqn:Bub}. On the
  event $\eta_n (x) \neq 0$, because of the last part of
  Assumption~\ref{ass:flow}, for $n'<n$ there must be $y\in B_{(n-n')R}(x)$
  with $\eta_{n'}(y)\neq 0$. In particular, using that $C$ in the definition
  of $J$ satisfies \eqref{eqn:nwaC},
  \begin{equation*}
    \{\eta_n(x) \neq 0\}
    \subset \bigcap_{k\in J} D_k, \qquad \text{with }
    D_k=\big\{\exists z\in B_{\rho^2 L_\bl/2}(\bbracket x):
      \eta_{n-(k+t_0)T_\bl}(z)\neq 0\big\}.
  \end{equation*}
  Writing $A_k = \bigcup_{y\in B^\bl_{\rho^2}(\bbracket x)} A_{k,y}$ so that
  $A^c = \bigcap_{k\in J}A_k^c$, we obtain
  \begin{equation*}
    \{\eta_n(x) \neq 0\} \cap A^c
    \subset \bigcap_{k\in J} (D_k \cap A^c)
    \subset \bigcap_{k\in J} (D_k \cap A_k^c)
  \end{equation*}
  and thus
  \begin{equation*}
    \P_{\eta_0}(\{\eta_n(x) \neq 0\} \cap A^c )
    \le
    \P_{\eta_0}\Big(\bigcap_{k\in J} (D_k \cap A_k^c)\Big).
  \end{equation*}
  Note that the event $D_k$ only depends on $\eta_{n-(k+t_0)T_\bl}$ and $A_k$
  on $\eta_{n-kT_\bl}$. Therefore, by the Markov property applied on the time
  $s = n-(\min J + t_0)T_\bl$,
  \begin{equation*}
    \P_{\eta_0}\big(\{\eta_n(x) \neq 0\} \cap A^c \big)
    \le
    \E_{\eta_0}\bigg(\P_{\eta_s}\big( A_{\min J}^c\big);
      D_{\min J}\cap \bigcap_{k\in J\setminus \{\min J\}} (D_k \cap A_k^c)
      \bigg).
  \end{equation*}
  On the event $D_{\min J}$, due to our choice of $t_0 = c_0(r_1 \vee \rho )$,
  Lemma~\ref{lem:manygoodblocks} applied with $r = \rho $ and $\ell = 1$
  implies that $\P_{\eta_s}\big(A_{\min J}^c\big)\le 1-\varepsilon '$. Using
  this, and then by  applying the same argument iteratively, it follows that
  \begin{equation*}
    \begin{split}
      \P_{\eta_0}\big(\{\eta_n(x) \neq 0\} \cap A^c \big)
      &\le (1- \varepsilon ')
      \P_{\eta_0}\bigg(
        \bigcap_{k\in J\setminus\{\min J\}} (D_k \cap A_k^c)\bigg)
      \\&\le (1- \varepsilon' )^{\abs J} = (1- \varepsilon ')^{C\rho +1 }
      \le e^{-c'\rho }.
    \end{split}
  \end{equation*}
  This proves the required bound on the last probability in \eqref{eqn:Bub}
  and thus completes the proof of the lemma.
\end{proof}

The second technical lemma states that a point in a `coupled region' cannot
be unoccupied for a long time. For this, recall the notation from
Section~\ref{sec:comparison_perc}, in particular~\eqref{eqn:Kdef}.

\begin{lemma}
  \label{lem:incoupled}
  For $\eta_0\in E^\Z$, $x\in \Z$ and $0\le t \in \L_\ttime$
  let $A$ be the event
  \begin{equation*}
    A \coloneq \{\bbracket x \in \tilde K^{W(\eta_0)}_t \}.
  \end{equation*}
  Then, there are universal constants $C,c\in (0,\infty)$ such that, for
  every $T\ge t$,
  \begin{equation*}
    \P_{\eta_0}(\eta_s(x) = 0 \text{ for all }s\in [t,T], A
      ) \le C e^{-c \sqrt{T-t}}.
  \end{equation*}
\end{lemma}

\begin{proof}
  Let $\ell \coloneq T_\bl\floor{\sqrt{T-t}/ T_\bl}\in \L_\ttime$,
  $t_i \coloneq t+i \ell\in \L_\ttime$, and $k\coloneq \max\{i: t_{i+1}\le T\}$.
  Note that $k$ can be approximated by $\sqrt{T-t}$ and that the intervals
  $I_i=[t_{i},t_{i+1})$, $i=0,\dots,k$ are disjoint and contained in
  $[t,T]$, by construction. For every $0\le i \le k$, we set
  \begin{equation*}
    \tau_i
    = \inf\big\{n\ge 0: n\in \L_\ttime, \bbracket x \in W(\eta_{t_{i}+n})\big\}.
  \end{equation*}
  By \eqref{eqn:Kdef}, if $\bbracket x \in \tilde K^{W(\eta_0)}_t$,
  $\bbracket x \in \tilde K^{W(\eta_0)}_s$ for all $s\ge t$. Therefore, by
  Lemma~\ref{lem:percolation}\eqref{perc:manyzeros}, due to
  \eqref{eqn:fromZtoeta},
  \begin{equation*}
    \P_{\eta_0}\big( \tau_i \ge \ell,A\big) \le C e^{-c \ell}.
  \end{equation*}
  The probability in the statement of the lemma then satisfies
  \begin{equation}
    \label{eqn:iop}
    \begin{split}
      &\P_{\eta_0}\big(\eta_s(x) = 0 \text{ for all }s\in [t,T],A\big)
      \\&\le \P_{\eta_0}(A,\exists i\le  k: \tau_i \ge \ell)
      + \P_{\eta_0}\Big(\bigcap_{i\le k} \{\tau_i < \ell\}
        \cap \{\eta_s(x) = 0 \,\forall s \in I_i\}\Big).
    \end{split}
  \end{equation}
  The first probability on the right-hand side can be bounded by
  \begin{equation*}
    \P_{\eta_0}\big(A,\exists i\le k: \tau_i \ge \ell\big)
    \le k C e^{-c\ell} \le \sqrt{T-t}C e^{-c\sqrt {T-t}}
    \le Ce^{-c\sqrt{T-t}}.
  \end{equation*}
  Let $\mathcal{F}_n = \sigma(\{ \eta_j : j \le n \})$.
  The second probability in \eqref{eqn:iop} satisfies
  \begin{equation*}
    \begin{split}
      &\P_{\eta_0}\Big(\bigcap_{i\le k} \{\tau_i < \ell\}
        \cap \{\eta_s(x) = 0 \,\forall s \in I_i\}\Big)
      \\&= \E_{\eta_0}\bigg[
        \P_{\eta_0}\big(\tau_k<\ell, \eta_s(x) = 0
          \,\forall s\in I_k \bigm| \mathcal F_{t_k}\big)
        \prod_{i=0}^{k-1}
        \ind\big\{\tau_i<\ell, \eta_s(x)=0\,\forall s\in I_i\big\}\bigg].
    \end{split}
  \end{equation*}
  To bound this, we observe that
  \begin{equation*}
    \begin{split}
      \P_{\eta_0}&\big(\tau_{k}<\ell, \eta_s(x) = 0 \,\forall s\in I_k
        \bigm| \mathcal F_{t_k}\big)
      \\&\le  \E_{\eta_0}\big( \P_{\eta_0}(\eta_{t_k+\tau_k+1}(x) =0
          \mid \mathcal F_{t_k+\tau_k})
        \ind\{\tau_k<\ell\}\bigm| \mathcal F_{t_k}\big).
    \end{split}
  \end{equation*}
  By definition of $\tau_k$, $\bbracket x \in W(\eta_{t_k+\tau_k})$ and thus
  $\eta_{t_k+\tau_k} \in G_\loc^\eta(\bbracket x)$. Therefore, by
  Assumption~\ref{ass:coarse}(a), the probability inside the expectation
  is smaller than $1-\hat \varepsilon$. Inserting this back, we obtain
  \begin{equation*}
    \P_{\eta_0}\Big(\bigcap_{i\le k}\{\tau_i < \ell\} \cap \{\eta_s(x) = 0
        \,\forall s \in I_i\}\Big)
    \le (1-\hat \varepsilon)
    \P_{\eta_0}\Big(\bigcap_{i\le k-1} \{\tau_i < \ell\}
      \cap \{\eta_s(x) = 0 \,\forall s \in I_i\}\Big).
  \end{equation*}
  Applying the same argument $k$-times then yields
  \begin{equation*}
    \P_{\eta_0}\Big(\bigcap_{i\le k} \{\tau_i < \ell\} \cap \{\eta_s(x) = 0
        \,\forall s \in I_i\}\Big) \le (1-\hat \varepsilon)^k
    \le C e^{-c\sqrt{T-t}},
  \end{equation*}
  which completes the proof.
\end{proof}

The last technical lemma gives some preliminary estimates on $\tau_\hit(x)$.

\begin{lemma}
  \label{lem:hittingtime}
  There are $0<v_1<v_2<\infty$ and $\Cl{c:htsmall},\Cl{c:htbig}\in(0,\infty)$
  such that for every $\eta_0$ which is non-zero only at the origin, every
  $x\in \N$, and $a>0$,
  \begin{equation*}
    \begin{split}
      \P_{\eta_0}(\tau_\hit(x)<x/v_2) &= 0,
      \\ \P_{\eta_0}(x/v_1 + a \le \tau_\hit(x)< \infty)
      & \le \Cr{c:htbig} e^{-\Cr{c:htsmall}\sqrt {x + av_1}}.
    \end{split}
  \end{equation*}
\end{lemma}

\begin{proof}
  The first claim is obvious with $v_2 = R$, since our process runs in
  discrete time and has a finite range by Assumption~\ref{ass:flow}.

  To prove the second claim we recall \eqref{eq:vlowerbound}, fix
  \begin{equation}
    \label{eqn:qwerty}
    v_1 = \frac{\tilde v}8 \cdot \frac{L_\bl}{ T_\bl},\qquad
    t=\frac{x}{v_1} + a, \qquad
    \varepsilon = \frac{\tilde v L_\bl}{8 T_\bl R} \wedge \frac 14 , \qquad
    \delta = \frac 12\Cr{c:nwalb}\sqrt \varepsilon.
  \end{equation}
  Let $H$ be the event
  \begin{equation*}
    H= \Big\{ \exists y\in \L_\sspace, \abs {y}\le \varepsilon t R,
      \abs[\big]{B^\bl_{\delta \sqrt{ t}}(y) \cap
        W(\eta_{\bbfloor{\varepsilon t}})}
      \ge p_\nu V_{\delta \sqrt{t}}/2\Big\},
  \end{equation*}
  that is, at some intermediate time $\bbfloor{\varepsilon t}$ there is a
  ball of radius $\delta \sqrt t$ (on the block scale) containing many
  well-started blocks.
  Then,
  \begin{equation}
    \label{eqn:poi}
    \big\{t \le \tau_\hit(x) < \infty\big\} \subset
    \big(H^c \cap \{t\le \tau_\hit(x) < \infty\}\big)
    \cup \big(H \cap \{\tau_\hit(x) \ge t\}\big).
  \end{equation}
  On $t\le \tau_\hit (x)< \infty$ there must be at least one particle in
  $\eta_{\bbfloor{\varepsilon t}}$. In addition, by the union bound and
  Lemma~\ref{lem:neverwalkalone}, applied with $n=\bbfloor{\varepsilon t}$
  and $\rho  = \delta \sqrt t$ (which is allowed due to the choice of
    $\delta $),
  \begin{equation*}
    \P_{\eta_0}\big(\exists y\in \Z: \eta_{\bbfloor{\varepsilon t}}(y)\neq 0,
      \abs{y}
      \le  \varepsilon tR, D_{\delta \sqrt t} (y,\eta_{\bbfloor{\varepsilon
            t}})> \delta^2 t\big)
    \le c R \varepsilon  t \Cr{c:nwabig} e^{-\Cr{c:nwasmall} \delta \sqrt t}
    \le C e^{-c\sqrt t}.
  \end{equation*}
  Together these claims imply that
  the first event on the right-hand side of \eqref{eqn:poi} satisfies
  \begin{equation*}
    \P_{\eta_0}(H^c \cap \{t\le \tau_\hit(x) < \infty\})
    \le C e^{-c  \sqrt t}.
  \end{equation*}
  On the other hand, on $H$, with a suitable $y$ which makes the event $H$
  occur, for
  \(
    A \coloneq B^\bl_{\delta \sqrt t}(y)
    \cap W(\eta_{\bbfloor{\varepsilon t}})
  \),
  we have $\abs A \ge c \sqrt t$ and
  $A\subset B_{\bbfloor{\varepsilon t} R}(0)$.  By
  applying the Markov property on time $\bbfloor{\varepsilon t}$, using
  Lemma~\ref{lem:shift_blocks} (with this $A$,
    $\ell = \varepsilon t R/L_\bl$, $\rho = \delta \sqrt t$, $z=x = (t-a)v_1$
    and with $t$ there being $(\frac 34-\varepsilon )t\ge t/2$ for
    $t$ as in \eqref{eqn:qwerty}, so that the condition of this
    lemma is verified by our choice of constants), we obtain
  \begin{equation*}
    \P_{\eta_0}\Big(H
      \cap \Big\{\bbracket x\notin \tilde K^A_{\bbfloor{ 3t/4}}\Big\}\Big)
    \le C e^{-c \sqrt t}.
  \end{equation*}
  This implies that with high probability $(\bbracket x, \bbfloor{3t/4})$ is in
  the coupled region of the oriented percolation.
  Finally, Markov property and Lemma~\ref{lem:incoupled} yield
  \begin{equation*}
    \P_{\eta_0}\Big(\tau_\hit(x) \ge t
      \Bigm| \bbracket x\in \tilde K^A_{\bbfloor{ 3t/4 }}\Big)
    \le C e^{-c \sqrt t}.
  \end{equation*}
  Combining these together implies that the second event on the right-hand
  side of \eqref{eqn:poi} satisfies the same upper bound as the first one.
\end{proof}

We can now prove Proposition~\ref{prop:hitandblock}.

\begin{proof}[Proof of Proposition~\ref{prop:hitandblock}]
  Recall the constants $v_1$ and $v_2$ from Lemma~\ref{lem:hittingtime}. Observe
  that
  \begin{equation*}
    \begin{split}
      &\big\{D_\rho(x, \eta_{\tau_\hit(x)})> \rho^2,
        \tau_\hit(x)<\infty\big\}
      \\&\subset
      \big\{x/v_1 \le \tau_\hit(x) < \infty\big\}
      \cup\bigg( \bigcup_{t\in [x/v_2,x/v_1] }
        \bigcup_{y\in [-tR,tR]}
        \big\{\eta_t(y)\neq 0, D_\rho (y, \eta_t) > \rho^2\big\}\bigg).
    \end{split}
  \end{equation*}
  The probability of the first event is smaller than
  $\Cr{c:htbig} e^{-\Cr{c:htsmall}\sqrt x}$ by Lemma~\ref{lem:hittingtime}.
  The probability of the second event is bounded by
  Lemma~\ref{lem:neverwalkalone}, using the fact that we are taking union of
  at most $O(x^2)$ events, by
  \begin{equation*}
    c \Cr{c:nwabig} x^2 e^{-\Cr{c:nwasmall} \rho }
    \le c \Cr{c:nwabig} e^{-\Cr{c:nwasmall} \rho + 2 \log x }
  \end{equation*}
  for some $c < \infty$.
  Combining the last two estimates easily completes the proof.
\end{proof}

\begin{remark}
  \label{rem:hitandblock}
  Proposition~\ref{prop:hitandblock} ensures that at time $\tau_\hit(x)$, the
  configuration contains many well-started blocks not far from $x$ with
  suitably controlled probability. Note that the stopping time $\tau_\hit(x)$
  will in general not be concentrated on $\L_\ttime$. In some of the
  arguments below, it will be more convenient to work with the fixed
  space-time grid $\L = \L_\sspace \times \L_\ttime$ and not with a version
  which is shifted in the time direction by some $t \in \{1,2,\dots, T_\bl-1\}$.

  In fact, under the assumptions of Proposition~\ref{prop:hitandblock} we
  also have
  \begin{equation*}
    \P_{\eta_0}\big( D_\rho(x, \eta_{\bbfloor {\tau_\hit(x)} })> \rho^2,
    \tau_\hit(x)<\infty\big)
    < \Cr{c:liverbig}\big(e^{-\Cr{c:liversmall} \sqrt x}
    + e^{-\Cr{c:liversmall} \rho + 2\log x}\big).
  \end{equation*}
\end{remark}

\begin{proof}
  Note that by the last part of Assumption~\ref{ass:flow} and the
  definition of $\tau_\hit(x)$, the event $\tau_\hit(x)<\infty$
  implies that
  $\{ y : \eta_{\bbfloor{ \tau_\hit(x) }}(y) \neq 0 \} \cap
  B_{(\tau_\hit(x) - \bbfloor{ \tau_\hit(x) })R}(x) \neq
  \emptyset$.  We can then repeat the arguments from the proofs of
  Lemma~\ref{lem:neverwalkalone} and of
  Proposition~\ref{prop:hitandblock}, beginning with an (arbitrarily
  chosen) particle from
  $B_{(\tau_\hit(x) - \bbfloor{ \tau_\hit(x) })R}(x)$ in the
  configuration at time $\bbfloor{ \tau_\hit(x) }$ and thus
  enforce that all the blocks we use in the construction are aligned
  with $\L$.
\end{proof}

We close this section with the following two observations that follow from the
proof of Proposition~\ref{prop:hitandblock}. The first one gives moment
estimates for $\tau_\hit(x)$.

\begin{corollary}\label{cor:hit_moments}
  For every $\eta_0$ which is non-zero only at the origin and every $k>0$
  there is a constant $c = c(k)$ such that for every $x\in \N$
  \begin{equation*}
    \E_{\eta_0}[\tau_\hit(x)^k \ind_{\tau_\hit(x)<\infty}] \le c x^k.
  \end{equation*}
\end{corollary}

\begin{proof}
  From Lemma~\ref{lem:hittingtime} it follows that
  \begin{align*}
    \E_{\eta_0}[\tau_\hit(x)^k \ind_{\tau_\hit(x)<\infty}]
    & = \int_0^\infty k t^{k-1} \P(t < \tau_\hit(x) < \infty) \,dt
    \\ & \leq (x/v_1)^k
    + \Cr{c:htbig} \int_0^\infty k (x/v_1+a)^{k-1}
    e^{-\Cr{c:htsmall} \sqrt{x+a v_1}} \, d a \leq c x^k,
  \end{align*}
  which proves the claim.
\end{proof}

The second observation follows from Lemma~\ref{lem:neverwalkalone} and
implies that when the process is conditioned on survival
(cf.~\eqref{eqn:barPzeta}), it will
produce many well-started blocks with probability very close to $1$.

\begin{lemma}
  \label{lem:blocks_somewhere}
  There exist $\Cl{c:bsbig},  \Cl{c:bstilde} \in (0,\infty)$
  such that for all $n \in \N$ and any $\zeta \in E^\Z$ which is non-zero
  only at the origin,
  \begin{equation*}
    \bar{\P}_{\zeta}\big( D_{\Cr{c:nwalb}\sqrt{n}}(0, \eta_n)
      \leq \Cr{c:bstilde} n \big)
    \geq 1 - \Cr{c:bsbig} e^{-\Cr{c:nwasmall} \sqrt n},
  \end{equation*}
  where the constants $\Cr{c:nwasmall}, \Cr{c:nwalb}$ come from
  Lemma~\ref{lem:neverwalkalone}.
\end{lemma}

\begin{proof}
  For a suitable choice of $\Cr{c:bstilde}$ much larger than
  $\Cr{c:nwalb}^2$, we have
  \begin{align*}
    \{ D_{\Cr{c:nwalb}\sqrt{n}}(0, \eta_n) > \Cr{c:bstilde} n \}
    \subset \bigcap_{x \in [-nR, nR]}
    \{ D_{\Cr{c:nwalb}\sqrt{n}}(x, \eta_n) > \Cr{c:nwalb}^2 n \}.
  \end{align*}
  Furthermore,
  $\bar{\P}_\zeta(\cup_{x \in [-nR,nR]} \{ \eta_n(x) \neq 0 \}) = 1$. Thus
  \begin{align*}
    \bar{\P}_{\zeta}\big( D_{\Cr{c:nwalb}\sqrt{n}}(0, \eta_n)
      > \Cr{c:bstilde} n \big)
    &\leq \bar{\P}_{\zeta}\Big( \bigcup_{x \in [-nR,nR]} \{ \eta_n(x) \neq 0,
        D_{\Cr{c:nwalb}\sqrt{n}}(x, \eta_n) > \Cr{c:nwalb}^2 n \}\Big)
    \\& \leq \frac{1}{\P_{\zeta}(\tau_\ext=\infty)} \sum_{x=-nR}^{nR}
    \P_\zeta\big( \eta_n(x) \neq 0, D_{\Cr{c:nwalb}\sqrt{n}}(x, \eta_n) >
      \Cr{c:nwalb}^2 n \big)
    \\& \leq \Cr{c:bsbig} n R e^{-\Cr{c:nwasmall} \sqrt{n}}
  \end{align*}
  by Lemma~\ref{lem:neverwalkalone} and Remark \ref{rem:survival}, if
  $\Cr{c:bsbig}$ is chosen large enough.
\end{proof}

\section{Proof of Theorem~\ref{thm:perturbed_subadditivity}}
\label{sec:mainproof}

In order to show Theorem~\ref{thm:perturbed_subadditivity} we have to
compare the evolution of $\eta$ conditioned on non-extinction (started from a
  configuration $\eta_0$ with a single particle at the origin) after the time
${\tau}_\hit(x)$, with another copy of $\eta$ started from a configuration
that contains a single particle at $x$, conditioned again to survive.

We split the proof into several steps. First, we compare two copies, $\eta^1$
and $\eta^2$, of the process started from configurations with certain
properties and \emph{not} conditioned to survive, and prove that their fronts
can be coupled with high probability so that they stay ordered. This is the
content of Proposition~\ref{prop:couplefronts} in
Section~\ref{sec:couplefronts}, which is one of the key ideas of the paper. It
strongly relies on Assumptions~\ref{ass:coarse} and~\ref{ass:domination}, which
imply that if two configurations agree locally at some time, then they can be
kept locally identical for a long (possibly infinite) time. Another key
ingredient is Assumption~\ref{ass:coupling}, which implies that, eventually,
any two configurations become identical near their right edges.

We then show in Section~\ref{sec:delaybounds} that with high probability the
evolution of the original process from time $\tau_\hit(x)$ and the new
process started with one particle at $x$ reach these configurations
relatively quickly. In
Section~\ref{sec:shiftandsurvive} we modify the coupling from
Proposition~\ref{prop:couplefronts} to hold when both
processes are conditioned to survive. We will then prove
Theorem~\ref{thm:perturbed_subadditivity} in Section~\ref{sec:proof}.

Before we start with this programme, we recall from Assumption~\ref{ass:flow}
that the process $\eta$ is constructed on a probability space containing a
family $U = \{U(x,n): {x\in \Z}, {n\in \N_0}\}$ of i.i.d.~uniformly
distributed random variables by \eqref{eqn:flow1} and consequently also by
\eqref{eqn:eta_U_flow}. When constructing two copies $\eta^1$ and $\eta^2$ of
the process, we view them as being constructed on a probability space
containing two copies $U^1$, $U^2$ of $U$. Hence, constructing couplings of
two processes $\eta^1$ and $\eta^2$ actually means constructing couplings of
$U^1$ and $U^2$. Moreover, by construction, constructing a coupling of
$\eta^1$ and $\eta^2$ on certain time intervals $[t^1,t^1+t]$ and
$[t^2,t^2+t]$, respectively, means that we only couple the slices
$\{U^1(x,n):x\in \Z, t^1 < n \le t^1 +t \}$ and
$\{U^2(x,n):x\in \Z, t^2 < n \le t^2 +t \}$. This should be kept in mind
during the whole section and we will mostly not comment on it explicitly.

For typographical reasons, throughout this section we use the following
notation for translations of sets and configurations: For $z\in \Z$,
$A \subset \Z$ and  $\zeta \in E^\Z$, we set
\begin{equation}
  \label{eqn:shifts}
  A+z \coloneq \theta_z(A) = \{ a+z : a \in A \},
  \qquad
  \zeta + z  \coloneq \theta_z(\zeta) =  ( \zeta(x-z) : x \in \Z ).
\end{equation}

\subsection{Coupling the fronts}
\label{sec:couplefronts}

In this section we consider two copies $\eta^1$ and $\eta^2$ of our process,
whose starting configurations have certain suitable properties. We show that
with high probability, their fronts remain ordered forever. To state the
result we define
\begin{equation}
  \label{eqn:Maxdef}
  M_n^i \coloneq \max \{ x \in \Z : \eta_n^i(x) \neq 0 \},
  \quad n \ge 0, \ i =1,2,
\end{equation}
to be the position of the right-most particle in configuration $\eta^i$ at
time $n$. (We use the convention $M_n^i= - \infty$ if $\eta_n^i\equiv 0$.)

\begin{proposition}
  \label{prop:couplefronts}
  There are constants $c,C \in (0,\infty)$ such that for every
  $\ell,\rho \in \N$ with
  $(\ell - \rho L_\bl)/(2R ) \ge T_\bl+ (2\rho +1) T_\bl/\tilde v$
  and every
  $\eta_0^1, \eta_0^2$ satisfying
  \begin{align*}
    &\abs{ W(\eta^1_0 )\cap B^\bl_\rho(0) }\ge p_\nu V_\rho /2,
    \\&\abs{W(\eta_0^2)} \ge p_\nu V_\rho/2
    \qquad\text{and}\qquad\eta_0^2(x) = 0 \text{ for all }x > -\ell,
  \end{align*}
  there is a coupling $ P_{\eta_0^1,\eta_0^2}$ of processes $\eta^1$, $\eta^2$
  started from $\eta_0^1$, $\eta_0^2$, respectively, such that
  \begin{equation*}
    P_{\eta_0^1,\eta_0^2}\big(M^1_k \ge M^2_k \text{ for all } k\ge 0\big)
    \ge 1 - 2 C e^{-c \rho}
    - C \sqrt{\ell-\rho L_\bl}\,e^{ -c \sqrt{\ell -\rho L_\bl}}.
  \end{equation*}
\end{proposition}

\begin{proof}
  We start by explaining the idea of the proof. The goal is to construct a
  coupling so that, eventually, $M^1_n-M^2_n$ remains constant. This will be
  achieved as follows. We first use the coupling from
  Assumption~\ref{ass:coupling} iteratively to obtain a random time $T$ when
  $\eta^1$ and $\eta^2$ agree near their fronts and contain a well-started
  block there. When this happens, we couple $\eta^1$ and $\eta^2$ after $T$
  by setting $U^2$ to be the spatial shift of $U^1$ by $M^1_T-M^2_T$.
  In the case when the underlying oriented
  percolation started from the well-started block in $\eta^1$ survives
  forever, we show that the difference between the two maxima remains
  constant under this coupling, as required. When this oriented percolation
  does not survive, we repeat the construction starting with the coupling in
  Assumption~\ref{ass:coupling} again. We will show that the first
  alternative occurs after a bounded number of repetitions of these steps,
  with high probability. Since, at time 0, the configuration $\eta^1$ has a
  head start of size roughly $\ell$, this will imply that $\eta^2$ does not
  overtake $\eta^1$, which achieves our aim. For an illustration of the whole
  construction, see Figure~\ref{fig:shifted_coupling}.

  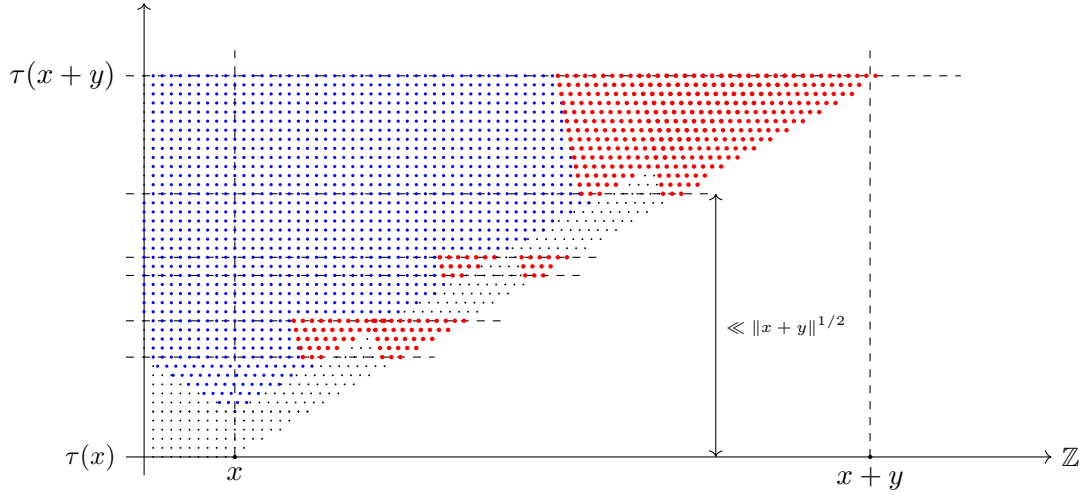
\begin{figure}[ht]
       \begin{tikzpicture}[scale=1.2]

    % ---------- Axes ----------
    \draw[->] (-0.2,0) -- (10,0) node[right] {\small{$\mathbb{Z}$}};
    \draw[->] (0,-0.2) -- (0,5);

    % ---------- Points x and x+y ----------
    \coordinate (x) at (1,0);
    \coordinate (xy) at (8,0);
    \draw[fill] (x) circle (0.5pt) node[below] {$x$};
    \draw[fill] (xy) circle (0.5pt) node[below] {$x+y$};

    \node[left] at (-0.2,0) {\small{$\tau(x)$}};

    % ----- Horizontal dashed line tau(x+y)
    \draw[dashed] (-0.2,4.2) -- (9,4.2);
    \node[left] at (-0.2,4.2) {$\tau(x+y)$};

    % -----  Horizontal lines of coupling times
    \draw[dashed] (-0.2,1.1) -- (3.2,1.1);
    \draw[dashed] (-0.2,1.5) -- (4,1.5);
    \draw[dashed] (-0.2,2) -- (4.8,2);
    \draw[dashed] (-0.2,2.2) -- (5,2.2);
    \draw[dashed] (-0.2,2.9) -- (6.3,2.9);

    % ----- Vertical dashed lines at x and x+y
    \draw[dashed] (1,0) -- (1,4.5);
    \draw[dashed] (8,0) -- (8,4.5);

    % ----- vertical lines of coupling time
    \draw[<->] (6.3,0) -- (6.3,2.9);
    \node[right] at (6.3,1.45) {\tiny{$\ll \| x+y \|^{1/2}$}};

    % ---------- REGION 1:
    \foreach \j in {0,...,5} {
      \pgfmathsetmacro{\y}{\j*0.1}
      \pgfmathsetmacro{\xmax}{1+1.7*\y}
      \pgfmathsetmacro{\N}{int(\xmax/0.1)}
      \foreach \i in {1,...,\N} {
        \pgfmathsetmacro{\x}{\i*0.1}
        \fill[black] (\x,\y) circle (0.3pt);
      }
    }

    % ---------- REGION 2:
    \foreach \j in {5,...,10} {
      \pgfmathsetmacro{\y}{\j*0.1}
      \pgfmathsetmacro{\xleftblack}{max(0,1-1.7*(\y-0.5))}
      \pgfmathsetmacro{\xblueleft}{\xleftblack}
      \pgfmathsetmacro{\xblueright}{1+1.7*(\y-0.5)}
      \pgfmathsetmacro{\xrightblack}{1+1.7*\y}

      % Left black
      \pgfmathsetmacro{\N}{int(\xleftblack/0.1)}
      \foreach \i in {1,...,\N} {
        \pgfmathsetmacro{\x}{\i*0.1}
        \fill[black] (\x,\y) circle (0.3pt);
      }
      % Blue
      \pgfmathsetmacro{\N}{int((\xblueright-\xblueleft)/0.1)}
      \foreach \i in {0,...,\N} {
        \pgfmathsetmacro{\x}{\xblueleft + \i*0.1}
        \fill[blue] (\x,\y) circle (0.5pt);
      }
      % Right black
      \pgfmathsetmacro{\N}{int((\xrightblack-\xblueright)/0.1)}
      \foreach \i in {0,...,\N} {
        \pgfmathsetmacro{\x}{\xblueright + \i*0.1}
        \fill[black] (\x,\y) circle (0.3pt);
      }
    }

    % ---------- REGION 3:
    \foreach \j in {11,...,15} {
      \pgfmathsetmacro{\y}{\j*0.1}

      % Boundaries
      \pgfmathsetmacro{\xblackleft}{max(0, 1 - 1.7*(\y-0.5))}
      \pgfmathsetmacro{\xblueleft}{\xblackleft}
      \pgfmathsetmacro{\xblueright}{1.9-0.25*(\y-0.5)}
      \pgfmathsetmacro{\xredright}{1 + 1.7*(\y-0.5)}
      \pgfmathsetmacro{\xblackright}{2.9-0.25*\y}
      \pgfmathsetmacro{\xredmax}{1+1.7*\y}

      % Left black
      \pgfmathsetmacro{\N}{int(\xblackleft/0.1)}
      \foreach \i in {1,...,\N} {
        \pgfmathsetmacro{\x}{\i*0.1}
        \fill[black] (\x,\y) circle (0.3pt);
      }

      % Blue
      \pgfmathsetmacro{\N}{int((\xblueright-\xblueleft)/0.1)}
      \foreach \i in {1,...,\N} {
        \pgfmathsetmacro{\x}{\xblueleft + \i*0.1}
        \fill[blue] (\x,\y) circle (0.5pt);
      }

      % Red
      \pgfmathsetmacro{\N}{int((\xredright-\xblueright)/0.1)}
      \foreach \i in {0,...,\N} {
        \pgfmathsetmacro{\x}{\xblueright + \i*0.1}
        \fill[red] (\x,\y) circle (0.7pt);
      }

      % Right black
      \pgfmathsetmacro{\N}{int((\xblackright-\xredright)/0.1)}
      \foreach \i in {1,...,\N} {
        \pgfmathsetmacro{\x}{\xredright + \i*0.1}
        \fill[black] (\x,\y) circle (0.3pt);
      }

      % Right red
      \pgfmathsetmacro{\N}{int((\xredmax-\xblackright)/0.1)}
      \foreach \i in {0,...,\N} {
        \pgfmathsetmacro{\x}{\xblackright + \i*0.1}
        \fill[red] (\x,\y) circle (0.7pt);
      }
    }

    % ---------- REGION 4:
    \foreach \j in {16,...,19} {
      \pgfmathsetmacro{\y}{\j*0.1}
      \pgfmathsetmacro{\xblackleft}{max(0,1 - 1.7*(\y-0.5))}
      \pgfmathsetmacro{\xblueleft}{\xblackleft}
      \pgfmathsetmacro{\xblueright}{1 + 1.7*(\y-0.5)}
      \pgfmathsetmacro{\xblackright}{1+1.7*\y}

      % Left black
      \pgfmathsetmacro{\N}{int(\xblackleft/0.1)}
      \foreach \i in {0,...,\N} {
        \pgfmathsetmacro{\x}{\i*0.1}
        \fill[black] (\x,\y) circle (0.3pt);
      }
      % Blue
      \pgfmathsetmacro{\N}{int((\xblueright-\xblueleft)/0.1)}
      \foreach \i in {0,...,\N} {
        \pgfmathsetmacro{\x}{\xblueleft + \i*0.1}
        \fill[blue] (\x,\y) circle (0.5pt);
      }
      % Right black
      \pgfmathsetmacro{\N}{int((\xblackright-\xblueright)/0.1)}
      \foreach \i in {0,...,\N} {
        \pgfmathsetmacro{\x}{\xblueright + \i*0.1}
        \fill[black] (\x,\y) circle (0.3pt);
      }
    }

    % ---------- REGION 5:
    \foreach \j in {20,...,22} {
      \pgfmathsetmacro{\y}{\j*0.1}

      % Boundaries
      \pgfmathsetmacro{\xblackleft}{max(0, 1 - 1.7*(\y-0.5))}
      \pgfmathsetmacro{\xblueleft}{\xblackleft}
      \pgfmathsetmacro{\xblueright}{3.6- 0.2*(\y-0.5)}
      \pgfmathsetmacro{\xredright}{1 + 1.7*(\y-0.5)}
      \pgfmathsetmacro{\xblackright}{4.6-0.2*\y}
      \pgfmathsetmacro{\xredmax}{1+1.7*\y}

      % Left black
      \pgfmathsetmacro{\N}{int(\xblackleft/0.1)}
      \foreach \i in {1,...,\N} {
        \pgfmathsetmacro{\x}{\i*0.1}
        \fill[black] (\x,\y) circle (0.3pt);
      }

      % Blue
      \pgfmathsetmacro{\N}{int((\xblueright-\xblueleft)/0.1)}
      \foreach \i in {0,...,\N} {
        \pgfmathsetmacro{\x}{\xblueleft + \i*0.1}
        \fill[blue] (\x,\y) circle (0.5pt);
      }

      % Red
      \pgfmathsetmacro{\N}{int((\xredright-\xblueright)/0.1)}
      \foreach \i in {0,...,\N} {
        \pgfmathsetmacro{\x}{\xblueright + \i*0.1}
        \fill[red] (\x,\y) circle (0.7pt);
      }

      % Right black
      \pgfmathsetmacro{\N}{int((\xblackright-\xredright)/0.1)}
      \foreach \i in {0,...,\N} {
        \pgfmathsetmacro{\x}{\xredright + \i*0.1}
        \fill[black] (\x,\y) circle (0.3pt);
      }

      % Right red
      \pgfmathsetmacro{\N}{int((\xredmax-\xblackright)/0.1)}
      \foreach \i in {0,...,\N} {
        \pgfmathsetmacro{\x}{\xblackright + \i*0.1}
        \fill[red] (\x,\y) circle (0.7pt);
      }
    }

    % ---------- REGION 6:
    \foreach \j in {23,...,28} {
      \pgfmathsetmacro{\y}{\j*0.1}

      % Boundaries
      \pgfmathsetmacro{\xblackleft}{max(0, 1 - 1.7*(\y-0.5))}
      \pgfmathsetmacro{\xblueleft}{\xblackleft}
      \pgfmathsetmacro{\xblueright}{1 + 1.7*(\y-0.5)}
      \pgfmathsetmacro{\xblackright}{1+1.7*\y}

      % Left black
      \pgfmathsetmacro{\N}{int(\xblackleft/0.1)}
      \foreach \i in {0,...,\N} {
        \pgfmathsetmacro{\x}{\i*0.1}
        \fill[black] (\x,\y) circle (0.3pt);
      }

      % Blue
      \pgfmathsetmacro{\N}{int((\xblueright-\xblueleft)/0.1)}
      \foreach \i in {0,...,\N} {
        \pgfmathsetmacro{\x}{\xblueleft + \i*0.1}
        \fill[blue] (\x,\y) circle (0.5pt);
      }

      % Right black
      \pgfmathsetmacro{\N}{int((\xblackright-\xblueright)/0.1)}
      \foreach \i in {0,...,\N} {
        \pgfmathsetmacro{\x}{\xblueright + \i*0.1}
        \fill[black] (\x,\y) circle (0.3pt);
      }
    }

    % ---------- REGION 7:
    \foreach \j in {29,...,31} {
      \pgfmathsetmacro{\y}{\j*0.1}

      % Boundaries
      \pgfmathsetmacro{\xblackleft}{max(0, 1 - 1.7*(\y-0.5))}
      \pgfmathsetmacro{\xblueleft}{\xblackleft}
      \pgfmathsetmacro{\xblueright}{5.3 -0.2*(\y-0.5)}
      \pgfmathsetmacro{\xredright}{1+ 1.7*(\y-0.5)}
      \pgfmathsetmacro{\xblackright}{6.3-0.2*\y}
      \pgfmathsetmacro{\xredmax}{1+1.7*\y}

      % Left black
      \pgfmathsetmacro{\N}{int(\xblackleft/0.1)}
      \foreach \i in {1,...,\N} {
        \pgfmathsetmacro{\x}{\i*0.1}
        \fill[black] (\x,\y) circle (0.3pt);
      }

      % Blue
      \pgfmathsetmacro{\N}{int((\xblueright-\xblueleft)/0.1)}
      \foreach \i in {1,...,\N} {
        \pgfmathsetmacro{\x}{\xblueleft + \i*0.1}
        \fill[blue] (\x,\y) circle (0.5pt);
      }

      % Red
      \pgfmathsetmacro{\N}{int((\xredright-\xblueright)/0.1)}
      \foreach \i in {0,...,\N} {
        \pgfmathsetmacro{\x}{\xblueright + \i*0.1}
        \fill[red] (\x,\y) circle (0.7pt);
      }

      % Right black
      \pgfmathsetmacro{\N}{int((\xblackright-\xredright)/0.1)}
      \foreach \i in {1,...,\N} {
        \pgfmathsetmacro{\x}{\xredright + \i*0.1}
        \fill[black] (\x,\y) circle (0.3pt);
      }

      % Right red
      \pgfmathsetmacro{\N}{int((\xredmax-\xblackright)/0.1)}
      \foreach \i in {0,...,\N} {
        \pgfmathsetmacro{\x}{\xblackright + \i*0.1}
        \fill[red] (\x,\y) circle (0.7pt);
      }
    }

    % ---------- REGION 8:
    \foreach \j in {32,...,42} {
      \pgfmathsetmacro{\y}{\j*0.1}

      % Boundaries
      \pgfmathsetmacro{\xblackleft}{max(0, 1 - 1.7*(\y-0.5))}
      \pgfmathsetmacro{\xblueleft}{\xblackleft}
      \pgfmathsetmacro{\xblueright}{5.3 -0.2*(\y-0.5)}
      \pgfmathsetmacro{\xredright}{1+ 1.7*(\y-0.5)}
      \pgfmathsetmacro{\xblackright}{6.3-0.2*\y}
      \pgfmathsetmacro{\xredmax}{1+1.7*\y}

      % Left black
      \pgfmathsetmacro{\N}{int(\xblackleft/0.1)}
      \foreach \i in {1,...,\N} {
        \pgfmathsetmacro{\x}{\i*0.1}
        \fill[black] (\x,\y) circle (0.3pt);
      }

      % Blue
      \pgfmathsetmacro{\N}{int((\xblueright-\xblueleft)/0.1)}
      \foreach \i in {1,...,\N} {
        \pgfmathsetmacro{\x}{\xblueleft + \i*0.1}
        \fill[blue] (\x,\y) circle (0.5pt);
      }

      % Red
      \pgfmathsetmacro{\N}{int((\xredright-\xblueright)/0.1)}
      \foreach \i in {0,...,\N} {
        \pgfmathsetmacro{\x}{\xblueright + \i*0.1}
        \fill[red] (\x,\y) circle (0.7pt);
      }

      % Right red
      \pgfmathsetmacro{\N}{int((\xredmax-\xblackright)/0.1)}
      \foreach \i in {0,...,\N} {
        \pgfmathsetmacro{\x}{\xblackright + \i*0.1}
        \fill[red] (\x,\y) circle (0.7pt);
      }
    }

  \end{tikzpicture}

      \caption{Shifted coupling construction: the original process (in black)
        started from $\eta_0$ at time 0 is compared after time $\tau(x)$ with
        another process (in blue) started from $\eta_x$ at some time after
        $\tau(x)$ so that the original process has enough `lead'. At each
        step, the fronts agree with (tiny but) positive probability,
        indicated in red. When this occurs, from this time we couple the
        evolutions of the processes. If the coupling fails, we wait until the
        fronts agree again. With high probability, from some time of order
        $\ll \abs{x+y}^{1/2}$ the coupling never fails. Until this time, the
        original process has preserved its lead, and so it hits $x+y$ before
        the blue process.}
      \label{fig:shifted_coupling}
  \end{figure}

  We now proceed with the formal proof. Let $t_5$ be as in
  Assumption~\ref{ass:coupling}. For $t\in  \L_\ttime$ define the set
  \begin{equation*}
    \mathcal Z_t \coloneq \big\{ z \in \L_\sspace : \eta_t^1 \in G_\loc^\eta(z),
      \eta_t^1(x+M_t^1)=\eta_t^2(x+M_t^2) \, \forall x > z- M_t^1- K L_\bl
      \big\}.
  \end{equation*}
  Note that if the event $\{ \mathcal Z_t \neq \emptyset \}$ occurs, then
  $\eta^1_t$ and $\eta^2_t$ agree near their fronts and $\eta^1_t$ contains a
  well-started block. (As a consequence, $\eta^2_t$ contains a well-started
    block as well, but this block might not be aligned with the lattice
    $\L_\sspace$.)

  We now construct the required coupling of the processes $\eta^1$ and
  $\eta^2$. To this end, we also construct two sequences of stopping times
  $(T_j)$, $(S_j)$ given by $T_0=S_0=0$, and by iterating the following two
  steps for $j\ge 1$:

  \smallskip

  \textbf{Step 1.}
  If $S_{j-1}<\infty$, we iteratively construct a coupling of the two
  processes in the time intervals $(S_{j-1}+k t_5, S_{j-1}+(k+1)t_5 ]$,
  $k\in \N_0$, using Assumption~\ref{ass:coupling}, until
  $\{\mathcal Z_t\neq \emptyset\}$ occurs. To this end, for each such $k$, if
  both $\eta_{S_{j-1}+k t_5}^i$, $i=1,2$, are non-empty, we define the
  starting configurations (recall the notation \eqref{eqn:shifts})
  \begin{equation*}
    \hat \eta_0^1
    \coloneq \eta_{S_{j-1}+k t_5}^1 -\bbceil{M_{S_{j-1}+k t_5}^1}
    \qquad \text{and} \qquad
    \hat \eta_0^2
    \coloneq \eta_{S_{j-1}+k t_5}^2 - \bbceil{M_{S_{j-1}+k t_5}^2},
  \end{equation*}
  (note that those implicitly depend on $j$ and $k$ and satisfy the
    hypothesis of Assumption~\ref{ass:coupling} on the initial conditions).
  We then let  $\hat \eta^1$ and $\hat \eta^2$ evolve according to the
  coupling $\P_{\hat \eta_0^1,\hat \eta_0^2}$ from
  Assumption~\ref{ass:coupling} for $t_5$ steps, and for $s \in [1,t_5]$, we set
  \begin{equation*}
    \eta_{S_{j-1}+ k t_5 +s}^1
    = \hat \eta_s^1 +\bbceil{M_{S_{j-1}+k t_5}^1}
    \quad \text{and} \quad
    \eta_{S_{j-1}+ k t_5 +s}^2
    = \hat \eta_{s}^2 +\bbceil{M_{S_{j-1}+k t_5}^2}.
  \end{equation*}
  If one of the configurations $\eta_{S_{j-1}+k t_5}^i$, $i=1,2$, is empty,
  we let $\eta^1$ and $\eta^2$ run independently after $S_{j-1}+k t_5$, and
  set $T_j=\infty$. Otherwise, we define
  \begin{equation*}
    T_j  \coloneq S_{j-1}
    + t_5 \cdot \inf\{ k\in\N : \mathcal{Z}_{S_{j-1}+k t_5} \neq\emptyset \}.
  \end{equation*}
  If $S_{j-1}=\infty$, we set $T_j=\infty$.

  \smallskip

  \textbf{Step 2.}
  If $T_j < \infty$,  let $Z_j = \max \mathcal Z_{T_j}$ (which is well
    defined since $\mathcal Z_{T_j}$ is non-empty). Similarly as in
  Assumption~\ref{ass:domination}, we set
  $Y^1(z,t)$ to be the indicator of the event
  $\{U^1\in G_\loc^U(z,t)\}$, and consider an i.i.d.~Bernoulli
  field $\tilde Y^1$ satisfying $Y^1 \ge \tilde Y^1$, having the same law as
  $\tilde Y$ in this assumption. Similarly as in
  Section~\ref{sec:comparison_perc}, for $(z,s)\in \L$,
  we define the oriented percolation process
  \begin{equation*}
    \tilde Z^{(z,s)}_t(z')
    \coloneq \ind\big\{ \text{there is an open path in $\tilde Y^1$
        from $(z,s)$ to $(z',t)$}
      \big\},
    \qquad (z',t)\in \L, \, t> s,
  \end{equation*}
  and the corresponding extinction time
  \begin{equation*}
    \begin{split}
      \tilde \tau_\ext^{(z,s)} &\coloneq
      \inf \big\{t \in \L_\ttime : t>s, \tilde Z_t^{(z,s)}(\cdot)\equiv 0\big\}.
    \end{split}
  \end{equation*}
  We then set
  \begin{align*}
    S_j & \coloneq \tilde \tau_\ext^{(Z_j, T_j)}.
  \end{align*}
  With $\Delta_j  \coloneq M_{T_j}^{1}-M_{T_j}^{2}$, we then couple $\eta^1$ and
  $\eta^2$ on the interval $(T_j,S_j]$ by setting
  $U^2(x,T_j+n) = U^1(x+\Delta_j, T_j+n)$ for every $x\in \Z$ and
  $1\le n \le S_j-T_j$. That is, for such $n$ we iteratively define
  \begin{equation}
    \label{eqn:shiftedcoupling}
    \begin{split}
      \eta_{T_j+n}^1(x)
      &= F\big(\eta_{T_j+n-1}^1, U^1(x,T_j+n)\big),
      \\\eta_{T_j+n}^2(x)
      &= F\big(\eta_{T_j+n-1}^2, U^1(x+\Delta_j,T_j+n)\big).
    \end{split}
  \end{equation}
  If $T_j=\infty$, we set $S_j=\infty$.
  \smallskip

  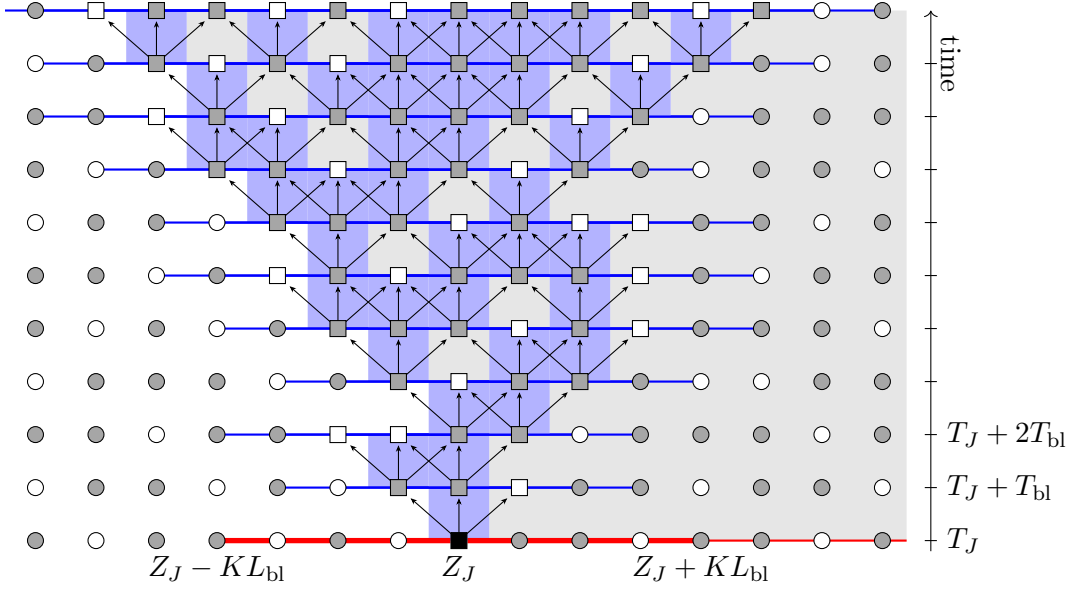
\begin{figure}[t]
    \centering
      \begin{tikzpicture}[xscale=0.8,yscale=0.7]
    %\useasboundingbox (-8.5, -0.5) rectangle (11.5, 10.5);
    \newcommand{\agreeline}[1]{%
      \begin{pgfonlayer}{background}
        \begin{scope}
          \clip (-7.5,-0.5) rectangle (7.5,10.5);
          \draw[blue, thick] (#1.center) -- +(2,0);
          \draw[blue, thick] (#1.center) -- +(-2,0);
        \end{scope}
      \end{pgfonlayer}%
    }
    \newcommand{\agreebox}[1]{%
      \begin{pgfonlayer}{background}
        \begin{scope}
          \clip (-7.5,-0.5) rectangle (7.5,10.5);
          \fill[blue!30] (#1.center) -- +(0.5,0) -- +(0.5,1) -- +(-0.5,1) -- +(-0.5,0);
        \end{scope}
      \end{pgfonlayer}%
    }
    % Styles:
    %  edot    = white circle  (unreachable, white)
    %  gdot    = grey circle   (unreachable, grey)
    %  esquare = white square  (reachable terminal, white)
    %  gsquare = grey square   (reachable, grey — path passes through)
    %  origin  = grey square   (the source (0,0))
    %  coupled = grey rectangle in background

    \tikzset{
      arr/.style     = {-{Stealth[length=2.5pt,width=2pt]}, thin },
      ultrathick/.style = {line width=2pt},
      edot/.style    = {circle,    draw=black, fill=white,   inner sep=0pt, minimum size=6pt},
      gdot/.style    = {circle,    draw=black, fill=gray!70, inner sep=0pt, minimum size=6pt},
      esquare/.style = {
        rectangle, draw=black, fill=white,   inner sep=0pt, minimum size=6pt,
        append after command={
          \pgfextra{
            \agreeline{\tikzlastnode};
          }
        }
      },
      gsquare/.style = {
        rectangle, draw=black, fill=gray!70, inner sep=0pt, minimum size=6pt,
        append after command={
          \pgfextra{
            \draw[arr] (\tikzlastnode.center) -- +(-0.8,0.8);
            \draw[arr] (\tikzlastnode.center) -- +(0,0.8);
            \draw[arr] (\tikzlastnode.center) -- +(0.8,0.8);
            \agreebox{\tikzlastnode};
            \agreeline{\tikzlastnode};
          }
        }
      },
      origin/.style  = {
        rectangle, draw=black, fill=black, inner sep=0pt, minimum size=6pt,
        append after command={
          \pgfextra{
            \draw[arr] (\tikzlastnode.center) -- +(-0.8,0.8);
            \draw[arr] (\tikzlastnode.center) -- +(0.8,0.8);
            \draw[arr] (\tikzlastnode.center) -- +(0,0.8);
            \agreebox{\tikzlastnode};
            \begin{pgfonlayer}{background}
              \draw[red,ultrathick] (\tikzlastnode) -- +(4,0);
              \draw[red,ultrathick] (\tikzlastnode) -- +(-4,0);
            \end{pgfonlayer}
          }
        }
      },
      coupled/.style = {
        append after command={
          \pgfextra{
            \fill[black!10] (\tikzlastnode.center) -- +(0.5,0) -- +(0.5,1) -- +(-0.5,1) -- +(-0.5,0);
          }
        }
      }
    }

    %coupled region to the right and in the holes
    \begin{pgfonlayer}{background}
      \fill[black!10] (0.5,0) -- (7.4,0) -- (7.4,10.0) -- (4.5,10.0) -- (4.5,9)
      -- (3.5,9) -- (3.5,8) -- (2.5,8) -- (2.5,7) -- (1.5,7) -- (1.5,6)
      -- (2.5,6) -- (2.5,3) -- (1.5,3) -- (1.5,2) -- (0.5,2);

      \foreach \x in {(0,3), (1,4), (-1,5), (0,6), (-2,7), (1,7),
        (-3,8), (2,8), (-4,9), (-2,9), (3,9)}
      {
        \node[coupled] at \x {};
      }
      \draw[red, thick] (4,0)--(7.4,0);
    \end{pgfonlayer}

    % y=0
    \foreach \x/\f in {-7/1,-6/0,-5/1,-4/1,-3/0,-2/1,-1/0,0/4,1/1,2/1,3/0,4/1,5/1,6/0,7/1}
    { \ifnum\f=0 \node[edot] at (\x,0) {};
      \else\ifnum\f=1 \node[gdot] at (\x,0) {};
      \else\ifnum\f=2 \node[gsquare] at (\x,0) {};
      \else\ifnum\f=3 \node[esquare] at (\x,0) {};
      \else \node[origin] at (\x,0) {};
      \fi\fi\fi\fi }

    % y=1
    \foreach \x/\f in {-7/0,-6/1,-5/1,-4/0,-3/1,-2/0,-1/2,0/2,1/3,2/1,3/1,4/0,5/1,6/1,7/0}
    { \ifnum\f=0 \node[edot] at (\x,1) {};
      \else\ifnum\f=1 \node[gdot] at (\x,1) {};
      \else\ifnum\f=2 \node[gsquare] at (\x,1) {};
      \else\ifnum\f=3 \node[esquare] at (\x,1) {};
      \else \node[origin] at (\x,1) {};
      \fi\fi\fi\fi }

    % y=2
    \foreach \x/\f in {-7/1,-6/1,-5/0,-4/1,-3/1,-2/3,-1/3,0/2,1/2,2/0,3/1,4/1,5/1,6/0,7/1}
    { \ifnum\f=0 \node[edot] at (\x,2) {};
      \else\ifnum\f=1 \node[gdot] at (\x,2) {};
      \else\ifnum\f=2 \node[gsquare] at (\x,2) {};
      \else\ifnum\f=3 \node[esquare] at (\x,2) {};
      \else \node[origin] at (\x,2) {};
      \fi\fi\fi\fi }

    % y=3
    \foreach \x/\f in {-7/0,-6/1,-5/1,-4/1,-3/0,-2/1,-1/2,0/3,1/2,2/2,3/1,4/0,5/0,6/1,7/1}
    { \ifnum\f=0 \node[edot] at (\x,3) {};
      \else\ifnum\f=1 \node[gdot] at (\x,3) {};
      \else\ifnum\f=2 \node[gsquare] at (\x,3) {};
      \else\ifnum\f=3 \node[esquare] at (\x,3) {};
      \else \node[origin] at (\x,3) {};
      \fi\fi\fi\fi }

    % y=4
    \foreach \x/\f in {-7/1,-6/0,-5/1,-4/0,-3/1,-2/2,-1/2,0/2,1/3,2/2,3/3,4/1,5/1,6/1,7/0}
    { \ifnum\f=0 \node[edot] at (\x,4) {};
      \else\ifnum\f=1 \node[gdot] at (\x,4) {};
      \else\ifnum\f=2 \node[gsquare] at (\x,4) {};
      \else\ifnum\f=3 \node[esquare] at (\x,4) {};
      \else \node[origin] at (\x,4) {};
      \fi\fi\fi\fi }

    % y=5
    \foreach \x/\f in {-7/1,-6/1,-5/0,-4/1,-3/3,-2/2,-1/3,0/2,1/2,2/2,3/3,4/1,5/0,6/1,7/1}
    { \ifnum\f=0 \node[edot] at (\x,5) {};
      \else\ifnum\f=1 \node[gdot] at (\x,5) {};
      \else\ifnum\f=2 \node[gsquare] at (\x,5) {};
      \else\ifnum\f=3 \node[esquare] at (\x,5) {};
      \else \node[origin] at (\x,5) {};
      \fi\fi\fi\fi }

    % y=6
    \foreach \x/\f in {-7/0,-6/1,-5/1,-4/0,-3/2,-2/2,-1/2,0/3,1/2,2/3,3/3,4/1,5/1,6/0,7/1}
    { \ifnum\f=0 \node[edot] at (\x,6) {};
      \else\ifnum\f=1 \node[gdot] at (\x,6) {};
      \else\ifnum\f=2 \node[gsquare] at (\x,6) {};
      \else\ifnum\f=3 \node[esquare] at (\x,6) {};
      \else \node[origin] at (\x,6) {};
      \fi\fi\fi\fi }

    % y=7
    \foreach \x/\f in {-7/1,-6/0,-5/1,-4/2,-3/2,-2/3,-1/2,0/2,1/3,2/2,3/1,4/0,5/1,6/1,7/0}
    { \ifnum\f=0 \node[edot] at (\x,7) {};
      \else\ifnum\f=1 \node[gdot] at (\x,7) {};
      \else\ifnum\f=2 \node[gsquare] at (\x,7) {};
      \else\ifnum\f=3 \node[esquare] at (\x,7) {};
      \else \node[origin] at (\x,7) {};
      \fi\fi\fi\fi }

    % y=8
    \foreach \x/\f in {-7/1,-6/1,-5/3,-4/2,-3/3,-2/2,-1/2,0/2,1/2,2/3,3/2,4/0,5/1,6/1,7/1}
    { \ifnum\f=0 \node[edot] at (\x,8) {};
      \else\ifnum\f=1 \node[gdot] at (\x,8) {};
      \else\ifnum\f=2 \node[gsquare] at (\x,8) {};
      \else\ifnum\f=3 \node[esquare] at (\x,8) {};
      \else \node[origin] at (\x,8) {};
      \fi\fi\fi\fi }

    % y=9
    \foreach \x/\f in {-7/0,-6/1,-5/2,-4/3,-3/2,-2/3,-1/2,0/2,1/2,2/2,3/3,4/2,5/1,6/0,7/1}
    { \ifnum\f=0 \node[edot] at (\x,9) {};
      \else\ifnum\f=1 \node[gdot] at (\x,9) {};
      \else\ifnum\f=2 \node[gsquare] at (\x,9) {};
      \else\ifnum\f=3 \node[esquare] at (\x,9) {};
      \else \node[origin] at (\x,9) {};
      \fi\fi\fi\fi }

    % Redefine square styles without arrows for the top row
    \tikzset{
      gsquare/.style = {
        rectangle, draw=black, fill=gray!70, inner sep=0pt, minimum size=6pt,
        append after command={
          \pgfextra{
            \agreeline{\tikzlastnode};
          }
        }
      }
    }

    % y=10
    \foreach \x/\f in {-7/1,-6/3,-5/2,-4/2,-3/3,-2/2,-1/3,0/2,1/2,2/2,3/2,4/3,5/2,6/0,7/1}
    { \ifnum\f=0 \node[edot] at (\x,10) {};
      \else\ifnum\f=1 \node[gdot] at (\x,10) {};
      \else\ifnum\f=2 \node[gsquare] at (\x,10) {};
      \else\ifnum\f=3 \node[esquare] at (\x,10) {};
      \else \node[origin] at (\x,10) {};
      \fi\fi\fi\fi }

    \draw[->] (7.8, -0.2) -- (7.8, 10);
    \node[rotate=-90, anchor=south] at (7.8, 9) {time};
    \draw (7.7,0) -- (7.9,0) node[right] {$T_J$};
    \draw (7.7,1) -- (7.9,1) node[right] {$T_J+T_\bl$};
    \draw (7.7,2) -- (7.9,2) node[right] {$T_J+2T_\bl$};
    \foreach \x in {3, ..., 9} {\draw (7.7,\x) -- (7.9,\x);}

    \node[below] at (0,-.1) {$Z_J$};
    \node[below] at (-4,-.1) {$Z_J-KL_\bl$};
    \node[below] at (4,-.1) {$Z_J+KL_\bl$};
  \end{tikzpicture}

    \caption{Illustration for \eqref{eqn:orderedforever}: The black square is
      located at $(Z_J,T_J)$. Grey (and black) dots and squares are
      positioned on the centre of the bottom face of good blocks, white dots
      and squares correspond to bad blocks. Squares are points for which
      $\tilde Z^{(Z_J,T_J)}=1$, that is they are reachable by an open path
      from the black square. Due to the construction of $T_J$, the
      configurations $\eta^1$ and $\eta^2$ (shifted by their maxima) agree at
      the bottom of $\Block(Z_J,T_J)$ (thick red line) and to the right of it
      (thin red line). By Assumption~\ref{ass:coarse}(c), they then agree on
      blue horizontal lines, by Assumption~\ref{ass:coarse}(d) also on the blue
      rectangles, and by the coupling construction also on grey blocks.  On
      the figure $K=4$.}
    \label{fig:propagation}
  \end{figure}

  We define
  \begin{equation*}
    J  \coloneq \inf \{ j \in \N : S_j = \infty \},
  \end{equation*}
  and say that the coupling is successful when both $J$ and $T_J$ are
  finite. In this case $\tilde \tau_\ext^{(Z_J,T_J)}$ is infinite and
  $\eta^1$ and $\eta^2$ are coupled by \eqref{eqn:shiftedcoupling} for all
  times larger than $T_J$. If the coupling is unsuccessful, then
  eventually, a.s., $\eta^1$ and $\eta^2$ run independently.

  The following observation is the key ingredient of the proof:
  \begin{equation}
    \label{eqn:orderedforever}
    \parbox{11cm}{If the coupling is successful, then $M^1_t-M^2_t =
      \Delta_J$ for all $t\ge T_J$.}
  \end{equation}
  In order to follow the proof of this observation, we recommend that the reader
  look at Figure~\ref{fig:propagation} first. By definition of $T_J$,
  $\eta^1_{T_J}\in G_\loc^\eta(Z_J)$ and the oriented percolation cluster
  described by $\tilde Z^{(Z_J,T_J)}$ never dies. In particular, this means
  that $\Block(Z_J,T_J)$ is good (the black square on the figure). By
  Assumption~\ref{ass:coarse}(b,c), for every $y\in\{-L_\bl,0,L_\bl\}$, it
  holds that $\eta^1_{T_J+T_\bl}\in G_\loc^\eta(Z_J+y)$ (illustrated by the
    square nodes) and the value of $\eta_{T_J+T_\bl}$ in $B_{2L_\bl}(Z_J+y)$
  only depends on $\{U(z,t):(z,t)\in \Block(Z_J,T_J)\}$ and not on
  $\eta^1_{T_J}$ (the balls $B_{2L_\bl}(Z_J+y)$ are illustrated by the blue
    line at level $T_J+T_\bl$). Moreover, by  Assumption~\ref{ass:coarse}(d),
  any configuration agreeing with $\eta^1_{T_J}$ in $B_{2L_\bl}(Z_J)$ will
  agree with $\eta^1$ for all $x\in B_{L_\bl/2}$ and $t\in [T_J, T_J+T_\bl]$
  (the lowest blue rectangle). Applying this argument iteratively means that
  the values
  \begin{equation*}
    \{\eta^1_t(z): z\in B_{2L_\bl}(x)
      \text{ with } \tilde Z^{(Z_J,T_J)}(x,t)=1\}
  \end{equation*}
  (the union of all blue lines on the figure) are determined by
  \begin{equation*}
    \{U^1(z,s): (z,s) \in \Block(x,t) \text{ which is good and }
      \tilde Z^{(Z_J,T_J)}(x,t)=1 \},
  \end{equation*}
  and every configuration agreeing with $\eta^1_{T_J}$ in $B_{K L_\bl}(Z_J)$
  will agree with it on the union of the blue rectangles (and the blue lines)
  in the figure. Since, by construction of the coupling,
  $\eta^1_{T_J} = \Delta_J + \eta^2_{T_J}$ on $B_{K L_\bl}(Z_J)$ and
  $U^1 = \Delta_J + U^2$ for all times  larger than $T_J$, by the flow
  construction from Assumption~\ref{ass:flow}, this means that
  $\eta^1 = \Delta_J + \eta^2$ on the union of the blue rectangles.

  It remains to consider the sites that are to the right of the cluster
  $\{\tilde Z^{(Z_J,T_J)} =1\}$ (light grey region on the figure). On the
  bottom of this region $\eta^1$ and $\Delta_J + \eta^2$ agree, by
  construction of $T_J$. By Assumption~\ref{ass:flow}, $\eta^1_t(x)$ depends
  only on $\eta^1_{t-1}(y)$ and $U^1(y,t)$ for $\abs{y-x}\le R$. Since the
  width of the blue rectangles is $L_\bl\ge 2R$ (by \eqref{eqn:L}), iterating
  this property, using again that $U^1 = \Delta_J + U^2$ for all times larger
  than $T_J$, one proves easily that $\eta^1 = \Delta_J + \eta^2$ also in the
  light grey region on Figure~\ref{fig:propagation}. Since the right-most
  particle in $\eta^1$ must be in the grey or blue region,
  \eqref{eqn:orderedforever} directly follows.

  We proceed by bounding the time $T_J$. First note that if $T_j$ is finite,
  then $\tilde \tau_\ext^{(Z_j,T_j)}$ has the same law as $\tilde \tau_\ext^0$,
  and thus
  \(
    P_{\eta_0^1, \eta_0^2}(S_j=\infty\mid T_j <\infty)
    =\P(\tilde\tau_\ext^0 = \infty) \eqqcolon q>0
  \)
  by Lemma~\ref{lem:percolation}\eqref{perc:extinction}. Therefore,
  \begin{equation}
    \label{eqn:Jlarge}
    P_{\eta_0^1, \eta_0^2} (J > k )
    \le P_{\eta_0^1, \eta_0^2}(T_j<\infty, S_j-T_j<\infty
      \text{ for all } 1\le j \le k)
    \le (1-q)^k,
  \end{equation}
  in particular $J$ is a.s.~finite. Further, by
  Lemma~\ref{lem:percolation}\eqref{perc:die_late}, for every $j\in \N$,
  $k \in \L_\ttime$,
  \begin{equation}
    \label{eqn:b2}
    P_{\eta_0^1, \eta_0^2}  ( k  \le S_j-T_j < \infty \mid T_j<\infty )
    \le \tilde C e^{-\tilde \gamma  k } .
  \end{equation}
  By Assumption~\ref{ass:coupling}, in the $k$-th sub-step of Step 1, the
  probability that $\mathcal Z_{kt_5}\neq\emptyset$ is at least $\varepsilon$
  (if both processes are still alive). Therefore,
  \begin{equation}
    \label{eqn:b1}
    P_{\eta_0^1, \eta_0^2} (kt_5< T_j-S_{j-1} <\infty\mid S_{j-1}<\infty)
    \le (1-\varepsilon)^{k}.
  \end{equation}
  In particular, the event $T_j-S_{j-1}=\infty$ a.s.~occurs only if one of
  the processes dies out. Since we assume that the initial conditions
  $\eta^1_0$ and $\eta^2_0$ contain of order $\rho$ well-started blocks,
  we obtain from the second claim of Lemma~\ref{lem:spread_goodconf} that
  \begin{equation}
    \label{eqn:b1a}
    P_{\eta_0^1, \eta_0^2} (\exists j\in \N: S_{j-1}<\infty, T_j=\infty)
    \le P_{\eta_0^1, \eta_0^2} (\tau_\ext^1<\infty
      \text{ or } \tau_\ext^2<\infty)
    \le 2 Ce^{-c \rho}.
  \end{equation}
  This implies that the probability that the coupling is unsuccessful is at
  most $2Ce^{-c\rho}$. Next, we observe that
  $T_J = (T_J-S_{J-1}) + \sum_{j=1}^{J-1} (S_j-T_j)+(T_j-S_{j-1})$, and thus
  \begin{equation}
    \begin{split}
      &\{\infty > T_J > k^2, J \le k\}
      \\&\subseteq \bigcup_{j=1}^k
      \Big\{\infty> T_j - S_{j-1}> \frac{k}{2} , S_{j-1}<\infty \Big\}
      \cup \Big\{ S_j - T_j > \frac{k}{2}, T_j<\infty \Big\}.
    \end{split}
  \end{equation}
  From this and~\eqref{eqn:b2}, \eqref{eqn:b1} we deduce that
  \begin{align}\label{eqn:b11}
    & P_{\eta_0^1, \eta_0^2} (\infty> T_{J} > k^2, J \le k
      ) \le k (1-\varepsilon)^{k/(2t_5)} + k \tilde C e^{-\tilde \gamma k /2}.
  \end{align}

  By the finite speed of spreading (the last part of
    Assumption~\ref{ass:flow}), since $\eta_0^2(x) = 0$ for all $x\ge -\ell$
  by assumption, for every $n\ge 0$,
  \begin{equation}
    \label{eqn:m2pos}
    M^2_n \le -\ell + n R.
  \end{equation}
  By applying Lemma~\ref{lem:spread_goodconf} with $\ell=\rho$, $z=0$ and
  $a=p_\nu V_\rho/2$ on process $\eta^1$, for every
  $n\ge (2\rho+1)T_\bl/\tilde v$ with $n\in \L_\ttime$,
  \begin{equation}
    P_{\eta_0^1, \eta_0^2}(M^1_n\le - \rho L_\bl) \le C e^{-c \rho}.
  \end{equation}
  On the event $M^1_n > - \rho L_\bl$, $M^1_k > -\rho L_\bl - n R$ for all
  $0\le k \le n$, again by the finite speed of spreading. Therefore,
  \begin{equation}
    \label{eqn:m1pos}
    P_{\eta_0^1, \eta_0^2}\Big(
      \bigcup_{k=1}^n \{M^1_k\le - \rho L_\bl - nR\}\Big)
    \le C e^{-c \rho}.
  \end{equation}
  Inequalities \eqref{eqn:m2pos} and \eqref{eqn:m1pos} together imply that
  for $n\ge (2\rho +1)T_\bl/\tilde v$ with $n\in \L_\ttime$,
  \begin{equation}
    \label{eqn:m12diff}
    P_{\eta_0^1, \eta_0^2}\Big(
      \bigcup_{k=0}^{n}\{M^1_k-M^2_k \le -\rho L_\bl +\ell - 2nR\}\Big)
    \le C e^{-c\rho}.
  \end{equation}

  We can now prove the claim of the proposition. To this end we fix
  $n = \bbfloor{(\ell-\rho L_\bl)/(2R)}$ (so that
    $-\rho L_\bl + \ell - 2n R$  appearing in \eqref{eqn:m12diff} is
    non-negative). By the assumption of the proposition this $n$ satisfies the
  condition required for \eqref{eqn:m12diff} to hold. If $T_J\le n$ (which
    also implies that the coupling is successful) and the
  complement of the event in \eqref{eqn:m12diff} occurs, then by
  \eqref{eqn:orderedforever} we have $M^1_k \ge M^2_k$ for all $k\ge 0$.
  Hence,
  \begin{equation*}
    \begin{split}
      P_{\eta_0^1, \eta_0^2}\big(\exists k\ge 0: M^1_k < M^2_k\big)
      \le{} &
      P_{\eta_0^1, \eta_0^2} \big(J > \sqrt n \big)
      + P_{\eta_0^1, \eta_0^2} (\infty > T_J > n, J \le \sqrt n )
      \\&+ P_{\eta_0^1, \eta_0^2}
      \big(\exists j\in \mathbb N: S_{j-1}<\infty, T_j =\infty\big)
      \\&+ P_{\eta_0^1, \eta_0^2}
      \Big(\bigcup_{k=0}^{n}\{M^1_k-M^2_k \le -\rho L_\bl +\ell - 2nR\}\Big).
    \end{split}
  \end{equation*}
  The four probabilities on the right-hand side can be bounded by
  \eqref{eqn:Jlarge}, \eqref{eqn:b11}, \eqref{eqn:b1a}, \eqref{eqn:m12diff},
  respectively. These bounds easily imply the claim of the proposition.
\end{proof}

\subsection{Delay to reach good starting configurations}
\label{sec:delaybounds}

In this section, we bound the time it takes $\eta^1, \eta^2$ conditioned to
survive and started from one particle at the origin and one particle at $x$,
respectively, to reach the starting configurations needed to use
the shifted coupling constructed in Proposition~\ref{prop:couplefronts}.

We emphasise that in the remainder of this section we mostly work
with the processes conditioned on survival, that is we work under
$\bar \P_{\eta_0}$. However, by Remark~\ref{rem:survival}, the survival
probability is bounded from below uniformly over all non-trivial initial
conditions. Therefore many bounds that were proved under $\P_{\eta_0}$ also
trivially hold under $\bar \P_{\eta_0}$, up to a deterministic constant factor.
This will repeatedly be used without further mention. We also point out that the
process $(\eta_n)_{n\ge 0}$ is a (strong) Markov process also under the
conditioned measure $\bar \P_{\eta_0}$.

To state the result, for  $x, \rho, \ell \in \N$ we set
\begin{equation}
  x_{\rho,\ell} \coloneq x+ \ell + R \bbceil{( \Cr{c:nwalb}^{-1} \rho)^2},
\end{equation}
and define the sets of configurations satisfying the assumptions of
Proposition~\ref{prop:couplefronts}, shifted to the right by $x_{\rho,\ell}$
\begin{equation}
  \label{eqn:Thetas}
  \begin{split}
    \Theta_x^1 &= \Theta_x^1(\rho,\ell)
    \coloneq\big\{ \eta \in E^\Z :
      \abs{W(\eta) \cap B_\rho^\bl(\bbceil{ x_{\rho,\ell}  })}
      \ge p_\nu V_\rho/2 \big\},
    \\ \Theta_x^2 &= \Theta_x^2(\rho)
    \coloneq\big\{ \eta\in E^\Z :\abs{ W(\eta)}  \ge p_\nu V_\rho/2, \
      \eta(y)=0 \ \forall y > x_{\rho,\ell} - \ell
      \big\}.
  \end{split}
\end{equation}

\begin{lemma}
  \label{lem:delaybounds}
  There are constants $C,c \in (0,\infty)$
  such that:
  \begin{enumerate}
    \item For every $x>0$, $\rho < \varepsilon  \sqrt{x}$ (with
      $\varepsilon $ as in Proposition~\ref{prop:hitandblock}), $\ell > 0$,
    and every starting configuration $\zeta$ which is non-zero only at the
    origin
    \begin{equation*}
      \bar \P_\zeta\big(\eta_{{\tau_\hit(x)+ \bbceil{C(\rho^2+\ell)}}}
        \notin \Theta^1_x(\rho ,\ell)\big)
      \le C(1+x^2)e^{-c \rho } + Ce^{-c\sqrt x}.
    \end{equation*}

    \item
    For every $\rho \ge 0$, and every starting configuration $\zeta$ which is
    non-zero only at $x$
    \begin{equation*}
      \bar \P_\zeta\big( \eta_{\bbceil{(\Cr{c:nwalb}^{-1} \rho )^2}} \not\in
        \Theta_x^2(\rho) \big) \le C e^{-c \rho}
    \end{equation*}
    with $\Cr{c:nwalb}$ from Lemma~\ref{lem:neverwalkalone}.
  \end{enumerate}
\end{lemma}

\begin{proof}
  We first show (b). Set $n \coloneq \bbceil{(\Cr{c:nwalb}^{-1} \rho )^2}
  \in \L_\ttime$ and note that $\Cr{c:nwalb} \sqrt n \ge \rho$. By
  Lemma~\ref{lem:blocks_somewhere} (applied after translating the origin
    to $x$),
  \begin{equation*}
    \bar\P_\zeta\big( D_{\Cr{c:nwalb} \sqrt n}(x, \eta_{n} ) <
      \infty\big)
    \ge 1-C e^{-c \rho}.
  \end{equation*}
  On the event $\{D_{\Cr{c:nwalb} \sqrt n}(x,\eta_{n}) < \infty\}$, the
  configuration $\eta_n$ contains a ball (at the block scale) of radius
  $\Cr{c:nwalb} \sqrt n \ge \rho$ with at least
  $p_\nu V_{\Cr{c:nwalb}\sqrt n}/2 \ge p_\nu V_\rho/2$ well-started blocks,
  so in particular $\abs{W(\eta_{n})} \ge p_\nu V_\rho/2$.
  Moreover, since $x_{\rho,\ell}-\ell = x + Rn$, the finite speed of spreading
  gives that
  \(
    \{\eta_{n} (y)= 0 \, \forall y > x_{\rho,\ell}-\ell\}
  \)
  always holds true under $\bar \P_\zeta$, and claim (b) follows.

  To prove (a), let $\zeta$ be any configuration which is non-zero only at
  the origin. We first remark that $\tau_\hit(x)<\infty$,
  $\bar \P_\zeta$-a.s. This can easily be seen from
  Lemma~\ref{lem:blocks_somewhere} (which allows us to create arbitrarily
    many well-started blocks with an arbitrarily large
    $\bar \P_\zeta$-probability), Lemma~\ref{lem:shift_blocks} (which proves
    that with an arbitrarily large $\P_\zeta$-probability, and thus also
    $\bar\P_\zeta$-probability, when conditioned on having many well-started
    blocks at some time, $x$ will be eventually in the coupled region), and
  Lemma~\ref{lem:incoupled} (which states that $x$ is eventually hit, when in
    coupled region).

  Define the event
  $A_{x,\rho } \coloneq \{D_\rho(x,\eta_{ \tau_\hit(x)}) \le \rho^2\}$.
  By Proposition~\ref{prop:hitandblock} and the remark from the last paragraph,
  there are constants $c,C \in (0,\infty)$, and $\varepsilon\in (0,1)$ such that
  for every $\rho < \varepsilon \sqrt{x}$
  \begin{equation*}
    \bar{\P}_\zeta( A_{x,\rho }^c )
    \le C (e^{-c \sqrt x} + e^{-c \rho + 2 \log x}).
  \end{equation*}
  On the event $A_{x,\rho }$, we have
  \(
    \abs{ W(\eta_{ \tau_\hit(x)}) \cap B_{\rho+\rho^2}^\bl(\bbracket x)}
    \ge p_\nu V_\rho/2
  \).
  Therefore, by Lemma~\ref{lem:spread_goodconf} applied with
  $\ell=\rho+\rho^2$, $\rho=\rho$, $a= p_\nu V_\rho/2$,
  $z=\bbceil {x_{\rho,\ell}} - \bbracket x$, and
  $n = \bbceil{C(\rho^2 + \ell)}$ (with
    $C$ large, so that $n$ satisfies the last assumption of this lemma),
  using also the translation invariance (shift by $\bbracket x$) and the
  strong Markov property,
  \begin{equation*}
    \bar \P_\zeta( \eta_{\tau_\hit(x)+n} \notin \Theta_x^1 \mid A_{x,\rho })
    \le C e^{-c \rho} + C\rho e^{-c\rho} \le Ce^{-c\rho }.
  \end{equation*}
  Combining the last two displays proves (a).
\end{proof}

\subsection{Shifted coupling conditioned on survival}
\label{sec:shiftandsurvive}

The shifted coupling from Proposition~\ref{prop:couplefronts} works under the
unconditioned measures, but it is easy to change it so that it works for
conditioned measures. Recall the sets $\Theta^1_x, \Theta^2_x$ from
\eqref{eqn:Thetas}, and $M_n^1, M_n^2$ from \eqref{eqn:Maxdef}.

\begin{lemma}
  \label{lem:shiftedconditioned}
  There are constants $c,C\in (0,\infty)$ so that for every
  $x,\rho,\ell \in \N$, $\eta^1_0\in \Theta_x^1(\rho,\ell)$ and
  $\eta^2_0\in \Theta_x^2(\rho)$ there exists a coupling
  $\bar P_{\eta^1_0,\eta^2_0}$ of two processes $\eta^1$, $\eta^2$
  respectively distributed according to $\bar \P_{\eta^1_0}$ and
  $\bar \P_{\eta^2_0}$, such that
  \begin{equation*}
    \bar P_{\eta^1_0,\eta^2_0}(M_k^1 \ge M_k^2 \text{ for all }k\ge 0)
    \ge 1 -  C e^{-c \rho}
    - C \sqrt{\ell-\rho L_\bl}\,e^{ -c \sqrt{\ell -\rho L_\bl}}.
  \end{equation*}
\end{lemma}

\begin{proof}
  Every configuration in $\Theta^i_x$, $i=1,2$, contains at least
  $p_\nu V_\rho/2 \sim c \rho$ well-started blocks. Hence, by
  Lemma~\ref{lem:spread_goodconf},
  $\P_{\eta_0^i}(\tau_\ext = \infty) \ge 1-Ce^{-c\rho}$. In particular,
  for any $\eta_0 \in \Theta^i_x$, $i=1,2$,
  \begin{equation*}
    d_\tv(\bar \P_{\eta_0}, \P_{\eta_0}) \le Ce^{-c\rho},
  \end{equation*}
  where $d_\tv$ stands for the total variation distance. Therefore, there
  exists a coupling $P_{\eta_0}$ of two processes $\hat\eta$ and
  $\eta$ distributed according to  $\P_{\eta_0}$ and
  $\bar \P_{\eta_0}$, respectively, so that
  \begin{equation}
    P_{\eta_0}(\eta_n = \hat \eta_n \text { for all } n\ge 0)
    \ge 1-C e^{-c\rho }.
  \end{equation}
  The coupling $\bar P_{\eta^1_0,\eta^2_0}$ of $\eta^1, \eta^2$ required by
  the lemma is then constructed by chaining (or disintegration) of
  $P_{\eta^1_0}$, $P_{\eta^1_0,\eta^2_0}$ from
  Proposition~\ref{prop:couplefronts}, and $P_{\eta^2_0}$: For all measurable
  sets $A,B$,
  \begin{equation*}
    \bar P_{\eta^1_0,\eta^2_0}(\eta^1 \in A, \eta^2 \in B) =
    \int P_{\eta^1_0,\eta^2_0}(\D \hat \eta^1,  \D \hat \eta^2 )
    P_{\eta^1_0}(\eta \in A \mid \hat \eta^1)
    P_{\eta^2_0}(\eta \in B \mid \hat \eta^2).
  \end{equation*}
  It can be easily proved that
  $d_\tv(\bar P_{\eta^1_0,\eta^2_0}, P_{\eta^1_0,\eta^2_0}) \le 2Ce^{-c\rho }$.
  This and Proposition~\ref{prop:couplefronts} then directly imply the claim
  of the lemma.
\end{proof}

We now extend the coupling constructed in the last lemma slightly, in order to
give the second process a time to reach $\Theta_x^2$.

\begin{lemma}
  \label{lem:shiftedcondfromsinglepoint}
  The claim of Lemma~\ref{lem:shiftedconditioned} holds (with different
    constants $c,C$) if $\eta_0^1\in\Theta_x^1$ and $\eta_0^2$ is non-zero only
  at $x$.
\end{lemma}

\begin{proof}
  Let $t_0 = \bbceil{(\Cr{c:nwalb}^{-1}\rho )^2}$ (cf.
    Lemma~\ref{lem:delaybounds}(b)). To construct the required coupling, for
  the first $t_0$ time steps we let $\eta^1$ and $\eta^2$ run independently
  according to $\bar \P_{\eta_0^1}$ and $\bar \P_{\eta_0^2}$, respectively.
  If $\eta^1_{t_0} \notin \Theta_x^1$ or $\eta^2_{t_0} \notin \Theta_x^2$, the
  coupling fails, and we let the two processes run independently forever.
  By Lemma~\ref{lem:spread_goodconf} and Lemma~\ref{lem:delaybounds}(b),
  \begin{equation}
    \bar{\P}_{\eta^1_0}(\eta^1_{t_0} \notin  \Theta^1_x) \le Ce^{-c\rho }
    \quad \text{and}\quad
     \bar{\P}_{\eta^2_0}(\eta^2_{t_0} \notin  \Theta^2_x) \le Ce^{-c\rho },
  \end{equation}
  so this occurs with probability at most $Ce^{-c\rho }$. Note also that due
  to the definition of $x_{\rho ,\ell}$, we always have $M^1_k \ge M^2_k$ for
  all $k \le t_0$.

  On the other hand, if $\eta_{t_0}^1 \in \Theta_x^1$ and
  $\eta_{t_0}^2 \in \Theta_x^2$, after the time $t_0$, we can couple $\eta^1$
  and $\eta^2$ using  $\bar P_{\eta^1_{t_0}, \eta^2_{t_0}}$ constructed in
  Lemma~\ref{lem:shiftedconditioned}. The coupling constructed in these two
  steps has all claimed properties.
\end{proof}

\subsection{Proof of Theorem~\ref{thm:perturbed_subadditivity}}
\label{sec:proof}

We now have all ingredients to prove Theorem~\ref{thm:perturbed_subadditivity}.

\begin{proof}[Proof of Theorem~\ref{thm:perturbed_subadditivity}]
  Since the roles of $\bar \sigma_\hit(y)$ and $\bar \tau_\hit(x)$ are
  symmetric in the conditions (a)--(c) of the theorem, we can assume, without
  loss of generality, that $x \ge y$.

  We now construct a probability space $(\Omega ,\mathcal A, \Q)$ and the
  random variables $\bar \tau_\hit(x)$, $\bar \tau_\hit(x+y)$,
  $\bar \sigma_\hit(y)$ and $V(x,y)$ having the properties stated in the
  theorem. Let $\eta^1$ be a process defined on $(\Omega , \mathcal A, \Q)$
  distributed according to $\bar \P_\zeta $. We set
  \begin{align*}
    \bar \tau_\hit (x) &= \inf\{n\ge 0: \exists z\ge x \, \eta_n(z) \neq 0\},
    \\\bar \tau_\hit (x+y)
    &= \inf\{n\ge 0: \exists z\ge x+y\, \eta_n(z) \neq 0\}.
  \end{align*}
  These times have the distributions required in (a).

  We then fix $\rho = \floor{x^{1/6}}$,
  $\ell = \floor{x^{1/3}}$, $C$ as in Lemma~{\ref{lem:delaybounds}}, set
  $T = \bar \tau_\hit (x) + \bbceil{C(\rho^2+\ell)}$, and distinguish two cases:
  \begin{enumerate}[label=(\arabic*)]
    \item If
    $\eta^1_T \in \Theta_x^1(\rho ,\ell)$, we construct on
    $(\Omega , \mathcal A, \Q)$ another process $\eta^2$, so that the joint
    law of $\eta^1_{T+\cdot}$ and $\eta^2_\cdot$ is the coupling
    $\bar P_{\eta^1_T,\theta_x\zeta}$ constructed in
    Lemma~\ref{lem:shiftedcondfromsinglepoint}. We set
    $V(x,y) =\bbceil{C(\rho^2+\ell)}$ if $M^1_{T+k} \ge M^2_k$ for all
    $k\ge 0$ (by Lemma~\ref{lem:shiftedcondfromsinglepoint} this has
      probability at least $1- C e^{-cx^{1/6}}$). Otherwise we set
    $V(x,y) = \bar \tau_\hit(x+y)$.

    \item If $\eta^1_T \notin \Theta_x^1(\rho ,\ell)$, we let $\eta^2$ be
    $\bar \P_{\theta_x\zeta}$-distributed, independent of $\eta^1$. We set
    $V(x,y) = \bar \tau_\hit(x+y)$.
  \end{enumerate}
  In both cases, we set
  $\bar \sigma_\hit(y) = \inf\{n\ge 0: \exists z\ge x+y\, \eta^2_n(z) \neq 0\}$.

  Let $\mathcal F_T = \sigma (\eta^1_k: k\le T)$. By construction,
  $\Q$-a.s.~under the conditional law $\Q(\cdot \mid \mathcal F_T)$, $\eta^2$
  is $\bar \P_{\theta_x \zeta }$-distributed. Hence, due to the translation
  invariance, $\bar \sigma_\hit(y)$ has the required distribution, and
  $\bar \sigma_\hit(y)$ is independent of $\mathcal F_T$. In particular,
  $\bar \sigma_\hit(y)$ is independent of $\bar \tau_\hit(x)(\le T)$ as
  required in (a).

  We now check (b). In the cases when $V(x,y) = \bar \tau_\hit(x+y)$ there is
  nothing to show. In the only remaining case we have $M^1_{T+k} \ge M^2_k$ for
  all $k\ge 0$. This implies that
  \begin{equation*}
    \bar \sigma_\hit(y)
    \ge \bar \tau_\hit(x+y) - T
    = \bar \tau_\hit(x+y)
    - \big(\bar \tau_\hit(x) + \bbceil{C(\rho^2+\ell)}\big)
    = \bar \tau_\hit(x+y) - \bar \tau_\hit(x) - V(x,y),
  \end{equation*}
  and (b) follows.

  It remains to check (c). Let $A$ be the event
  $\{V(x,y) = \bbceil{C(\rho^2+\ell)}\}$. By Lemma~\ref{lem:delaybounds}(a),
  $\Q(\eta^1_T\in \Theta_x^1(\rho ,\ell)) \ge 1 - C e^{-c x^{1/6}}$. Together
  with the observation in point (1), this implies that
  $\Q(A^c) \le  C e^{-c x^{1/6}}$. Hence,
  \begin{equation*}
    \begin{split}
      E^{\Q}\big[V(x,y)^2\big]
      &= E^{\Q}\big[V(x,y)^2 \ind_A \big] + E^{\Q}\big[V(x,y)^2 \ind_{A^c} \big]
      \\&\le C x^{2/3} + E^{\Q}\big[\bar \tau_\hit(x+y)^2 \ind_{A^c} \big]
      \\&\le C x^{2/3} + E^{\Q}\big[ \tau_\hit (x+y)^4 \big]^{1/2} \Q(A^c)^{1/2}
    \end{split}
  \end{equation*}
  By Corollary~\ref{cor:hit_moments},
  $E^{\Q}\big[\tau_\hit (x+y)^4 \big] \le C(x+y)^4$, and (c) follows.
\end{proof}

\section{Proof of Theorem~\ref{thm:shape1}}
\label{sec:proof_speed}

When the initial condition $\zeta$ is non-zero only at the origin,
Theorem~\ref{thm:shape1} is a straightforward consequence of
Theorem~\ref{thm:perturbed_subadditivity}: Using the arguments of
\cite[Lemma on p.~674]{Hammersley:AOP2} (whose assumptions directly follow
  from Theorem~\ref{thm:perturbed_subadditivity} and
  Corollary~\ref{cor:hit_moments}),
there exists $\gamma \in (0,\infty)$ such that
\begin{equation*}
  \lim_{k\to\infty} \frac{\tau_\hit(m2^k)}{m 2^k}
  = \gamma, \qquad \bar{\P}_\zeta\text{-a.s.}
\end{equation*}
for all $m\in \N$. Together with the fact that the sequence
$x\mapsto \tau_\hit(x)$ is monotone by definition, this implies that (as
  explained in \cite[Remark 2, p.~675]{Hammersley:AOP2}),
\begin{equation*}
  \lim_{x\to\infty} \frac{\tau_\hit(x)}{x}
  = \gamma, \qquad \bar{\P}_\zeta\text{-a.s.}
\end{equation*}
This proves the claim of the theorem for this initial condition, with
$v = 1/\gamma $. Note however, that $\gamma $ and thus also $v$ might depend
on the value of $\zeta(0)$.

It remains to show that $\gamma$ is independent of $\zeta(0)$ (if the initial
  condition contains only one particle at $0$) and that the theorem holds for
initial conditions $\zeta$ that are non-zero at the origin and may contain
arbitrarily many particles on the negative half-line. The argument is the
same for both claims. We only sketch it, as it resembles those used in
Section~\ref{sec:mainproof}.

Let $\zeta'$ be any fixed configuration that is non-zero
only at the origin and let $\zeta$ be any configuration satisfying the
assumptions of the theorem. We claim that for every $\varepsilon >0$ there is a
coupling of $\bar \P_\zeta $ and $\bar \P_{\zeta'}$ such that, with
probability at least $1-\varepsilon $ in the coupling, the difference of the
maxima of the two processes remains eventually constant. Since the theorem
holds under $\bar \P_{\zeta'}$ by the above arguments, the existence of such
a coupling then implies that
\begin{equation*}
  \bar \P_\zeta\Big(\lim_{x\to\infty} \frac{\tau_\hit(x)}{x} = \gamma\Big)
  \ge 1-\varepsilon.
\end{equation*}
Since $\varepsilon$ is arbitrary, this implies the claim of the theorem.

To construct the coupling, we first run the two processes independently.
Using the same arguments as in Lemmas~\ref{lem:neverwalkalone}
and~\ref{lem:blocks_somewhere}, with probability at least $1-\varepsilon/2$,
both processes produce many well-started blocks eventually. After this, we
can use (a variant of) the coupling from Proposition~\ref{prop:couplefronts}
and Lemma~\ref{lem:shiftedconditioned} to run the two processes together. Note
that here we do not require the maxima to be eventually ordered as
previously, but to eventually have a constant difference. The only assumption
of the initial condition needed to prove this is to have sufficiently many
well-started blocks, so that the probability that the coupling succeeds is at
least $1-\varepsilon /2$. \hfill\qed

\section{Proof of Theorem~\ref{thm:real_shape}}
\label{sec:proof_shape}

Theorem~\ref{thm:shape1} implies that $\{ x : \eta_n(x) \neq 0\}$ is
contained in $(-\infty, v(1+\varepsilon) n]$ eventually almost surely. Since
we assume that the dynamics is invariant under reflection around the origin,
we analogously have
$\{ x : \eta_n(x) \neq 0\} \subset [-v(1+\varepsilon) n, \infty)$ for all $n$
large enough and part~(a) follows.

The broad idea for part~(b) is that at time $n-\sqrt{n}$, say, there are many
well-started blocks in $[-(1-\varepsilon/2)vn, (1-\varepsilon/2)vn]$,
e.g.~by repeated applications of Proposition~\ref{prop:hitandblock}, and the
cluster that grows from them over the remaining time $\sqrt{n}$ has only
holes of size $\leq C \log n$. Here is a more detailed argument:

Let
\begin{equation}
  \tau_{\mathrm{dom}}(x)
  \coloneq \inf\big\{ t \in \L_\ttime :
    \ind(\bbracket{x} \in W(\eta_u)) \ge \tilde{Z}^\nu_u(\bbracket{x})
  \text{ for all } \L_\ttime \ni u \ge t \big\}
\end{equation}
be the time (on the grid $\L_\ttime$) when the property that the block
centred at $\bbracket{x}$ is well-started begins to dominate the stationary
percolation cluster defined above \eqref{eq:p_nu}. We claim that for any
$\zeta \not\equiv 0$, we have
\begin{align}
  \label{eq:hitthendominate}
  \P_\zeta(\tau_\hit(x) < \infty,
    \tau_{\mathrm{dom}}(x) > \bbfloor{ \tau_\hit(x)}  + C' \sqrt{x})
  \le C e^{-c x^{1/4}}.
\end{align}
Indeed, by Proposition~\ref{prop:hitandblock} and
Remark~\ref{rem:hitandblock} (with $\rho = x^{1/4}$), we have
\begin{align*}
  & \P_\zeta\big(\tau_\hit(x) < \infty,
    \exists\, z \in \L_\sspace, \abs{z-x} \leq L_\bl x^{1/2} :
    \abs[\big]{W(\eta_{\bbfloor{ \tau_\hit(x)} }) \cap B^\bl_{x^{1/4}}(z)}
    \geq p_\nu V_{x^{1/4}}/2 \big)
  \\ & \geq 1 - \Cr{c:liverbig}(e^{-\Cr{c:liversmall} x^{1/2}}
    + e^{-\Cr{c:liversmall} x^{1/4} + 2\log x}).
\end{align*}
Hence, with high probability, there are at time
$\bbfloor{ \tau_\hit(x)} $ at least on the order of $x^{1/4}$ many
well-started blocks within distance $x^{1/2}$ from $x$. By the percolation
comparison from Assumption~\ref{ass:domination} and
Lemma~\ref{lem:percolation}, the probability that the cluster started from
these blocks survives, grows at some linear speed and couples to the
stationary cluster is bounded below by $1 - C e^{-c x^{1/4}}$. This proves
\eqref{eq:hitthendominate}.

Since $\bar{\P}_\zeta(\tau_\hit(x) < \infty) = 1$ for all $x \in \Z$ (and
  $\P_\zeta(\tau_\ext = \infty) > 0$), we obtain that $\bar{\P}_\zeta$-a.s.
\begin{equation}
  \tau_{\mathrm{dom}}(x) \le \bbfloor{ \tau_\hit(x)}
  + \bbceil{ C' \sqrt{x}} \quad
  \text{for all large enough } x
\end{equation}
(and by the assumed symmetry also for all $x \in \Z$ which are sufficiently
  negative). Furthermore, by Theorem~\ref{thm:shape1}, we have
$\bar{\P}_\zeta$-a.s.
\begin{equation}
  \tau_\hit(x) \leq \frac{1+\varepsilon/2}{v} \abs{x} \quad
  \text{for all $\abs{x}$ large enough}.
\end{equation}
Thus, we have $\bar{\P}_\zeta$-a.s.\ for all $t \in \L_\ttime$ large enough
(using also that $\tau_{\mathrm{dom}}(x) < \infty$ $\bar{\P}_\zeta$-a.s.~for
  all $x$)
\begin{equation}
  \label{eq:taudom<=t}
  \tau_{\mathrm{dom}}(x)
  \le t \text{ for all } x \in [- (1-\varepsilon)v t,(1-\varepsilon)v t].
\end{equation}

Furthermore,
\begin{equation}
  \begin{split}
    & \P\big(\exists \, x \in [- v t, vt] \, :
      \, \tilde{Z}^\nu_t(\bbracket{x} + i L_\bl) = 0
      \text{ for }i=0,1,\dots, C\log t\big)
    \\ & \leq 2vt\, \P\big(\tilde{Z}^\nu_0(i L_\bl) = 0
      \text{ for }i=0,1,\dots, C\log t\big)
    \leq 2 \tilde{C} v t\, e^{-\tilde{\gamma} C \log t/2}
  \end{split}
\end{equation}
by Lemma~\ref{lem:percolation}\eqref{perc:large_dev}, which is summable in
$t \in \L_\ttime$ for $C > 4/\tilde{\gamma}$.

Combining this with \eqref{eq:taudom<=t} shows (b) along
$n = t \in \L_\ttime$. Furthermore, since $\bbracket x \in W(\eta_n)$ with
$n \in \L_\ttime$, $n \geq T_\bl$ implies in particular that
\begin{equation*}
  \{ y : \eta_{n-m}(y) \neq 0 \} \cap B_{m R + K L_\bl}(x)
  \neq \emptyset \quad
  \text{for } m=0,1,2,\dots,T_\bl-1,
\end{equation*}
we see that Claim~(b) holds in fact for all $n\in \N$ large enough.
\hfill\qed

\section{Application to branching annihilating random walk}
\label{sec:barw}

In this section we apply the abstract theory developed in the first part of
the paper to a specific model, the branching annihilating random walk (BARW)
that we studied previously in \cite{BCC24}. We will show that this model
satisfies Assumptions~\ref{ass:flow}--\ref{ass:coupling}, meaning that the
population, conditioned on survival, spreads with linear speed.

We recall the definition of the model from \cite{BCC24}. The state space is
$E=\{0,1\}$, where $0$ represents an empty site, and $1$ a site occupied by a
particle. The model has two parameters, $R\in \mathbb N$ (the dispersal
  range) and $\mu> 0$ (the mean number of offspring before annihilation). To
define the dynamics, let
\begin{equation}
  \label{eqn:phi}
  \varphi_\mu(w) \coloneq \mu w e^{-\mu w}, \quad w \in [0,1],
\end{equation}
and for $\eta\in E^\Z$ and $x\in \Z$, let
\begin{equation}
  \label{eqn:density}
  \delta_R(x, \eta) \coloneq \frac{1}{V_R} \sum_{y \in B_R(x)} \eta(y)
\end{equation}
be the density of $\eta$ in $B_R(x)$. Given $\eta_n$, the configuration
$\eta_{n+1}$ is defined for every $x \in \Z$ by
\begin{equation}
  \label{eqn:pca}
  \eta_{n+1}(x)
  = \begin{cases}
    1 \qquad \text{with probability } \varphi_\mu( \delta_R(x, \eta_n))
    \\ 0 \qquad \text{otherwise,}
  \end{cases}
\end{equation}
with $( \eta_{n+1}(x):x\in \Z)$ being conditionally independent given $\eta_n$.
This can be realised by setting
\begin{equation}
  \label{eqn:barwass1}
  \eta_{n+1}(x) = \ind_{U(x,n+1) \le \varphi_\mu( \delta_R(x, \eta_n))},
\end{equation}
where $\{ U(y,j): y \in \Z, j \in \N\}$ are i.i.d.~uniform, as in
Assumption~\ref{ass:flow}. From now on, we use \eqref{eqn:barwass1} as the
definition of the process.

\label{page:nameBARWexplained}
To explain the name `branching annihilating random walk', note that the model
can alternatively be built by iteratively applying the following algorithm, as
explained in \cite{BCC24}:
\begin{enumerate}[1.]
  \item (branching) At every time step, all particles independently branch, each
  of them produces a Poisson$(\mu )$-distributed number of children.

  \item (random walk) For a parent particle located at $x$, its children
  are placed independently to a uniformly chosen site in
  $B_R(x)=[x-R, x+R]\cap \Z$.

  \item (annihilation)
  Whenever at least two offspring particles (non-necessary with the same parent)
  land on the same site, all particles on that site are killed immediately.
  Surviving particles make up the population in the next time step.
\end{enumerate}

This alternative description is a model for a population evolving in demes,
where annihilation represents a very strong form of local competition among
individuals. It obviously makes the model non-monotone. We refer
to~\cite{BCC24} for a broader discussion on the applicability of the model.
On the other hand~\eqref{eqn:pca} can also be viewed as a particular
probabilistic cellular automaton.

Note also that
\begin{equation}
  \label{eqn:celauto}
  \E[\delta_R(x,\eta_{n+1}) \mid \eta_n]
  = \frac{1}{V_R} \sum_{y \in B_R(x)} \varphi_\mu( \delta_R(y,\eta_n))
\end{equation}
and thus, when $R$ is large enough for the law of large
numbers to take effect,
\begin{equation}
  \label{eqn:celautoapr}
  \delta_R(x,\eta_{n+1}) \approx \varphi_\mu (\delta_R(x,\eta_n)).
\end{equation}

It follows naturally that the solutions to the fixpoint equation
$w=\varphi_\mu(w)$ behave as fixpoints for the evolution of the (stochastic)
density. It is easy to see that (cf.~\cite[Section~1.1]{BCC24}):

\begin{enumerate}[(a)]
  \item For every $\mu>1$, the fixpoint equation $\varphi_\mu(w)=w$ has two
  solutions, $\{0, \theta_\mu\}$, where $\theta_\mu \coloneq \mu^{-1} \log \mu$.

  \item  The fixpoint $0$ is repulsive for every $\mu >1$, whereas
  $\theta_\mu$ is attractive if and only if $\mu \in(1,e^2)$.
\end{enumerate}

In \cite{BCC24} we showed that if $\mu >1$ and $R$ is large enough (depending
  on $\mu $) the population survives locally with positive probability (that
  is, every $x \in \Z$ is occupied infinitely often). In addition, if
$\mu \in (1,e^2)$ and $R$ is large enough, the BARW conditioned on survival
converges in distribution to its unique non-trivial extremal invariant
distribution. Intuitively, this can be understood from the above properties
of the function $\varphi_\mu $ and \eqref{eqn:celautoapr}, since, on survival,
the density of the population is attracted to $\theta_\mu$.

The main result of this section is that in the same range of values of $\mu$
for which there is complete convergence (and large $R$), the population
spreads linearly.

\begin{theorem}
  \label{thm:BARW_speed}
  For every $\mu \in (1,e^2)$ there is $R_0=R_0(\mu)$ such that for all
  $R \geq R_0$ there is a coarse-graining so that
  Assumptions~\ref{ass:flow}--\ref{ass:coupling} hold. As a consequence, the
  BARW with parameters $(\mu, R)$ has a linear spreading speed in the sense
  of Theorems~\ref{thm:shape1} and~\ref{thm:real_shape}.
\end{theorem}

In order to show Theorem~\ref{thm:BARW_speed} we need to construct a suitable
coarse-graining, that is, we need to fix the constants
$L_\bl, T_\bl, K \in \mathbb N$, and the sets
$G^\eta \subset E^{B_{KL_\bl}(0)}$, $G^U \subset [0,1]^{\Block(0,0)}$ as
before Assumption~\ref{ass:coarse}, and then show that the five assumptions
hold for this choice. Let us now comment on the individual assumptions:

\begin{itemize}[leftmargin=*]
  \item Assumption~\ref{ass:flow} follows easily from \eqref{eqn:barwass1}:
  the function $F$ required by this assumption is given by
  \begin{equation}
    \label{eqn:F_BARW}
    F(\eta, U)(x)
    =\ind_{ \{ U(x) \le \varphi_\mu(\delta_R(x,\eta)) \}},
    \qquad x \in \Z, \ \eta \in E^\Z, \ U\in [0,1]^\Z.
  \end{equation}
  This function is clearly translation invariant and has all other required
  properties. In particular, $R$ in the assumption (the range of the
    interaction) agrees with the parameter $R$ of the model.

  \item The proof of Assumption~\ref{ass:coarse} is technical and requires to
  fix suitable sets $G^\eta $, $G^U$ of well-started configurations and good
  driving noise. The construction is presented in Section~\ref{sec:blocks}
  and relies on the coarse-graining introduced in~\cite{BCC24} in order to
  show complete convergence, and on its refinement
  from~\cite{Oswald2024ancestral}. Compared to those references,
  Assumption~\ref{ass:coarse} requires a stronger control on the evolution of
  the process (e.g., it requires to control $\eta $ throughout the block, not
    only at its bottom and top layer). We therefore need to revisit the
  previously introduced techniques.

  \item As explained in Remark~\ref{rem:LSS}, for
  Assumption~\ref{ass:domination} it is sufficient to show that the
  probability of $U\in G_\loc^U(0,0)$ can be made large enough so that
  \cite{liggett1997domination} applies. We prove this in Section~\ref{sec:LSS}.

  \item Assumptions~\ref{ass:irreducibility} and~\ref{ass:coupling}
  require certain events happening in a finite time with positive
  probability, uniformly over certain sets of initial conditions.
  Since the state space $E = \{ 0, 1\}$ is finite, the observation from
  Remark~\ref{rmk:Ass5_Efinite} applies. With it, the proofs
  of these assumptions are straightforward, and are given
  in Section~\ref{sec:BARWass4and5}.
\end{itemize}

\subsection{Comparison density profiles}
\label{sec:tools}

The coarse-graining construction required by Assumption~\ref{ass:coarse}
builds on a number of comparison arguments giving control over the evolution
of the density of the process. In order to develop them, we recall here
several definitions and results from~\cite{BCC24}.

We start by introducing the so-called \emph{comparison density profiles},
which are deterministic functions prescribing a certain behaviour of the
density of $\eta$. These will be central for the definition of suitable sets
$G^U$ and $G^\eta$ in Sections~\ref{sec:blocks} and~\ref{sec:LSS}.

\begin{definition}
  \label{def:density_profiles}
  For given $\varepsilon, \delta  >0$ and $k_0 \in \N$, \emph{comparison
    density profiles} are deterministic functions
  $\xi_k^{-}, \xi_k^{+} : \Z \to [0,\infty)$, $k=0,1,\dots,k_0$, satisfying:
  \begin{enumerate}[(i)]
    \item \label{profile:i}
    For every $k=0,\dots,k_0$ and all $x \in \Z$,
    $(1+\delta )\xi_k^{-}(x) \le (1-\delta ) \xi_k^{+}(x)$.

    \item \label{profile:ii}
    For every $k=0,\dots,k_0$,
    $\Supp(\xi_k^-) \coloneq \{x\in \mathbb Z: \xi_k^-(x) > 0\}$ is finite, and
    $\xi_k^-(x)\ge \varepsilon $ for
    every $x\in \Supp(\xi_k^-)$.

    \item \label{profile:iii}
    For every $k = 0,\dots,k_0-1$, and $x\in \Supp(\xi_{k+1}^-)$ it holds
    that if $\zeta : B_R(x) \to \mathbb R$ satisfies $\zeta (y) \in
    [\xi_k^{-}(y), \xi_k^{+}(y)]$ for all $y\in B_R(x)$, then
    \begin{equation}
      \label{eqn:xi.comp.bd1}
      (1+\delta) \xi_{k+1}^{-}(x)
      \le V_R^{-1} \sum_{y \in B_R(x)} \varphi_\mu ( \zeta(y))
      \le (1-\delta) \xi_{k+1}^{+}(x).
    \end{equation}
  \end{enumerate}
\end{definition}

Recalling \eqref{eqn:celauto}, note that the important
property~\ref{profile:iii} of the above definition implies that if the
density of $\eta_n$ is controlled by $\xi_{k}^{\pm}$, then the
\emph{expectation} of $\eta_{k+1}$ is controlled by
$(1\mp \delta) \xi_{k+1}^{\pm}$. The following lemma shows that, with high
probability when $R$ is large, $\xi_{k+1}^{\pm}$ also provide good bounds for
the density of the true dynamics.

\begin{lemma}[\cite{BCC24}, Lemma~2.3]
  \label{lem:propagation_xi}
  Let $\xi^{\pm}_{k}$ be \emph{comparison density profiles}. For
  $k \in \{ 0,\dots,k_0-1 \}$, $x\in \Supp(\xi_{k+1}^-)$, and
  $\zeta \in E^\Z$ define the event
  \begin{equation*}
    A_k(x,\zeta) \coloneq
    \Big\{ \delta_R(y,\zeta) \in \big[\xi_k^{-}(y), \xi_k^{+}(y)\big]
      \quad \text{for all } y \in B_R(x) \Big\}.
  \end{equation*}
  Then
  \begin{equation*}
    \P\Big( \xi_{k+1}^{-}(x) \le \delta_R(x, F(\zeta,U)) \le \xi_{k+1}^{+}(x)
      \Bigm| A_k(x,\zeta) \Big)
    \ge 1-C \exp(- c V_R),
  \end{equation*}
  where $c,C>0$ are constants dependent on $ \varepsilon, \delta$ and $\mu$
  but not $R$.
\end{lemma}

We now define precisely one version of comparison density profiles
$\xi_k^\pm$ that we later use for the definition of good blocks. The exact
definitions of these profiles are not particularly important in the present
paper, and are included mostly for completeness. The crucial properties we
need here are the extension of their support and their shift properties,
see~\eqref{eqn:supp_xik} and Figure~\ref{fig:ellmaps} below.

We focus on the regime $\mu \in (1,e^2)$ where the fixpoint $\theta_\mu$ is
attractive. In this regime we can fix $\varepsilon_\bl >0$ small so that
$\varphi_\mu$ is a contraction on
$[\theta_\mu-\varepsilon_\bl, \theta_\mu+\varepsilon_\bl]$, and define
\begin{equation}
  \label{eqn:kappa}
  \kappa\coloneq\kappa(\mu,\varepsilon_\bl)
  \coloneq \sup \{ \abs{\varphi_\mu'(u)}
    : u \in [\theta_\mu-\varepsilon_\bl, \theta_\mu+\varepsilon_\bl] \} < 1
\end{equation}
to be the contraction coefficient on this interval. Further, by Lemma~4.1 of
\cite{BCC24}, we can fix two sequences $(\alpha_m)_{m \in \N}$ and
$(\beta_m)_{m \in \N}$ such that $\alpha_1$ is sufficiently small (we fix it
  later), $\beta_1 = \max \varphi_\mu = 1/e <1$, $(\alpha_m)_{m \in \N}$ is
increasing, $(\beta_m)_{m \in \N}$ is decreasing, both converge to
$\theta_\mu$, and satisfy
\begin{equation}
  \label{eqn:phi_iter}
  \varphi_\mu([\alpha_m, \beta_m]) \subseteq (\alpha_{m+1}, \beta_{m+1})
  \quad \text{for all } m \in \N.
\end{equation}
We also set
\begin{equation*}
  m_\bl \coloneq
  \min\{m \in \mathbb N:\beta_{m} - \alpha_{m} < \varepsilon_\bl\}.
\end{equation*}
As in \cite[Lemma~2.5]{BCC24}, we define function
\begin{equation}
  \label{eqn:fdef}
  f(x)\coloneq
  \min \big\{ (\varepsilon_0+x/ \ceil{wR} ) \ind_{x \ge 0}, 1\big\},
  \qquad x\in\mathbb R,
\end{equation}
where $w, \varepsilon_0>0$ are suitable parameters (see Lemma~\ref{lem:xi_cdp}).
This function will be used in the definition of comparison density profiles at
low densities.
We set
\begin{equation}
  \label{eqn:R_dens_xi}
  R_\dens \coloneq \frac{16}{\abs{ \log \kappa} }  \ceil{ R \log R}
  \qquad\text{and}\qquad
  R_\dens(k)= R_\dens+ k \ceil{ s R},\  k \in \N,
\end{equation}
where $s\in (0,1)$ is another parameter ($s$ is a proxy for the `wave speed'
  of these profiles).

The comparison density profiles that we use are then given by
\begin{equation}
  \label{eqn:xiplus}
  \xi^+_k(x) \coloneq
  \begin{cases}
    \beta_{m_\bl}  & \text{if } \abs{ x } \le R_\dens(k) \\
    \beta_{m_\bl -j+1} & \text{if } R_\dens(k) + (j-1) R
    < \abs{ x } \le R_\dens(k) + j R,
    \,1 \le j \le m_\bl \\
    1& \text{if } \abs{ x } > R_\dens(k) + m_\bl R,
  \end{cases}
\end{equation}
and
\begin{equation}
  \label{eqn:ximinus}
  \xi^-_k(x)
  \coloneq \begin{cases}
    \alpha_{m_\bl}
    & \text{if } \abs{ x } \le R_\dens(k)
    \\ \alpha_{m_\bl -j+1}
    & \text{if } R_\dens(k) + (j-1) R < \abs{ x }
    \le R_\dens(k) + j R, \,1 \le j \le m_\bl
    \\ \alpha_1 f_k(x)
    & \text{if } \abs{ x } > R_\dens(k) + m_\bl R,
  \end{cases}
\end{equation}
where
\begin{equation*}
  f_k(x) \coloneq
  f\big( R_\dens(k) + m_\bl R +  \ceil{ wR}  - \abs{x} \big).
\end{equation*}
See Figure~\ref{fig:ellmaps} for an illustration. Note that
\begin{equation}
  \label{eqn:supp_xik}
  \Supp(\xi_k^-)
  = B_{ R_\dens + k\ceil{sR} + m_\bl R + \ceil{wR} }(0)
  = B_{ R_\dens(k) +m_\bl R+ \ceil{wR} }(0)
\end{equation}
and that $\xi_{k}^\pm$ is exactly $\xi_0^\pm$ shifted outwards by $k\ceil{ sR}$.

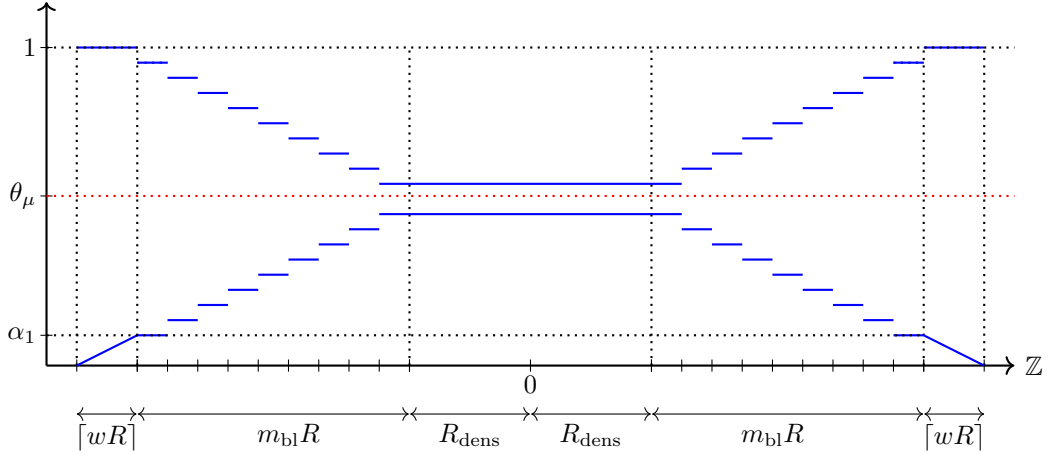
\begin{figure}[t]
  \centering
  \begin{tikzpicture}[scale = 0.8]
  % axis
  \draw[->, thick] (-8,0)--(8,0) node[right]{\small$\Z$};
  \draw[->, thick] (-8,0)--(-8,6);

  % labels on x axis
  \node[below] at (0,0){\small $0$};

  % bars on x axis
  \draw[] (0,0.1)--(0,-0.1);
  \draw[] (2,0.1)--(2,-0.1); % R_dens
  \draw[] (-2,0.1)--(-2,-0.1);
  \draw[] (6.5,0.1)--(6.5,-0.1);  % R_dens+mR
  \draw[] (-6.5,0.1)--(-6.5,-0.1);
  \draw[] (7.5,0.1)--(7.5,-0.1);  % R_dens+mR+wR
  \draw[] (-7.5,0.1)--(-7.5,-0.1);
  \foreach \i in {0,...,9} {		% R steps right
    \pgfmathsetmacro{\x}{2+0.5*\i}
    \draw (\x,0.1) -- (\x,-0.1);
  }
  \foreach \i in {0,...,9} {		% R steps left
    \pgfmathsetmacro{\x}{-2-0.5*\i}
    \draw (\x,0.1) -- (\x,-0.1);
  }

  % vertical lines
  \draw[dotted,thick, black](2,0)--(2,5.25); % R_dens
  \draw[dotted,thick, black](-2,0)--(-2,5.25);
  \draw[dotted,thick, black](6.5,0)--(6.5,5.25); % R_dens + mR
  \draw[dotted,thick, black](-6.5,0)--(-6.5,5.25);
  \draw[dotted,thick, black](7.5,0)--(7.5,5.25); % R_dens+mR+wR
  \draw[dotted,thick, black](-7.5,0)--(-7.5,5.25);

  % draw left alpha staircase
  % draw tilde xi
  \draw[blue, thick] (-7.5,0)--(-6.5,0.5);
  %\draw[dotted, thick, blue] (-6.5,0)--(-6.5,0.5);
  \foreach \i in {0,...,7} {
    % starting point of current step
    \pgfmathsetmacro{\x}{-6.5 + 0.5*\i}
    \pgfmathsetmacro{\y}{0.5 + 0.25*\i}
    % horizontal segment
    \draw[blue, thick] (\x,\y) -- ++(0.5,0);
    % vertical dotted segment
    % \draw[dotted, thick, blue] (\x+0.5,\y) -- ++(0,0.25);
  }

  % draw right alpha staircase
  % draw tilde xi
  \draw[blue, thick] (7.5,0)--(6.5,0.5);
  %\draw[dotted, thick, blue] (6.5,0)--(6.5,0.5);
  \foreach \i in {0,...,7} {
    % starting point of current step
    \pgfmathsetmacro{\x}{6 - 0.5*\i}
    \pgfmathsetmacro{\y}{0.5 + 0.25*\i}
    % horizontal segment
    \draw[blue, thick] (\x,\y) -- ++(0.5,0);
    % vertical dotted segment
    %\draw[dotted, thick, blue] (\x,\y) -- ++(0,0.25);
  }
  \draw[blue, thick] (-2.5,2.5)--(2.5,2.5);

  % draw left beta staircase
  % draw tilde xi
  \draw[blue, thick] (-7.5,5.25)--(-6.5,5.25);
  \draw[dotted, thick, blue] (-6.5,5)--(-6,5);
  \foreach \i in {0,...,7} {
    % starting point of current step
    \pgfmathsetmacro{\x}{-6.5 +0.5*\i}
    \pgfmathsetmacro{\y}{5-0.25*\i}
    % horizontal segment
    \draw[blue, thick] (\x,\y) -- ++(0.5,0);
    % vertical dotted segment
    %\draw[dotted, thick, blue] (\x+0.5,\y) -- ++(0,-0.25);
  }

  % draw right beta staircase
  % draw tilde xi
  \draw[blue, thick] (7.5,5.25)--(6.5,5.25);
  \draw[dotted, thick, blue] (6,5)--(6.5,5);
  \foreach \i in {0,...,7} {
    % starting point of current step
    \pgfmathsetmacro{\x}{6 - 0.5*\i}
    \pgfmathsetmacro{\y}{5 -0.25*\i}
    % horizontal segment
    \draw[blue, thick] (\x,\y) -- ++(0.5,0);
    % vertical dotted segment
    %\draw[dotted, thick, blue] (\x,\y) -- ++(0,-0.25);
  }
  \draw[blue, thick] (-2.5,3)--(2.5,3);

  % arrows on positive integers
  \draw[<->]  (0.01,-0.8)--(1.99,-0.8);
  \draw[<->]  (2.01,-0.8)--(6.49, -0.8);
  \draw[<->]  (6.51,-0.8)--(7.5, -0.8);
  \node[below] at (1,-0.8){\small $R_\dens$};
  \node[below] at (4,-0.8){\small $m_\bl R$};
  \node[below] at (7,-0.8){\small $\lceil wR \rceil$};

  % arrows on negative integers
  \draw[<->]  (-0.01,-0.8)--(-1.99,-0.8);
  \draw[<->]  (-2.01,-0.8)--(-6.49, -0.8);
  \draw[<->]  (-6.51,-0.8)--(-7.5, -0.8);
  \node[below] at (-1,-0.8){\small $R_\dens$};
  \node[below] at (-4,-0.8){\small $m_\bl R$};
  \node[below] at (-7,-0.8){\small $\lceil wR \rceil$};

  % labels on y axis
  \node[left] at (-8,5.25){\small$1$};
  \node[left] at (-8,2.8){\small$\theta_\mu$};
  \node[left] at (-8,0.5){\small$\alpha_1$};

  % bars on y axis
  \draw[] (-8.1,5.25)--(-7.9,5.25);
  \draw[] (-8.1,2.8)--(-7.9,2.8);
  \draw[] (-8.1,0.5)--(-7.9,0.5);

  % horizontal lines
  \draw[dotted,thick, red](-8,2.8)--(8,2.8); % theta_mu
  \draw[dotted,thick, black](-8,5.25)--(8,5.25); % 1
  \draw[dotted,thick, black](-8,0.5)--(6.5,0.5); % alpha_1
\end{tikzpicture}
  \caption{The density profiles $\xi_0^\pm$. For $k \ge 1$
    the profiles $\xi_{k}^\pm$ can be obtained by shifting
    $\xi_{0}^\pm$ outwards by $k\ceil{sR}$.}
  \label{fig:ellmaps}
\end{figure}

The following lemma states that the free parameters can indeed be chosen so
that these functions are density profiles.

\begin{lemma}[\cite{BCC24}, Lemma~4.2]
  \label{lem:xi_cdp}
  There exist $w \ge 2$, $\varepsilon_0 \in (0,1)$, $s \in (0,1)$,
  $\alpha_1 >0$ and $R_0 \in \N$ such that for every $R \ge R_0$ the maps
  $\xi_k^+, \xi_k^-$ defined in \eqref{eqn:xiplus}, \eqref{eqn:ximinus} are
  comparison density profiles in the sense of
  Definition~\ref{def:density_profiles}.
\end{lemma}

\subsection{Assumption~\ref{ass:coarse}: definition of local goodness}
\label{sec:blocks}

As we mentioned in the introduction to Section~\ref{sec:barw}, comparison
with percolation and the associated notion of good blocks and well-started
configurations for the BARW have been introduced in~\cite{BCC24} to prove
survival and complete convergence, and then slightly modified
in~\cite{Oswald2024ancestral}, where finer controls were required on the
density of $\eta$ for the study of ancestral lineages. Here we ask even
stricter requirements in the definition of good blocks compared to the ones
introduced in~\cite{Oswald2024ancestral}. We briefly outline the differences
to these versions:

\begin{itemize}[leftmargin=*]
\item  The definition of good blocks in \cite{BCC24}
  requires that that the `spreading of goodness' and the `production of local
  agreement' (in the sense of Assumption~\ref{ass:coarse}(b,c)) hold. However,
  the notion of good driving noise is  dependent on the \emph{specific}
  configuration at the bottom layer of the block.

  We adopt here an alternative definition (that already appeared
  in~\cite{Oswald2024ancestral}) in which the set of good driving
  noise $G^U$ is such that the same properties of how configurations
  propagate hold \emph{uniformly} over well-started
  configurations. We refer to Remark 4.5 in~\cite{BCC24} and Section 4
  in~\cite{Oswald2024ancestral} for further discussion.

 \item Neither of the block constructions in~\cite{BCC24}
 nor~\cite{Oswald2024ancestral} requires coupling preservation (in the sense
   of (d) in Assumption~\ref{ass:coarse}) \emph{for all} times in a block,
 but only that equal configurations are produced at its top time layer. In
 Section~\ref{sec:LSS} we prove that even with this additional requirement
 good blocks are typical.
\end{itemize}

We now proceed with the definition of the coarse-graining construction.
Recall the comparison density $\xi_k^\pm$ from \eqref{eqn:xiplus},
\eqref{eqn:ximinus}, and
$R_\dens$ from \eqref{eqn:R_dens_xi}. Let
\begin{equation}
  \label{eqn:L_BARW}
  L_\bl \coloneq R_\dens = \frac{16}{\abs{\log \kappa}} \ceil{ R \log R},
  \qquad
  T_\bl \coloneq \ceil[\bigg]{ \frac{3L_\bl}{\ceil{sR}}} =O(\log R),
\end{equation}
where $s \in (0,1)$ is the `speed of the comparison density profiles' fixed
by Lemma~\ref{lem:xi_cdp}. We further set $K=\ceil{10/s}$ (note that
  $K \ge 4$, as required in~\eqref{eqn:K}). Recall from
\eqref{eqn:L_bl}--\eqref{eqn:block} that $\L = L_\bl \Z \times T_\bl \N_0$
and, for $(z,t) \in \mathbb{L}$,
\begin{equation*}
  \Block(z,t) \coloneq
  \big\{(x,n) \in \Z \times \N_0 :
    x\in B_{K L_\bl}(z), \, t < n \le t+T_\bl \big\}.
\end{equation*}

We define the set of well-started configurations as the set of the
configurations whose density at the bottom of the block is controlled by
$\xi_0^\pm$:
\begin{equation}
  \label{eqn:Geta_BARW}
  G^\eta \coloneq \big\{ \eta \in E^{B_{K L_\bl}(0)} :
  \delta_R(x, \eta) \in[ \xi_0^-(x), \xi_0^+(x)]
  \ \forall x \in B_{KL_\bl-R}(0) \big\}.
\end{equation}
Note that the bounds on the density in the set $G^\eta$ are required in fact
only for points $x \in \Supp(\xi_0^-)$ (because $\xi_0^-=0$ and $\xi_0^+=1$
  outside of this set), and $\Supp(\xi_0^-) \subseteq B_{K L_\bl-R}(0)$ for
large enough $R$ (cf.~\eqref{eqn:supp_xik}, \eqref{eqn:L_BARW}).

Recall from \eqref{eqn:Geta} that we write $\eta_n \in G_\loc^\eta(x)$ if
$\eta_n(x+\cdot) \vert_{B_{K L_\bl}(0)} \in G^\eta$. Recall also from
\eqref{eqn:eta_U_flow} that for every $n \in \N$ we can
represent $\eta_n$ as
\begin{equation}
  \label{eqn:eta_U_flow_simplified}
  \eta_{n}=F_{0,n}(\eta_0, U),
\end{equation}
where the function $F_{0,n}(\zeta,\cdot):  [0,1]^{\Z \times \N} \to E^{\Z}$
is measurable with respect to the $\sigma $-algebra
$\mathcal{U}_{0,n}$ defined in \eqref{eqn:calU}. To
simplify the notation we write $F_n \coloneq F_{0,n}$ and
$\mathcal U_n \coloneq \mathcal U_{0,n}$. With this notation, we implicitly
define the set $G^U$ of good driving noises.

\begin{definition}
  \label{def:GU}
  Let $G^U \subseteq [0,1]^{\Block(0,0)}$ be the set of all $U$'s which satisfy
  \begin{enumerate}[(i)]
    \item For every $\zeta \in G_\loc^\eta(0)$,
      $n=1,\dots,T_\bl$, and $y \in \Supp(\xi_n^-)$
    \begin{equation}
      \label{eqn:gooddens}
      \delta_R( y ,F_n(\zeta,U))\in[\xi_n^-(y),\xi^+_n(y)].
    \end{equation}

    \item For every $\zeta, \zeta' \in G_\loc^\eta(0)$
    \begin{equation}
      \label{eqn:couple}
      F_{T_\bl}(\zeta,U)(x)
      =F_{T_\bl}(\zeta',U)(x)
      \quad \text{for every }
      x \in B_{2L_\bl}(z'),\ z' \in \{-L_\bl, 0, L_\bl\}.
    \end{equation}

    \item  If $\zeta(x)=\zeta'(x)$ for every $x \in B_{2L_\bl}(0)$, then
    \begin{equation}
      F_{n}(\zeta,U)(x)
      = F_{n}(\zeta',U)(x) \quad\text{for every } x \in B_{L_\bl/2}(0),
      \ n=1,\dots,T_\bl.
    \end{equation}
  \end{enumerate}
\end{definition}

Parts (ii), (iii) of this definition correspond to parts (c), (d) of
Assumption~\ref{ass:coarse}, while part (i) will imply parts (a) and (b). The
fact that the properties (i)--(iii) in the definition above involve only the
$U$'s in $\Block(0,0)$ is easy to check. We do this in the next lemma. Note
also that, by the shift properties of $\xi_k^{\pm}$ (see below
  \eqref{eqn:supp_xik}), \eqref{eqn:gooddens} is equivalent to
\begin{equation}
  \label{eqn:gooddens_alt}
  F_{n}(\zeta, U) \in G_\loc^\eta(x)
  \quad \text{for every } \zeta\in G_\loc(0),
  \ x \in B_{n \ceil{sR}}(0),
  \ n=1,\dots,T_\bl.
\end{equation}

\begin{lemma}
  \label{lem:p12}
  There exists $R_0 \in \N$ such that Assumption~\ref{ass:coarse} holds for
  every $R \ge R_0$ with $G^\eta$ defined in~\eqref{eqn:Geta_BARW} and $G^U$
  from Definition~\ref{def:GU}.
\end{lemma}

\begin{proof}
  Parts (c) and (d) of Assumptions~\ref{ass:coarse} correspond exactly to
  points (ii) and (iii) in Definition~\ref{def:GU}, so we should only verify
  parts (a) and (b).

  To see (a), observe that the definitions of $\xi^\pm$ and $G^\eta $ imply that
  if $\eta_n \in G^\eta$, then
  $\delta_R(y, \eta_n) \in [\theta_\mu-\varepsilon_\bl,
    \theta_\mu+\varepsilon_\bl]$ for all $y \in
  B_{R_\dens}(0)=B_{L_\bl}(0)$. In view of~\eqref{eqn:pca}, this
  implies~\eqref{eqn:locgood_to1} with
  \begin{equation*}
    \varepsilon_2
    \coloneq\min \{ \varphi_\mu(u):
      u \in [\theta_\mu-\varepsilon_\bl, \theta_\mu+\varepsilon_\bl]\} >0.
  \end{equation*}

  To show (b) we only have to check that for any $\zeta \in G_\loc^\eta(0)$,
  by time $T_\bl$ the local goodness of $\zeta$ spreads enough, that is
  $F_{T_\bl}(\zeta,U) \in G_\loc^\eta(z)$ for every
  $z\in \{-L_\bl, 0, L_\bl\}$. This is immediate because (i) is equivalent to
  \eqref{eqn:gooddens_alt},  which gives local
  goodness on $B_{T_\bl \ceil{sR}}(0)$. Since
  $B_{T_\bl \ceil{ sR }}(0) \supset B_{3L_\bl}(0)$, part (b) follows.

  It only remains to check that our definition of $G^U$ is local, i.e.,~that
  the properties (i)--(iii) in Definition~\ref{def:GU} are determined by the
  $U$'s in $\Block(0,0)$.  Recall from~\eqref{eqn:supp_xik} with $k=T_\bl$ that
  \begin{equation*}
    \Supp(\xi_{T_\bl}^-)
    = B_{L_\bl+ m_\bl R +\ceil{sR}T_\bl+ \ceil{wR}}(0),
  \end{equation*}
  Since (i) involves density computations, this means that it requires a
  control on $F_n(\zeta,U)(y)$ for $n = 1, \dots, T_\bl$ and
  $y \in B_{L_\bl+ (m_\bl+1) R +\ceil{sR}T_\bl+ \ceil{wR}}(0)$ (in fact, a
    smaller, but more complicated, subset of those values would suffice). The
  sets of values of $F_n(\zeta ,U)(y)$ that need to be controlled for (ii)
  and (iii) are smaller, since
  $B_{L_\bl/2}(0) \subset B_{3L_\bl}(0)\subset \Supp(\xi_0^-)$. Recalling the
  observation below \eqref{eqn:calU}, the values we need to control thus depend
  only on $U(x,n)$ for $(x,n)$ in the set
  \begin{equation*}
    B_{T_\bl R + R}(\Supp(\xi_{T_\bl}^-))
    \times \{1,\dots,T_\bl \}
    = B_{L_\bl+ (m_\bl+ T_\bl+1) R +\ceil{sR}T_\bl+ \ceil{wR}}(0)
    \times \{1,\dots,T_\bl \}
  \end{equation*}
  (again this set is larger than necessary).
  Next observe that $(m_\bl+1) R + \ceil{ wR }  \le L_\bl$
  for all $R$ large enough (depending on $\varepsilon_\bl$, $m_\bl$, $w$), and
  (using \eqref{eqn:L_BARW} in the middle inequality)
  \begin{equation*}
    \frac{T_\bl (\ceil{sR}+R)}{L_\bl}
    \le \frac{2T_\bl R}{L_\bl}
    \le 2  \Big( \frac{3 L_\bl}{sR} +1 \Big) \frac{R}{L_\bl}
    \le 2\Big(\frac{3}{s} + 1\Big).
  \end{equation*}
  Combining these two facts we deduce that
  \begin{equation*}
    B_{T_\bl R + R}(\Supp(\xi_{T_\bl}^-))
    %\subseteq B_{T_\bl R + \ceil{ sR } T_\bl + 2L_\bl}(0)
    \subseteq B_{(6/s+4)L_\bl}(0)
    \subseteq B_{(10/s) L_\bl}(0) \subseteq B_{KL_\bl}(0)
  \end{equation*}
  by our choice of $K$, which completes the proof.
\end{proof}

\subsection{Assumption~\ref{ass:domination}: good blocks are typical}
\label{sec:LSS}

With the definition of $G^U$ at hand, recall from
\eqref{eqn:Yzt}, that
\begin{equation*}
  Y(0,0) = \ind_{ U \in G_\loc^U(0,0)} = \ind_{U|_{\Block(0,0)} \in G^U},
\end{equation*}
where in the last equality we used \eqref{eqn:GU}. We now show that good
blocks (that is blocks with $Y(\cdot,\cdot) = 1$) are typical when $R$ is
large enough. As explained in Remark~\ref{rem:LSS}, this is sufficient to
prove Assumption~\ref{ass:domination}.

\begin{proposition}
\label{prop:goodblocks}
  Let $\mu \in (1,e^2)$. Then
  \begin{equation*}
    \lim_{R \to \infty} \P( Y(0,0) = 1)=1.
  \end{equation*}
\end{proposition}

We first introduce some notation and then explain the strategy of the proof
(see Figure~\ref{fig:blockb} for an illustration). For $n=1,\dots,T_\bl$, we
define the event
\begin{equation}
  \label{eqn:G}
  G_n= \big\{ F_{k}(\zeta,U) \in G_\loc^\eta(x)
    \text{ for all } x \in B_{k \ceil{ sR }}(0),
    \ k=1,\dots,n, \text{ and all }  \zeta \in G_\loc^\eta(0)\big\}.
\end{equation}
Note that, by \eqref{eqn:gooddens_alt}, $G_{T_\bl}$ corresponds to (i) in
Definition~\ref{def:GU} of $G^U$. We will later show that $G_{T_\bl}$ occurs
with high probability (see Lemma~\ref{lem:propagation_xi_uniform} below).
From the definition of $\xi_0^\pm$ it is easy to see that on $G_n$
\begin{equation}
  \label{eqn:closetotheta}
  \delta_R(x, F_{k}(\zeta,U))
  \in [\theta_\mu-\varepsilon, \theta_\mu+\varepsilon]
  \quad \forall x \in B_{L_\bl+k \ceil{ sR }}(0),
  \ k=1,\dots,n, \ \zeta \in G_\loc^\eta(0).
\end{equation}
Let
\begin{equation}
  \Theta_n \coloneq
  \big\{ (x,k) : x \in B_{L_\bl+k \ceil{ sR }}(0), \ k=1,\dots,n\big\}
\end{equation}
be the region of space-time points $(x,k)$ in $\Block(0,0)$ involved in
\eqref{eqn:closetotheta}. The set $\Theta_{T_\bl}$ is the region in between
the two blue dashed lines in Figure~\ref{fig:blockb}, with initial width
$2L_\bl$ at the bottom of the block and growing linearly with slope
$\ceil{ sR }$.

Next, for $m<n$, we set
\begin{equation}
  A_{m,n} =
  \{(x,k) : x \in B_{L_\bl/2 + k\ceil{ sR }  }(0), \, k =m+1,\dots,n \}.
\end{equation}
The set $A_{0,T_\bl}$ is the region in between the two red dashed lines in
Figure~\ref{fig:blockb}, with initial width $L_\bl$ at the bottom of the
block and growing linearly with slope $\ceil{ sR }$.

Finally, observe that if $\zeta(x)=\zeta'(x)$ for every
$x \in B_{2L_\bl}(0)$, then
\begin{equation}\label{eqn:green}
  F_{k}(\zeta,U)(x)
  =F_{k}(\zeta',U)(x) \quad \text{for all } x \in B_{2L_\bl-kR}(0)
  \text{ and }
  k=0,\dots, \floor{ 2L_\bl/R}.
\end{equation}
The region where this holds is the triangle formed by the two green
dashed lines with slope $R$ in Figure~\ref{fig:blockb}.

We denote by $h_0$ the first time at which the blue and green lines cross
(rounded up), and by $h_1$ the first time at which the red and green lines
cross (rounded down). Let $h\coloneq h_1-h_0$. It is not hard to see that $h$
is of order $\log R$ (see the proof of Proposition~\ref{prop:goodblocks} for
  more details). We write
\begin{equation*}
  C_h(x,n)\coloneq\{ (y,k) : y \in B_{(n-k)R}(x), \, k=n-h, \dots, n \}.
\end{equation*}
for the cone with vertex at $(x,n)$, slope $R$ and height $h$. By construction,
\begin{equation}
  \label{eqn:h0_cone}
  \text{if }(x,n) \in A_{h_1, T_\bl},
  \text{ then }C_h(x,n) \subseteq \Theta_{T_\bl}.
\end{equation}

\begin{figure}[t]
  \centering \begin{tikzpicture}[xscale=0.50,yscale=0.7]
  %good blocks
  \fill [solid,gray!30] (-10,0) rectangle (10,4);

  % A(h1,Tbl)
  \fill[red!20]
  (-7,4) -- (7,4) --(2.32,0.86) -- (-2.32,0.86)-- cycle;

  % grid and labels
  \draw[step=0.1cm,black!5](-11.5,-.3) grid (11.5,5.4);
  \draw[xstep=2cm,ystep=4] (-11.5,-.3) grid (11.5,5.4);

  \node[below] at (-10,-0.27){\small $-KL_\bl$};
  \draw[] (-10,0.1)--(-10,-0.1);

    \node[below] at (-6,-0.27){\small $-3L_\bl$};
  \draw[] (-6,0.1)--(-6,-0.1);

  \node[below] at (-2,-0.27){\small $-L_\bl$};
  \draw[] (-2,0.1)--(-2,-0.1);

  \node[below] at (0,-0.27){\small $0$};
  \draw[] (0,0.1)--(0,-0.1);

  \node[below] at (2,-0.27){\small $L_\bl$};
  \draw[] (2,0.1)--(2,-0.1);

    \node[below] at (6,-0.27){\small $-3L_\bl$};
  \draw[] (6,0.1)--(6,-0.1);

  \node[below] at (10,-0.27){\small $KL_\bl$};
  \draw[] (10,0.1)--(10,-0.1);

  \node[left] at (-11.8,0){\small $0$};
  \node[left] at (-11.8,4){\small $T_\bl$};

  %\draw[stealth-stealth] (-0.5,4.25)--(2.5,4.2);
  %\node[above] at (1,4.35){\small $L_\bl$};

  %\draw[stealth-stealth] (-6.5,-1.4)--(8.5,-1.4);
  %\node[below] at (1,-1.4){\small $c_\sspace L_\bl$};

  % meeting times
  %\draw[dashed,black] (2.7,0.93)--(-11.5,0.93) node[left,yshift=2pt] {\small $h_1$};
  %\draw[dashed,black] (3.5,0.61)--(-11.5,0.61) node[left] {\small $h_0$};
  \draw[dashed,black] (11.5,0.86)--(-11.5,0.86) node[left,yshift=2pt] {\small $h_1$};
  \draw[dashed,black] (11.5,0.55)--(-11.5,0.55) node[left] {\small $h_0$};
  \draw[<->,solid,black] (11.7,0.55)--(11.7,0.86) node[midway,right] {\small $h$};

  % equilibrium region
  \draw[thick,dashed,blue] (-2,0)--(-8,4) (2,0)--(8,4);

  % coupled region
  \draw[thick,dashed,red] (-1,0)--(-7,4) (1,0)--(7,4);

  % secured coupling
  \draw[thick,dashed,Green] (-4,0)--(0,2) (4,0)--(0,2);

  %cone Cxn
  \draw[black,thin] (-4.05, 2.68) -- (-2.45, 2.68) -- (-3.25, 3) -- cycle;
  \node[above] at (-3.25, 3) {\small $C_h(x,n)$};
  \fill (-3.25, 3) circle (2pt);

  % density profiles
  % \draw[thick,dashed] (-2.5,0.2)--(2.5,0.2) (-2.5,0.4)--(2.5,0.4)
  %(-3.25,0)--(-2.5,0.2) (-3.25,0.6)--(-2.5,0.4)
  %(3.25,0)--(2.5,0.2) (3.25,0.6)--(2.5,0.4);

  %\draw[thick,dashed] (-8.75,4.2)--(8.75,4.2) (-8.75,4.4)--(8.75,4.4)
  %(-9.5,4)--(-8.75,4.2) (-9.5,4.6)--(-8.75,4.4)
  %(9.5,4)--(8.75,4.2) (9.5,4.6)--(8.75,4.4);

\end{tikzpicture}
  \caption{The dark grey region corresponds to $\Block(0,0)$. The blue dashed
    lines enclose $\Theta_{T_\bl}$, that is the region of space-time points
    where on $G_{T_\bl}$ the density is in
    $[\theta_\mu-\varepsilon_\bl, \theta_\mu+\varepsilon_\bl]$. The red
    dashed lines delimit the coupled region $A_{0,T_\bl}$. Both regions grow
    linearly with slope $\ceil{ sR }$. The green triangle (with slope $R$) is
    the region of points where agreement is guaranteed by the fact that
    $\zeta(x)=\zeta'(x)$ for every $x \in B_{2L_\bl}(0)$. For every point
    $(x,n) \in A_{h_1,T_\bl}$ (the red region), the cone $C_h(x,n)$ (black
      triangle) is contained in $\Theta_{T_\bl}$. Note that the picture is
    not drawn to scale: $L_\bl$ is of order $R\log R$ while $T_\bl$ is of
    order $\log R$.}
  \label{fig:blockb}
\end{figure}
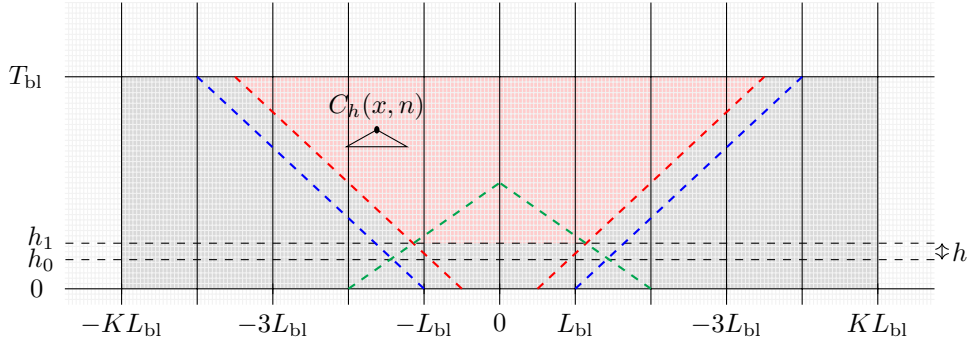

The proof of Proposition~\ref{prop:goodblocks} is organised into the
following steps: first, in Lemma~\ref{lem:disagreement}, we show that two
processes started from any initial configurations
$\zeta,\zeta' \in G_\loc^\eta(0)$, if $G_{n-1}$ occurs, agree at time $n$ on
every $x$ such that $(x,n) \in A_{h_1, T_\bl}$ with high probability. This
gives (ii) in the definition of $G^U$. Together with~\eqref{eqn:green}, this
implies that if $\zeta=\zeta'$ on $B_{2L_\bl}(0)$, the processes are equal on
$B_{L_\bl/2}(0)$ for all times until $T_\bl$, which corresponds to (iii) in
the definition of $G^U$.
The idea of the proof that agreement is produced at
$(x,n) \in A_{h_1, T_\bl}$ when starting from
$\zeta,\zeta' \in G_\loc^\eta(0)$, relies on the fact that on $G_{n-1}$ the
densities of both processes are close to $\theta_\mu$ for all points in
$C_h(x,n) \subseteq \Theta_{T_\bl}$ (cf.~\eqref{eqn:h0_cone}). Since $h$ can
be made large by increasing $R$, this essentially forces that two processes
agree at $(x,n)$.

Then, in Lemma~\ref{lem:propagation_xi_uniform}, we show that $G_{T_\bl}$
(and thus (i) in the definition of $G^U$, as explained above) occurs with
high probability if $R$ is sufficiently large.
Proposition~\ref{prop:goodblocks} is then proved easily, by combining the
results of Lemmas~\ref{lem:disagreement} and \ref{lem:propagation_xi_uniform}.

\begin{lemma}
  \label{lem:disagreement}
  With $\kappa$ as in \eqref{eqn:kappa}, it holds that
  \begin{equation*}
    \P \Big( \exists \zeta, \zeta' \in G_\loc^\eta(0),
      \exists (x,n) \in A_{h_1, T_\bl} :
      \big\{ F_{n}(\zeta,U)(x) \neq F_n(\zeta',U)(x)
      \} \cap G_{n-1} \Big)
    \le \abs{A_{h_1,T_\bl}} \kappa^h.
  \end{equation*}
\end{lemma}

\begin{proof}
  In this proof we write $G_\loc^\eta$ for $G_\loc^\eta(0)$. We start by
  showing that for every fixed $(x,n) \in A_{h_1,T_\bl}$ it holds that
  \begin{align}\label{eqn:disagreement}
    \E \Big[ \Big( \sup_{\zeta \in G_\loc^\eta} F_{n}(\zeta,U)(x)
        - \inf_{\zeta \in G_\loc^\eta} F_{n}(\zeta,U)(x) \Big)
      \ind_{G_{n-1}} \Big]
    \le \kappa^{h}.
  \end{align}
  To see this, take $(x,n) \in A_{h_1,T_\bl}$. Recall that
  $\mathcal{U}_{n}=\sigma(\{ U(y,j) : y \in \Z, 0 < j \le n\})$. Using that
  $F_{n-1}(\zeta,U)$ is $\mathcal{U}_{n-1}$--measurable, an easy computation
  gives
  \begin{align*}
    & \E \Big[ \sup_{\zeta \in G_\loc^\eta} F_n(\zeta,U)(x)
      - \inf_{\zeta \in G_\loc^\eta} F_n(\zeta,U)(x)
      \Bigm|  \mathcal{U}_{n-1}\Big]
    \\&= \E \Big[ \sup_{\zeta \in G_\loc^\eta}
      \ind_{\{U(x,n) \le \varphi_\mu(\delta_R(x, F_{n-1}(\zeta,U)))\}}
      - \inf_{\zeta \in G_\loc^\eta}\ind_{\{U(x,n)
          \le \varphi_\mu(\delta_R(x, F_{n-1}(\zeta, U))) \}}
      \Bigm| \mathcal{U}_{n-1}\Big]
    \\&= \E \Big[ \ind_{\{U(x,n) \le
          \sup_{\zeta \in G_\loc^\eta} \varphi_\mu(\delta_R(x, F_{n-1}(\zeta,U))) \}}
      - \ind_{\{U(x,n) \le
          \inf_{\zeta \in G_\loc^\eta}\varphi_\mu(\delta_R(x, F_{n-1}(\zeta,U))) \}}
      \Bigm| \mathcal{U}_{n-1}\Big]
    \\&= \P \Big( U(x,n) \le \sup_{\zeta \in G_\loc^\eta}
      \varphi_\mu(\delta_R(x, F_{n-1}(\zeta,U)))
      -\inf_{\zeta \in G_\loc^\eta}\varphi_\mu(\delta_R(x, F_{n-1}(\zeta,U)))
      \Bigm| \mathcal{U}_{n-1}\Big)
    \\&=\sup_{\zeta \in G_\loc^\eta} \varphi_\mu(\delta_R(x, F_{n-1}(\zeta,U)))
    -\inf_{\zeta \in G_\loc^\eta}\varphi_\mu(\delta_R(x, F_{n-1}(\zeta,U))).
  \end{align*}
  On the event $G_{n-1}$ (which is also measurable with respect to
    $\mathcal{U}_{n-1}$) we can use the contracting property of $\varphi_\mu$
  (see~\eqref{eqn:kappa}) since $\delta_R(x, F_{n-1}(\zeta,U)) \in
  [\theta_\mu-\varepsilon_\bl, \theta_\mu+\varepsilon_\bl]$ for every $\zeta
  \in G_\loc^\eta$. We obtain
  \begin{align*}
    &\E \Big[ \sup_{\zeta \in G_\loc^\eta} F_n(\zeta,U)(x)
      - \inf_{\zeta \in G_\loc^\eta} F_n(\zeta,U)(x)
      \Bigm| \mathcal{U}_{n-1}\Big] \ind_{G_{n-1}}
    \\& \le \kappa \Big(
      \sup_{\zeta \in G_\loc^\eta} \delta_R(x, F_{n-1}(\zeta,U))
      -\inf_{\zeta \in G_\loc^\eta} \delta_R(x, F_{n-1}(\zeta,U)) \Big)
    \ind_{G_{n-1}}
    \\&\le \kappa \cdot \frac{1}{V_R}
    \sum_{y \in B_R(x)}\Big( \sup_{\zeta \in G_\loc^\eta} F_{n-1}(\zeta,U)(y)
      - \inf_{\zeta \in G_\loc^\eta} F_{n-1}(\zeta,U)(y) \Big)  \ind_{G_{n-2}},
  \end{align*}
  where in the last inequality we have used that $G_{n-1} \subseteq G_{n-2}$.
  Taking expectations on both sides yields
  \begin{align*}
    & \E \Big[ \Big( \sup_{\zeta \in G_\loc^\eta} F_n(\zeta,U)(x)
        - \inf_{\zeta \in G_\loc^\eta} F_n(\zeta,U)(x) \Big)
      \ind_{G_{n-1}} \Big]
    \\&\le \frac{\kappa}{V_R} \sum_{y\in B_R(x)}
    \E \Big[ \Big( \sup_{\zeta \in G_\loc^\eta} F_{n-1}(\zeta,U)(y)
        - \inf_{\zeta \in G_\loc^\eta} F_{n-1}(\zeta,U)(y) \Big)
      \ind_{G_{n-2}} \Big].
  \end{align*}
  By~\eqref{eqn:h0_cone}, $C_h(x,n) \subseteq \Theta_{T_\bl}$ for every
  $(x,n) \in A_{h_1, T_\bl}$, so we can iterate the above $h$ times and
  obtain that
  \begin{align*}
    & \E \Big[ \Big( \sup_{\zeta \in G_\loc^\eta} F_n(\zeta,U)(x)
        - \inf_{\zeta \in G_\loc^\eta} F_n(\zeta,U)(x) \Big) \ind_{G_{n-1}}
      \Big]
    \\ & \le {\kappa^h} \bigg( \frac 1{V_R}\sum_{y_1 \in B_{R}(x)} \frac
      1{V_R} \sum_{y_2 \in B_{R}(y_1)}
      \cdots
      \\&\qquad \frac{1}{V_R} \sum_{y_h \in B_{R}(y_{h-1})}
      \E \Big[ \Big(\sup_{\zeta \in G_\loc^\eta} F_{n-h}(\zeta,U)(y_h)
          - \inf_{\zeta \in G_\loc^\eta} F_{n-h}(\zeta,U)(y_h) \Big)
        \ind_{G_{n-h-1}} \Big]\bigg).
  \end{align*}
  The right-hand side in the above is smaller than $\kappa^h$, because the
  expectation is smaller than 1. This completes the proof
  of~\eqref{eqn:disagreement}.

  Using the definition of conditional expectation and Markov's inequality
  (conditional to $\mathcal{U}_{n-1}$) we obtain
  \begin{equation}
    \label{eqn:onek}
    \begin{split}
      & \P \Big(  \Big\{ \sup_{\zeta \in G_\loc^\eta} F_n(\zeta,U)(x)
          - \inf_{\zeta \in G_\loc^\eta} F_n(\zeta,U)(x) =1 \Big\}
        \cap G_{n-1} \Big)
      \\&= \E \Big[ \P \Big( \sup_{\zeta \in G_\loc^\eta} F_n(\zeta,U)(x)
          - \inf_{\zeta \in G_\loc^\eta} F_n(\zeta,U)(x) =1
          \Bigm| \mathcal{U}_{n-1} \Big)
        \ind_{G_{n-1}} \Big]
      \\& \le \E \Big[ \E \Big[ \sup_{\zeta \in G_\loc^\eta} F_n(\zeta,U)(x)
          - \inf_{\zeta \in G_\loc^\eta} F_n(\zeta,U)(x)
          \Bigm| \mathcal{U}_{n-1} \Big] \ind_{G_{n-1}} \Big]
      \\&=\E \Big[ \Big( \sup_{\zeta \in G_\loc^\eta} F_n(\zeta,U)(x)
          - \inf_{\zeta \in G_\loc^\eta} F_n(\zeta,U)(x) \Big)
        \ind_{G_{n-1}} \Big]
      \le \kappa^{h},
    \end{split}
  \end{equation}
  where the last inequality follows by~\eqref{eqn:disagreement}. Since
  \begin{align*}
    &\bigcup_{\zeta,\zeta' \in G_\loc^\eta} \bigcup_{(x,n) \in A_{h_1,T_\bl}}
    \big\{ F_n(\zeta,U)(x) \neq F_n(\zeta',U)(x) \big\} \cap G_{n-1}
    \\ &\qquad  = \bigcup_{(x,n) \in A_{h_1,T_\bl}}
    \Big\{ \sup_{\zeta \in G_\loc^\eta} F_n(\zeta,U)(x)
      - \inf_{\zeta \in G_\loc^\eta} F_n(\zeta,U)(x) =1 \Big\}
    \cap G_{n-1}
  \end{align*}
  and the bound in~\eqref{eqn:onek} is uniform over
  $(x,n) \in A_{h_1,T_\bl}$, the result follows with a union bound.
\end{proof}

We next prove that the event $G_{T_\bl}$ (defined in~\eqref{eqn:G} and
  corresponding to (i) in the definition of $G^U$) holds with high
probability. This extends Lemma~\ref{lem:propagation_xi} so that it holds
uniformly over all starting configurations with density within $\xi_0^\pm$.

\begin{lemma}
  \label{lem:propagation_xi_uniform}
  Let $\xi^\pm$ be the \emph{comparison density profiles} from
  Section~\ref{sec:tools}. There exist constants $c,C>0$ such that
  \begin{equation*}
    \P \Big( \delta_R( \cdot , F_{n}(\zeta,U))  \in [\xi_n^-, \xi_n^+]
      \ \forall n=1,\dots,T_\bl, \, \forall \zeta \in G_\loc^\eta(0) \Big)
    \ge 1- C V_{T_\bl \ceil{ sR }} T_\bl e^{-c V_R}.
  \end{equation*}
\end{lemma}

\begin{proof}
  For every $y \in \Z$ define
  \begin{align*}
    \varphi^-(y)
    &\coloneq \min \{ \varphi_\mu(u) : u \in [\xi_0^-(y), \xi_0^+(y)]\}
    = \varphi_\mu \big(  \argmin_{[\xi_0^-(y), \xi_0^+(y)]} \varphi_\mu \big)
  \end{align*}
  and $\varphi^+$ analogously by replacing $\min$ and $\argmin$ with $\max$
  and $\argmax$. Let also
  \begin{equation*}
    \delta_R^\pm(x,U)
    = \frac{1}{V_R} \sum_{y \in B_R(x)} \ind_{U(y,1) \le \varphi^\pm(y)},
    \quad x \in \Z.
  \end{equation*}
  Take now any $\zeta \in G_\loc^\eta(0)$, such that
  $\delta_R(y,\zeta) \in [\xi_0^-(y), \xi_0^+(y)]$ for every $y \in \Z$ (note
    that in the definition of $G_\loc^\eta(0)$ this is required only for
    $y \in \Supp(\xi_0^-)$, but since $\xi_0^-=0$ and $\xi_0^+=1$ outside of
    this set, the density is automatically in $[\xi_0^-,\xi_0^+]$).
  Then
  \begin{equation*}
    \delta_R^-(\cdot,U) \le \delta_R(\cdot , F_1(\zeta,U))
    = \frac{1}{V_R} \sum_{y \in B_R (\cdot)} \ind_{U(y,1) \le
      \varphi(\delta_R(\cdot, \zeta))}
    \le \delta_R^
    +( \cdot ,U). \end{equation*}
 We deduce that
  \begin{align*}
    &\P\big( \exists \zeta \in G_\loc^\eta(0),
      \exists x \in \Supp(\xi_1^-) : \delta_R(x,F_1(\zeta,U))
      \notin [\xi_1^-(x), \xi_1^+(x)]\big)
    \\&\le \P( \exists x \in \Supp(\xi_1^-) : \delta_R^-(x,U) < \xi_1^-(x)
      \text{ or } \delta_R^+(x,U) > \xi_1^+(x)).
  \end{align*}
  Furthermore, by definition of $\argmin$, $\argmax$ and
  comparison density profiles
  (see~\ref{def:density_profiles}\ref{profile:iii}), it holds that
  \begin{equation*}\label{eqn:transfercdp}
    (1+\delta) \xi_{1}^-(x)
    \le \E[\delta_R^-(x,U)] \le \E[\delta_R^+(x,U) ]
    \le (1-\delta) \xi_{1}^+(x)
    \quad \text{for all } x \in \Supp(\xi_1^-).
  \end{equation*}
  This implies that
  \begin{equation*}
    \P\big( \delta_R^-(x,U) < \xi_1^-(x) \big)
    \le \P \big( \delta_R^-(x,U)
      -\E[\delta_R^-(x,U)]< -\delta \xi_1^-(x) \big),
  \end{equation*}
  which is smaller than $Ce^{-c V_R}$ by standard concentration inequalities
  for sums of centred, independent, and bounded random variables. A similar bound holds for
  $\P(\delta_R^+(x,U) > \xi_1^+)$, and the statement of the lemma follows
  from a union bound over $x$ and $n=1,\dots,T_\bl$.
\end{proof}

Proposition~\ref{prop:goodblocks} is an almost immediate consequence of
Lemmas~\ref{lem:disagreement} and \ref{lem:propagation_xi_uniform}. We
only need to give bounds on $h$ and check that the error term in
Lemma~\ref{lem:disagreement} is small.

\begin{proof}[Proof of Proposition~\ref{prop:goodblocks}]
  First observe that
  \begin{equation}
    \label{eqn:sufficient4Ygood}
    \{ F_{n}(\zeta,U)(x)=F_{n}(\zeta',U)(x) \ \forall (x,n) \in A_{h_1,T_\bl},
      \ \forall \zeta, \zeta' \in G_\loc^\eta(0) \}  \cap G_{T_\bl}
    \subseteq \{ Y(0,0)=1 \}.
  \end{equation}
  To see this, note that (i) in the definition of $G^U$ holds by definition
  of $G_{T_\bl}$. Furthermore,
  $L_\bl/2 + T_\bl \ceil{ sR } \ge L_\bl/2 + 3L_\bl \ge 3 L_\bl$, so
  $B_{3 L_\bl}(0) \times \{T_\bl\}\subseteq A_{0,T_\bl}$. This
  shows that the first part of (ii) holds. On the other hand, for all
  $\zeta,\zeta' \in G_\loc^\eta(0)$ such that $\zeta(x)=\zeta'(x)$ for every
  $x \in B_{2L_\bl}(0)$, by definition of $h_1$ it holds that
  \begin{equation*}
    F_{n}(\zeta,U)(x) = F_{n}(\zeta',U)(x)
    \quad \forall x \in B_{L_\bl/2}(0), \ n=1,\dots,h_1
  \end{equation*}
  (in fact, on the larger `green' triangle defined in~\eqref{eqn:green}, see
    Figure~\ref{fig:blockb}). Since
  $B_{L_\bl/2}(0) \times \{ h_1+1,\dots,T_\bl\} \subseteq A_{h_1, T_\bl}$, we
  conclude that on the event on the left-hand side
  of~\eqref{eqn:sufficient4Ygood} the second part of (ii) in the definition
  of $G^U$ holds.

  Now that we showed~\eqref{eqn:sufficient4Ygood}, we immediately deduce by
  Lemma~\ref{lem:propagation_xi_uniform} and Lemma~\ref{lem:disagreement} that
  \begin{equation*}
    \P(Y(0,0)=0)
    \le C V_{T_\bl \ceil{sR}} T_\bl e^{-c V_R} + \abs{A_{h_1,T_\bl}} \kappa^h.
  \end{equation*}
  It only remains to check that the second term tends to zero as
  $R \to \infty$. We first show that $h=h_1-h_0$ is of order $\log R$. Note that
  \begin{equation}
    \label{eqn:h0}
    h_1 \coloneq \floor[\Big]{
      \sup \{ k \in \N : L_\bl/2 + \ceil{sR} k \le 2L_\bl - Rk \}}
    = \floor[\bigg]{\frac{3L_\bl}{2(\ceil{sR}  +R)}}
  \end{equation}
  and
  \begin{equation*}
    h_0 \coloneq \ceil[\Big]{
      \sup \{ k \in \N : L_\bl + \ceil{sR}k \le 2L_\bl - Rk \}}
    = \ceil[\bigg]{\frac{L_\bl}{\ceil{sR}+R}}.
  \end{equation*}
  Recall that we define $L_\bl= 16 \ceil{R \log R} /\abs{\log \kappa}$ and
  $T_\bl = \ceil{ 3 L_\bl / \ceil{sR} }$. Furthermore, from
  Lemma~\ref{lem:xi_cdp} we have that $s \in (0,1)$ and so
  $\ceil{sR} + R \le 2R$. We easily obtain that
  \begin{align*}
    h =h_1 - h_0
    \ge \frac{3L_\bl}{2(\ceil{ sR } +R)}
    -\frac{L_\bl}{\ceil{ sR } +R} -2
    \ge \frac12 \cdot \frac{L_\bl}{2R} -2
    \ge \frac{4}{\abs{\log \kappa}} \log R-2.
  \end{align*}
  On the other hand $\abs{A_{h_1,T_\bl}}$ is essentially the area of the red trapezium in Figure~\ref{fig:blockb}, 
  so (using that $\ceil{ sR } T_\bl \le 3L_\bl$) we obtain
  \begin{align*}
    \abs{A_{h_1,T_\bl}}
    & \le 2 \cdot \frac12 ( L_\bl + \ceil{ sR } T_\bl )T_\bl
    \le 4 L_\bl T_\bl
    = O(R (\log R)^2).
  \end{align*}
  We have thus shown that $\abs{A_{h_1,T_\bl}} \kappa^h$ can be bounded with
  a term whose leading order is
  \begin{equation*}
    \exp \Big( \log R+ \frac{4}{\abs{\log \kappa}} \log R
      \cdot \log \kappa \Big).
  \end{equation*}
  which tends to zero as $R \to \infty$. This concludes the proof.
\end{proof}

\subsection{Assumptions \ref{ass:irreducibility}
  and \ref{ass:coupling}: local irreducibility}
\label{sec:BARWass4and5}

Assumptions \ref{ass:irreducibility} and \ref{ass:coupling} are easy
consequences of the finite range property of the model and the properties of
the Poisson distribution. The main observation is that there is
$\tilde\varepsilon >0$ such that for every fixed configuration
$\xi\in \{0,1\}^{B_{R}(0)}$ and for every initial condition $\zeta$
containing a particle at $0$
\begin{equation*}
  \P_\zeta \big( \eta_1(x) = \xi(x) \text{ for all } x\in B_R(0)\big)
  \ge \tilde\varepsilon.
\end{equation*}
This will be applied iteratively in the proof of these
assumptions.

Recall that Assumption~\ref{ass:irreducibility} requires that there are
$\varepsilon_{\ref{ass:irreducibility}} >0$ and
$t_{\ref{ass:irreducibility}} \in \L_\ttime$ such that for every
$\zeta \in E^\Z $  containing at least one particle in $ B_{L_\bl}(0)$,
\begin{equation}
  \mathbb P_\zeta\big(
    \eta_{t_{\ref{ass:irreducibility}}}\in G^\eta_\loc(0) \big)
  \ge \varepsilon_{\ref{ass:irreducibility}}.
\end{equation}
The event $\eta_{t_{\ref{ass:irreducibility}}} \in G^\eta_\loc(0)$ only
depends on the values of
$\{\eta_{t_{\ref{ass:irreducibility}}}(x):x\in B_{KL_\bl}\}$. Moreover, as
explained under \eqref{eqn:calU}, these values only depend on
$\{\eta_0(y):y \in B_{KL_\bl+R t_{\ref{ass:irreducibility}}}(0)\}$ (and
\(
  \{U(y,k): 1\le k \le t_{\ref{ass:irreducibility}},
    y \in B_{KL_\bl+R(t_{\ref{ass:irreducibility}}-k)} \}
\)).
Therefore, although the set of initial configurations considered in
Assumption~\ref{ass:irreducibility} is uncountable, effectively we only have
to deal with a finite set of initial configurations.

To prove Assumption~\ref{ass:irreducibility}, we fix an arbitrary
$\xi \in G^\eta_\loc(0)$ and
set $t_4 = \bbceil{2 KL_\bl/R}$, so that
the initial particle located in $B_{L_\bl}(0)$ can put an offspring anywhere
in $B_{KL_\bl}(0)$ in $t_4$ steps. For any initial condition
$\{\zeta (y): y\in B_{KL_\bl+R t_{\ref{ass:irreducibility}}}(0)\}$ we fix a
`path'
\(
  \big(\zeta_t|_{B_{KL_\bl+R t_{\ref{ass:irreducibility}}}(0)}:
    0\le t \le t_{\ref{ass:irreducibility}}\big)
\)
which is realisable for the model so that
$\zeta_0|_{B_{KL_\bl+Rt_{\ref{ass:irreducibility}}}} = \zeta$ and
$\zeta_{t_{\ref{ass:irreducibility}}}|_{B_{KL_\bl}}=\xi$ (e.g., for
  $1\le t < t_{\ref{ass:irreducibility}}$, we can choose $\zeta_t(x) =1 $ if
  $\zeta$ contains a particle in $B_{tR}(x)$). By the above observation, the
probability that the system follows exactly this path is at least
\begin{equation*}
  \tilde \varepsilon^{t_{\ref{ass:irreducibility}} \cdot \abs{B_{KL_\bl+
        t_{\ref{ass:irreducibility}}R}(0)}}\eqcolon
  \varepsilon_{\ref{ass:irreducibility}}>0
\end{equation*}
(in fact, the exponent is much larger than needed, but provides a lower bound).
This implies the claim of Assumption~\ref{ass:irreducibility}.

Assumption~\ref{ass:coupling} requires that there are
$\varepsilon_{\ref{ass:coupling}} >0$ and
$t_{\ref{ass:coupling}} \in \L_\ttime$ such that for every $\zeta^1$,
$\zeta^2$ such that $\zeta^i(z) =0$ for all $z >0$ and
$\{x:\zeta^i(x)\neq 0\}\cap B_{L_\bl}(0)\neq \emptyset$, $i=1,2$, there is a
coupling $\mathbb P_{\zeta^1,\zeta^2}$ of processes $\eta^1$,
$\eta^2$ started from $\zeta^1$ and $\zeta^2$ respectively, so that
\begin{equation*}
  \mathbb P_{\zeta^1,\zeta^2}\big(\eta^1_{t_5} \in G_\loc^\eta(0),
    \eta^1_{t_5}(z) = \eta^2_{t_5}(z) \ \forall z\ge - K
    L_\bl\big)\ge \varepsilon_5.
\end{equation*}

To prove this assumption, we can let $\eta^1$ and $\eta^2$ run independently
under $\P_{\zeta^1,\zeta^2}$. As in the proof of
Assumption~\ref{ass:irreducibility}, we fix
$t_{\ref{ass:coupling}} = t_{\ref{ass:irreducibility}}$, an arbitrary
$\xi \in G_\loc^\eta(0)$, and two paths $\zeta^1_t$, $\zeta^2_t$,
$0\le t \le t_{\ref{ass:coupling}}$, linking $\zeta^1$ or $\zeta^2$ to $\xi$,
respectively. By the same arguments as above, the claim of
Assumption~\ref{ass:coupling} then holds with
$\varepsilon_{\ref{ass:coupling}} = \varepsilon_{\ref{ass:irreducibility}}^2$.

\section{Discussion and outlook}
\label{sec:outlook}

We briefly discuss other models within the scope of our main results,
possible extensions, and future directions.

\subsection{Further examples}
\label{sec:othermodels}

For the following examples, Assumptions~\ref{ass:flow}--\ref{ass:coupling}
can be verified and thus analogues of Theorems~\ref{thm:shape1} and
\ref{thm:real_shape} hold for them.

\begin{example}[Other probabilistic cellular automata]
  \label{ex:generalphi}
  The specific form of the function $\varphi_\mu$ from \eqref{eqn:phi} in
  Section~\ref{sec:barw} is crucial for the interpretation of the process
  $\eta$ as a branching annihilating random walk (see the discussion on
    p.~\pageref{page:nameBARWexplained}).  On the other hand, inspection of
  the arguments from Section~\ref{sec:barw} shows that in the explicit
  definition of the dynamics in \eqref{eqn:barwass1}, the function
  $\varphi_\mu$ may be replaced by any continuously differentiable function
  $\widetilde{\varphi} : [0,1] \to [0,1]$ which is strictly positive on
  $(0,1]$, such that $\widetilde{\varphi}(0)=0$, $\widetilde{\varphi}'(0)>1$
  and there is a unique attracting fixpoint $\theta \in (0,1)$.
  Similarly, the requirement that the offspring particles are uniformly
  distributed over an interval of size $R$ around the parent particle is not
    essential. See also \cite[Remark~1.9]{BCC24}.
\end{example}
\smallskip

As briefly discussed in Section~\ref{sec:barw} (and in more detail in
\cite{BCC24}, see in particular Section~1.4 there), the branching
annihilating walk from Section~\ref{sec:barw} can be viewed as a
discrete stochastic approximation or `perturbation' of a
(deterministic) coupled map lattice
\begin{equation}
  \label{eq:CML}
  \Xi_{n+1}(x) = \varphi_\mu\big( \delta_R(x, \Xi_n) \big),
\end{equation}
compare this with \eqref{eqn:celauto} and \eqref{eqn:celautoapr}. One can
similarly consider variants where the deterministic iteration \eqref{eq:CML}
is subjected to a (possibly continuously distributed, but suitably small)
random perturbation and the local state space is thus $E=\R_+$. Non-spatial
versions of such systems are studied, e.g.,~in \cite{KlebanerLiptser1999},
\cite{KlebanerZeitouni1994}, \cite{KlebanerLazarZeitouni1998}.  They are of
the form
$X^\varepsilon_{n+1} = \varphi(X^\varepsilon_n) + \varepsilon
\eta^\varepsilon_n(X^\varepsilon_n,U_n)$ where the pertubation function can
be fairly general.

For definiteness, we formulate a specific example %of such a system
explicitly:

\begin{example}[A continuous stochastic perturbation of a coupled map lattice]
  \label{ex:continousnoise}
  Let $E=\R_+$,
  $\eta_n=(\eta_n(x) : x \in \Z)$ takes values in
  $E^\Z$. Its dynamics is given by
  \begin{equation}
    \label{eq:CML+normalperturb}
    \eta_{n+1}(x)
    = \max\Big\{ \varphi_\mu\big( \delta_R(x, \eta_n) \big)
      + \varepsilon g\big( \delta_R(x, \eta_n) \big) Z_{n+1}(x), 0 \Big\},
    \quad x \in \Z, n \in \N_0
  \end{equation}
  where $\varphi_\mu $, $\delta_R(x,\eta_n)$ are as in \eqref{eqn:phi},
  \eqref{eqn:density}, and the random variables
  $(Z_{n}(x): x \in \Z, n \in \N)$ are i.i.d.~standard normals. In
  \eqref{eq:CML+normalperturb}, we assume $\varepsilon>0$ and
  $g : [0,\infty) \to [0,\infty)$ is a bounded function with the properties
  \begin{equation}
    \label{eq:g.prop2}
    \begin{split}
      &g(0) = 0, \;\; g(w) > 0 \quad \text{for } w>0,\\
      & g \text{ is continuous on } (0,\infty)
      \text{ with $\lim_{w\downarrow 0} g(w)>0$}.
    \end{split}
  \end{equation}
  Thus, given $\eta_n$, $\eta_{n+1}(x)$ is a (truncated) normal random
  variable whose mean and variance depend on the local density in the
  previous generation.

  For $\mu \in (1,e^2)$, there exist $R_0 \in \N$ and $\varepsilon_0 >0$ such
  that for $R \geq R_0$ and $\varepsilon \in (0,\varepsilon_0]$, the system
  \eqref{eq:CML+normalperturb} satisfies
  Assumptions~\ref{ass:flow}--\ref{ass:coupling}. We sketch briefly below how
  the arguments from Section~\ref{sec:barw} can be adapted to show this. We
  remark that the requirement $g(0+)>0$ in \eqref{eq:g.prop2} is possibly a
  little artificial from the modelling point of view. This will be needed in
  particular in order to verify Assumption~\ref{ass:irreducibility}.
  Analogous to Example~\ref{ex:generalphi}, the function $\varphi_\mu$ in
  \eqref{eq:CML+normalperturb} can be replaced by a function
  $\widetilde{\varphi}$ with the properties stated there. \smallskip

  Note that \eqref{eq:CML+normalperturb} means that given $\eta_n$, the
  $\eta_{n+1}(x)$ are normal with mean
  $\varphi_\mu\big(\delta_R(x, \eta_n)\big)$ and (small) variance
  $\varepsilon^2 g^2\big( \delta_R(x, \eta_n) \big)$, truncated at $0$, and
  they are conditionally independent for different $x$'s. The mean of a
  truncated normal is of course not literally the same as that of its
  untruncated version. However, the difference is very small when the mean is
  much smaller than the standard deviation and we will ignore it for the
  purposes of the following sketchy arguments.

  \paragraph{Comparison density profiles.}
  The proof Lemma~\ref{lem:propagation_xi} (\cite[Lemma~2.3]{BCC24}) given in
  \cite{BCC24} relied on a Bernstein inequality, which leveraged the fact
  that in the setting of Section~\ref{sec:barw}, conditioned on $\eta_n$,
  $\delta_R(x, \eta_{n+1})$ is a scaled sum of independent Bernoullis. By
  \eqref{eq:CML+normalperturb}, in the setting of
  Example~\ref{ex:continousnoise}, $\delta_R(x, \eta_{n+1})$ given $\eta_n$
  is a scaled sum of (truncated) normals. Since truncated normals (also) have
  sub-gaussian tail bounds, the analogue of Lemma~\ref{lem:propagation_xi}
  holds when $R$ is large enough and $\varepsilon$ is small. The construction
  of comparison profiles can then essentially be copied from
  Sections~\ref{sec:tools} and~\ref{sec:blocks}.

  \paragraph{Couplings.}
  The obvious analogue of the flow construction \eqref{eqn:F_BARW} (or
    \eqref{eqn:barwass1}) for BARW is to simply use
  \eqref{eq:CML+normalperturb} (and to literally satisfy
    Assumption~\ref{ass:flow}, put
    $Z_{n+1}(x) \coloneqq \Phi^{-1}(U_{n+1}(x))$ where the $U_{n+1}(x)$ are
    i.i.d.\ uniform on $[0,1]$ and $\Phi^{-1}$ is the quantile function of
    the standard normal). This works well to leverage the contraction
  properties of $\varphi_\mu$ in combination with the comparison density
  profiles as in Section~\ref{sec:barw} and thus to bring the density of
  $\eta_n$ close to $\theta_\mu$ with high probability.  However, unlike the
  case of a discrete local state space (as in the BARW in
    Section~\ref{sec:barw}), in the continuous setting of
  Example~\ref{ex:continousnoise}, this coupling is unable to create literal
  agreement of configurations within a finite number of steps.

  In order to achieve the latter (as required in
    Assumptions~\ref{ass:coarse}(c),(d) and~\ref{ass:coupling}) we use the
  following observation: Denote by $f_{\mu,\sigma^2}$ the normal density with
  mean $\mu$ and variance $\sigma^2$. For (suitably tuned) small
  $\widetilde{\varepsilon}>0$ and $0 < \widetilde{\delta} \ll 1$, we have
  \begin{align}
    \widetilde{c} \coloneqq \int_{-\infty}^\infty
    \min_{w \in [\theta_\mu \pm \widetilde{\varepsilon}]}
    f_{\varphi_\mu(w),\varepsilon^2 g^2(w)}(x) \, dx
    \geq 1 - \widetilde{\delta}.
  \end{align}
  Note that if $Z_i$ is normal with mean $\varphi_\mu(w_i)$ and variance
  $\varepsilon^2 g^2(w_i)$, $i=1,2$, and
  $w_1, w_2 \in [\theta_\mu \pm \widetilde{\varepsilon}]$, then the total
  variation distance between laws of $Z_1$ and $Z_2$ is at most
  $1-\widetilde{c}$.

  Thus, we augment the flow construction \eqref{eq:CML+normalperturb} as
  follows:
  \begin{itemize}
    \item
    Independently for each space-time site $(x,n+1)$, let $B(x,n+1)$ be a
    $\mathrm{Bernoulli}(\widetilde{c})$ random variable. If $B(x,n+1)=1$,
    draw $\widetilde{Z}_{n+1}(x)$ with density
    \begin{equation*}
      \widetilde{c}^{-1} \min_{w \in [\theta_\mu \pm \widetilde{\varepsilon}]}
      f_{\varphi_\mu(w),\varepsilon^2 g^2(w)}(x).
    \end{equation*}
    In addition, let $Z_{n+1}(x)$ be a standard normal random variable,
    independent of everything else.

    \item
    Now for any $\eta_n$ put
    $D \coloneqq D(x, \eta_n) \coloneqq \delta_R(x, \eta_n)$.
    \begin{enumerate}[(a)]
      \item
      If $D \not\in [\theta_\mu \pm \widetilde{\varepsilon}]$,
      use \eqref{eq:CML+normalperturb} to define $\eta_{n+1}(x)$
      with $Z_{n+1}(x)$ as above.

      \item If $D \in [\theta_\mu \pm \widetilde{\varepsilon}]$ and
      $B(x,n+1)=1$, put
      $\eta_{n+1}(x) \coloneqq \max\{ \widetilde{Z}_{n+1}(x), 0 \}$.

      \item If $D \in [\theta_\mu \pm \widetilde{\varepsilon}]$ and
      $B(x,n+1)=0$, with
      \begin{equation*}
        h_D(x) \coloneqq f_{\varphi_\mu(D),\varepsilon^2 g^2(D)}(x)
        - \min_{w \in [\theta_\mu \pm \widetilde{\varepsilon}]}
        f_{\varphi_\mu(w),\varepsilon^2 g^2(w)}(x), \quad x \in \R,
      \end{equation*}
      draw $Z_D'$ from the density $h_D(x)/\int_\R h_D(y)dy$. (This can at
        least in principle be achieved by inserting a uniform $[0,1]$ random
        variable $U(x,n+1)$ into the corresponding quantile function.) Then
      put $\eta_{n+1}(x) \coloneqq \max\{ Z'_D, 0 \}$.
    \end{enumerate}
  \end{itemize}

  By construction, starting from any set of initial configurations with
  `reasonable' local densities, this coupling creates local agreement with
  high probability. It can be used to transfer Lemmas~\ref{lem:disagreement}
  and~\ref{lem:propagation_xi_uniform} to the setting of
  Example~\ref{ex:continousnoise}, which in turn yield
  Assumptions~\ref{ass:coarse} and~\ref{ass:domination}.

  \paragraph{Reaching good configurations with uniformly positive probability.}
  Assumption~\ref{ass:irreducibility} follows easily because by
  \eqref{eq:g.prop2}, when $\delta_R(x, \eta_n)$ is much smaller than
  $\varepsilon g(0+)$, the noise dominates the (possibly extremely small)
  term $\varphi_\mu\big( \delta_R(x, \eta_n) \big)$ in
  \eqref{eq:CML+normalperturb}. Thus, a minimal local density of the order
  $\varepsilon g(0+)$ can be reached within a few steps with uniformly
  positive probability from any non-zero initial condition.

  \paragraph{Creating agreement at the tip.}
  In order to verify Assumption~\ref{ass:coupling} note that we can use the
  coupling sketched above to create a large region of agreement within a
  finite number of steps with positive probability and then use the fact that
  by \eqref{eq:g.prop2} we have a uniform lower bound on the `noise strength'
  to enforce that at all the sites to the right of the agreement region, the
  density hits $0$ in the next step.

\end{example}

\subsection{Higher dimensions}
\label{sect:d>1}

Our results hold only in dimension $d=1$. It is natural to ask whether it is
possible to obtain similar results for $d\ge 2$.

Much of our method is, in fact, dimension-independent: the arguments of
Sections~\ref{sec:percol} and~\ref{sec:liverpool} carry over with essentially
no change to any dimension $d \ge 1$. The purely one-dimensional argument is
the coupling of the fronts that we construct in Section~\ref{sec:mainproof}.
It has no direct analogue in dimension $d\ge 2$, where the front of the
population has a considerably more complicated geometry, and comparing the
fronts of two copies of $\eta$ remains a challenging problem and a topic of
future research.

In the context of the BARW, Assumptions~\ref{ass:flow}--\ref{ass:irreducibility}
hold in any dimension $d \ge 1$ when $\mu \in (1,e^2)$ and $R$ is large
enough. Indeed, the definition of good blocks (and the comparison with
  percolation) was already introduced in~\cite{BCC24} in any dimension. In
fact, for Assumptions~\ref{ass:flow}, \ref{ass:coarse}(a,b),
\ref{ass:domination} and~\ref{ass:irreducibility} the restriction to
$\mu < e^2$ is not necessary: they hold for any $\mu > 1$, $d \ge 1$ and $R$
large enough. On the other hand, Assumption~\ref{ass:coupling} is specific to
dimension one and would need to be replaced by a suitable higher-dimensional
analogue for $d\ge 2$.

\subsection{Stochastic travelling waves, fluctuations}

Theorem~\ref{thm:shape1} shows that, starting from a single particle, the
right boundary 
\[
M_n = \max\{ x\in \Z : \eta_n(x) \neq 0 \}
\]
satisfies a law
of large numbers, $M_n/n \to v$ almost surely on the event of survival. This
raises several natural questions.

\begin{itemize}
  \item How does $M_n$ fluctuate around its mean? We conjecture that, for
  $d=1$, $(M_n - vn)/\sqrt{n}$ converges in distribution to a centred normal
  law with non-trivial variance.

  \item Does the interface near the tip stabilise? More precisely, does
  the law of $\eta_n(M_n + \cdot)$ converge as $n\to\infty$?
\end{itemize}

We expect that both questions can be approached by a regeneration argument
extending the construction in~\cite{Kuczek1989} for (monotone) oriented
percolation to our setting. Carrying this out is, however, non-trivial, and
we leave it to future work. Indeed, the `restart' arguments of
Section~\ref{sec:mainproof} can be used to define space-time contours that
separate the past from the future of $\eta$ seen from its right tip; but
while the times at which such contours arise have a regeneration flavour,
the resulting increments are not literally i.i.d., so that deducing a
central limit theorem would require establishing suitable mixing properties.

\bibliographystyle{jcamsalpha}
\bibliography{shapetheorem.bib}

\end{document}